\documentclass[11pt]{amsart}
\usepackage{amssymb, enumerate, mdwlist}

\evensidemargin\paperwidth
\advance\evensidemargin-\textwidth
\advance\oddsidemargin-1in
\evensidemargin\oddsidemargin

\topmargin\paperheight
\advance\topmargin-\textheight
\advance\topmargin-1in

\theoremstyle{plain}
\newtheorem{theorem}{Theorem}[section]
\theoremstyle{plain}
\newtheorem{proposition}[theorem]{Proposition}
\theoremstyle{plain}
\newtheorem{lemma}[theorem]{Lemma}
\theoremstyle{plain}
\newtheorem{corollary}[theorem]{Corollary}
\theoremstyle{plain}

\theoremstyle{plain}
\newtheorem{mthm}{Theorem}

\theoremstyle{plain}
\newtheorem{mcor}[mthm]{Corollary}
\theoremstyle{definition}
\newtheorem{definition}[theorem]{Definition}
\theoremstyle{remark}
\newtheorem{remark}[theorem]{Remark}
\theoremstyle{remark}
\newtheorem{example}[theorem]{Example}
\theoremstyle{remark}

\title[Group approximation in Cayley topology, II]
{Group approximation in Cayley topology and coarse geometry, \\ Part II: Fibred coarse embeddings.}
\author{Masato Mimura \and Hiroki Sako}
\thanks{Work was in part done while the first-named author was partially supported  by JSPS KAKENHI Grant Number JP25800033.}
\address{Masato Mimura\\
Mathematical Institute, Tohoku University, Japan\ /\ \'{E}cole Polytechnique F\'{e}d\'{e}rale de Lausanne, Switzerland}
\email{mimura-mas@m.tohoku.ac.jp}
\address{Hiroki Sako\\
School of Science, Department of Mathematical Sciences, Tokai University, Japan}
\email{hiroki.sako@gmail.com}
\date{\today}

\begin{document}

\begin{abstract}
The objective of this series is to study metric geometric properties of disjoint unions of Cayley graphs of amenable groups by  group properties of the Cayley accumulation points in the space of marked groups. In this Part II, we prove that a disjoint union admits a fibred coarse embedding into a Hilbert space (as a disjoint union) if and only if the Cayley boundary of the sequence in the space of marked groups is uniformly a-$\mathrm{T}$-menable. We furthermore extend this result to ones with other target spaces. By combining our main results with constructions of Osajda and Arzhantseva--Osajda, we construct two systems of markings of a certain sequence of finite groups with two opposite extreme behaviors of the resulting two disjoint unions: With respect to one marking, the space has property A. On the other hand, with respect to the other, the space does not admit fibred coarse embeddings into Banach spaces with non-trivial type (for instance, uniformly convex Banach spaces) or Hadamard manifolds; the Cayley limit group is, furthermore, non-exact.
\end{abstract}


\maketitle

\tableofcontents

\section{Introduction}
The main topics of this paper are the \textit{fibred coarse embeddings} of disjoint unions of Cayley graphs and \textit{equivariant coarse embeddings} of groups. Before proceeding to these two concepts, we first recall the definition of (genuine) \textit{coarse embeddings}. By  \textit{generalized metrics}, we mean metrics that possibly take the value $+\infty$. A basic example of generalized metric spaces is constructed as follows. For a sequence of metric spaces $(X_m,d_m)_{m\in \mathbb{N}}$, we define a generalized metric $d$ on $\bigsqcup_{m\in \mathbb{N}} X_m$ by $d(x,y)=d_m(x,y)$ if $x,y \in X_m$ for some $m$ and $d(x,y)=+\infty$ otherwise. We call the resulting generalized metric space $(\bigsqcup_{m\in \mathbb{N}} X_m,d)$ the \textit{disjoint union}, and simply write it as $\bigsqcup_{m\in \mathbb{N}}X_m$. 
\begin{definition}\label{definition=CoarseEmbeddings}
Let $(X, d_X)$ be a generalized metric space and $\mathcal{M}$ be a non-empty class of (genuine) metric spaces.
\begin{enumerate}[$(1)$]
\item Let $(M,d_M)$ be a (generalized) metric space. A (possibly discontinuous and possibly non-injective) map $f \colon X \to M$ is said to be a \textit{coarse embedding}, if
there exist two non-decreasing functions $\rho, \omega \colon [0,\infty) \to [0,\infty)$ that are proper (namely, $\lim_{r\to +\infty}\rho (r)=\lim_{r\to +\infty}\omega(r)=+\infty$) such that for all $x_1,x_2\in X$ \textit{such that} $d_X(x_1,x_2)<+\infty$, 
\[
\rho(d_X(x_1, x_2)) \leq d_M(f(x_1), f(x_2)) \leq \omega(d_X(x_1, x_2))
\]
holds true. The $\rho$, $\omega$, $(\rho,\omega)$ are, respectively, called a \textit{compression function}, an \textit{expansion function} and a \textit{control pair} for $f$. 
\item We say that $X$ \textit{admits a coarse embedding into $\mathcal{M}$} if there exist $M\in \mathcal{M}$ and a coarse embedding $f\colon X\to M$.
\item We say the pair $(\rho,\omega)$ of two non-decreasing proper functions $[0,\infty)\to [0,\infty)$ is a \textit{control pair} for $X$ into $\mathcal{M}$ if there exist $M\in \mathcal{M}$ and a coarse embedding $f\colon X\to M$ such that $(\rho,\omega)$ is a control pair for $f$. Denote by $\mathcal{CP}_{\mathcal{M}}(X)$ the set of all control pairs for $X$ into $\mathcal{M}$.
 \item If $X=\mathbf{G}$ is a marked group (with the metric $d_{\mathbf{G}}$; see Subsection~\ref{subsection=CayleyTopology}), we write $\mathcal{CP}_{\mathcal{M}}(X)$ as $\mathcal{CP}_{\mathcal{M}}^{\ast}(\mathbf{G})$ in order to distinguish it from the set $\mathcal{CP}_{\mathcal{M}}^{\sharp}(\mathbf{G})$ of \textit{equivariant} control pairs; compare with Definition~\ref{definition=a-M-menable}.
\end{enumerate}
\end{definition}
We make a remark that our convention on coarse embeddability of generalized metric spaces, as in $(i)$ above, is slightly non-standard. More precisely, we \textit{impose no condition on any pair of points with infinite distance to formulate coarse embeddabilty}. This is because our model example of generalized metric spaces is the disjoint unions of an infinite family of connected graphs; in that case, it is natural to put \textit{no} conditions on pairs of two vertices in distinct components.

The notion of \textit{fibred coarse embeddings} was introduced by Chen--Wang--Yu \cite{ChenWangYu}. This is a weakening of the (genuine) coarse embeddability; see Remark~\ref{remark=FromFibredToGenuine}. In this paper, since we consider the disjoint union of possibly infinite graphs, we relax the condition on exceptional sets, and call the modified notion that of \textit{fibred coarse embeddings as  a disjoint union}; see Definition~\ref{definition=GeneralizedFibredCoarseEmbeddings}. This new notion coincides with the original notion of \cite{ChenWangYu} for a coarse disjoint union of finite graphs; see Remark~\ref{remark=OriginalDefinition}. In \cite{ChenWangYu}, they  proved that if a coarse disjoint union $X$ of finite graphs of uniformly bounded degree admits a fibred coarse embedding into a Hilbert space, then the maximal Baum--Connes conjecture holds for $X$. Furthermore, Chen--Wang--Wang \cite{ChenWangWangNPC} proved that if $X$ above admits a fibred coarse embedding into a complete, connected and simply connected Riemannian manifold with non-positive sectional curvature (it is called an \textit{Hadamard manifold}), then the coarse Novikov conjecture holds for $X$. M. Finn-Sell \cite{FinnSell} studied a coarse disjoint union of finite connected graphs with uniformly bounded degree, in relation with the associated boundary groupoid,  that admits a fibred coarse embedding into a Hilbert space; he deduced the coarse strong Novikov conjecture for such a metric space.

The concept of \textit{equivariant coarse embedding} is defined for finitely generated groups in terms of isometric actions. It relates to Gromov's \textit{a-$\mathrm{T}$-menability} if the target space is a Hilbert space, and to \textit{a-$\mathcal{M}$-menability} in general cases; see Definition~\ref{definition=a-M-menable}.

We employ \textit{the space of $($$k$-$)$marked groups} $\mathcal{G}(k)$ to study a relationship between these two notions. This space was intensively studied by  R. I. Grigorchuk \cite[Section~6]{Grigorchuk}, and it is the space of (equivalence classes of) all pairs of a group and a $k$-generating ordered set. The space $\mathcal{G}(k)$ is equipped with the topology of \textit{local convergence as rooted diagrams}. This topology is sometimes called the \textit{Cayley topology}, and it is \textit{compact} and metrizable. We will briefly recall $\mathcal{G}(k)$ in Subsection~\ref{subsection=CayleyTopology}. For a sequence $(\mathbf{G}_m)_{m\in \mathbb{N}}$, we consider the following two objects:
\begin{itemize}
  \item The disjoint union $\bigsqcup_{m\in \mathbb{N}}\mathrm{Cay}(\mathbf{G}_m)$ of Cayley graphs, which is a generalized metric space \textit{without group structure}.
  \item The \textit{Cayley boundary} $\partial_{\mathrm{Cay}}(\mathbf{G}_m)_{m\in \mathbb{N}}(\subseteq \mathcal{G}(k))$, defined as follows.
\begin{definition}\label{definition=CayleyBoundary}
The \textit{Cayley boundary} $\partial_{\mathrm{Cay}}(\mathbf{G}_m)_{m\in \mathbb{N}}$  is defined as the set of all accumulation points of $(\mathbf{G}_m)_{m\in \mathbb{N}}$ in $\mathcal{G}(k)$ in the Cayley topology.
\end{definition}
It forms a non-empty compact set, consisting of marked \textit{groups} $\mathbf{G}_{\infty}\in \partial_{\mathrm{Cay}}(\mathbf{G}_m)_{m\in \mathbb{N}}$. 
\end{itemize}

\begin{definition}\label{definiiton=Uniformlya-M-menable}
Let $K$ be a non-empty subset of $\mathcal{G}(k)$ $(k\in \mathbb{N}_{\geq 1})$. For a non-empty class of metric spaces $\mathcal{M}$, we say that $K$ is \textit{uniformly a-$\mathcal{M}$-menable} if it admits equivariant \textit{equi}-coarse embeddings into $\mathcal{M}$. That means, there exists a common pair $(\rho,\omega)$ of non-decreasing proper functions $[0,\infty) \to [0,\infty)$ such that for every $\mathbf{G}\in K$, $(\rho,\omega)$ is an equivariant control pair from $\mathbf{G}$ into $\mathcal{M}$. In short, it holds that
\[
\bigcap_{\mathbf{G}\in K}\mathcal{CP}^{\sharp}_{\mathcal{M}}(\mathbf{G})\ne \emptyset;
\]
see Definition~\ref{definition=a-M-menable} for related definitions.
\end{definition}

Our main result, Theorem~\ref{mtheorem=MainTheorem}, requires several technical terminologies for the statement. In this introduction, instead of stating it, we exhibit a corollary to Theorem~\ref{mtheorem=MainTheorem}, Theorem~\ref{theorem=Corollary}. It, in particular, relates fibred coarse embeddability into a Hilbert space (as a disjoint union)  of the disjoint union of Cayley graphs of amenable marked groups to uniform a-$\mathrm{T}$-menability of the Cayley boundary. We refer the reader to Section~\ref{section=Precise} for the statement of Theorem~\ref{mtheorem=MainTheorem}.

\begin{theorem}[See Corollary~\ref{mcorollary=a-T-menable} for more detailed statements.]\label{theorem=Corollary}
Let $(\mathbf{G}_m)_{m\in \mathbb{N}}$ be a sequence of $\mathrm{amenable}$ marked groups  in $\mathcal{G}(k)$. 
The disjoint union $\bigsqcup_{m\in \mathbb{N}}\mathrm{Cay}(\mathbf{G}_m)$ admits a fibred coarse embedding into a Hilbert space as a disjoint union if and only if $\partial_{\mathrm{Cay}}(\mathbf{G}_m)_{m\in \mathbb{N}}$ is uniformly a-$\mathrm{T}$-menable. 

More generally, for fixed $q\in [1,\infty)$, $\bigsqcup_{m\in \mathbb{N}}\mathrm{Cay}(\mathbf{G}_m)$ admits a fibred coarse embedding into $L_q$, that means the Lebesgue $L_q$-space $L_q([0,1],\mathbb{R})$, if and only if $\partial_{\mathrm{Cay}}(\mathbf{G}_m)_{m\in \mathbb{N}}$ is uniformly a-$L_q$-menable.
\end{theorem}

Some work has been done by other researchers before our results in the context of \textit{box spaces} for a \textit{RF} (Residually Finite) group. If a finitely generated infinite group $G$ with a finite generating set $S$ admits a chain $(N_m)_{m\in \mathbb{N}}$, $N_{m+1}\leqslant N_m$, of normal subgroups of finite index in $G$ such that $\bigcap_{m\in \mathbb{N}}N_m=\{e_G\}$, then the \textit{box space} of $G$ is defined by
\[
\Box G=\Box_{(N_m)_m}G=\coprod_{m\in \mathbb{N}}\mathrm{Cay}(G/N_m;S\ \mathrm{mod} \ N_m),
\]
where $\coprod_m$ denotes a coarse disjoint union (see \cite[Definition~2.17.$(2)$]{MimuraSako} and Subsection~\ref{subsection=FibredCoarseEmbeddings}). Chen--Wang--Wang \cite{ChenWangWang} showed that  $\Box G$ admits a fibred coarse embedding into a Hilbert space if and only if $G$ is a-$\mathrm{T}$-menable. They also showed that for a metric space $M$, if $G$ is a-$M$-menable, then $\Box G$ admits a fibred coarse embedding into $M$. The present paper supplies several examples that admit fibred coarse embeddings into Hilbert spaces, but that do not admit genuine coarse embeddings; compare with Example~\ref{example=Selberg}.

Here we stress that the following points are visible \textit{only after} extending the framework from the class of box spaces to our general class; see the definitions of RF/LEF/LEA groups in Definition~\ref{definition=RFLEFLEA}.

\begin{enumerate}[$(a)$]
  \item The Cayley boundary $\partial_{\mathrm{Cay}} (\mathbf{G}_m)_m$ may consist of \textit{infinitely many} points.
  \item Even when $\partial_{\mathrm{Cay}} (\mathbf{G}_m)_m$ is a singleton $\{\mathbf{G}_{\infty}\}$, the Cayley limit group $\mathbf{G}_{\infty}=\lim_{m}\mathbf{G}_m$ is in the class of \textit{LEA} (Locally Embeddable into Amenable groups) group when $G_m$, $m\in \mathbb{N}$, is amenable; it is in the class of \textit{LEF} (Locally Embedabble into Finite groups) group when $G_m$, $m\in \mathbb{N}$, is furthermore finite. In general, the implications
\[
\mathrm{RF}\quad \Longrightarrow \quad \mathrm{LEF}\quad \Longrightarrow \quad \mathrm{LEA} 
\]
hold, and \textit{none} of the implications \textit{can be reversed}; see \cite[Subsection~2.2]{MimuraSako}.
  \item Coarse properties of $\bigsqcup_{m\in \mathbb{N}}\mathrm{Cay} ( G_m,S_m)$ may be \textit{considerably affected} by the choice of the system $(S_m)_m$ of generators of $G_m$, even when it might look a slight change.
\end{enumerate}

To illustrate point $(c)$ above, we study the following example. Here we set 
\[
\mathbb{N}_{\mathrm{odd}}=\{3,5,7,\ldots \}.
\]
(This set denotes the set of \textit{odd} integers \textit{at least} $3$; this is for a technical reason to avoid using $2m+1$ everywhere in the example below.)

\begin{example}\label{example=SpecialLinear}
Fix  a prime $p$. For $n\in \mathbb{N}_{\geq 1}$, denote by $\mathbb{F}_{p^n}$ the finite field of order $p^n$. It is well known that the multiplicative group $\mathbb{F}_{p^n}^{\times}$ is cyclic; for each $p$ and each $n$, we fix a generator $t_n=t_{p,n}\in \mathbb{F}_{p^n}$  of it. Fix a sequence $(n_m)_{m\in \mathbb{N}_{\mathrm{odd}}}$ of positive integers such that $\lim_{m\to \infty}n_m=+\infty$.

Let $G_m=\mathrm{SL}(m,\mathbb{F}_{p^{n_m}})$. Then for $m\in \mathbb{N}_{\mathrm{odd}}$, we consider the following two systems $(S_m)_{m\in \mathbb{N}_{\mathrm{odd}}}$, $(T_m)_{m\in \mathbb{N}_{\mathrm{odd}}}$ of generators of $G_m$.

\begin{itemize}
  \item For $m\in \mathbb{N}_{\mathrm{odd}}$, $S_m=(\sigma^{(m)},\upsilon^{(m)},\tau^{(m)})$. Here 
\[
\sigma^{(m)} = 
\left(
\begin{array}{ccccc}
1 & 1 & 0 & \cdots & 0 \\
0 & 1 & 0 &    \        & \vdots \\
\vdots & 0 & 1 & \ddots & \vdots \\
\vdots & \ & \ddots & \ddots & 0 \\
0 & \cdots & \cdots & 0 & 1 \\
\end{array}
\right),
\quad 
\upsilon^{(m)} = 
\left(
\begin{array}{ccccc}
1 & t_{n_m} & 0 & \cdots & 0 \\
0 & 1 & 0 &     \       & \vdots \\
\vdots & 0 & 1 & \ddots & \vdots \\
\vdots & \ & \ddots & \ddots & 0 \\
0 & \cdots & \cdots & 0 & 1 \\
\end{array}
\right),
\]
and 
\[
\tau^{(m)} = 
\left(
\begin{array}{ccccc}
0 & \cdots & \cdots & 0& 1 \\
1 & 0         &  \          &  \  & 0 \\
0 & 1         & 0         &  \  & \vdots \\
\vdots & \ddots & \ddots & \ddots & \vdots \\
0 & \cdots & 0 & 1 & 0 \\
\end{array}
\right).
\]
Define $X'=X'_{p,(n_m)}=\bigsqcup_{m\in \mathbb{N}_{\mathrm{odd}}}\mathrm{Cay}(G_m;S_m)$.
  \item For $m\in \mathbb{N}_{\mathrm{odd}}$, $T_m=(\sigma^{(m)},\sigma'^{(m)},\upsilon^{(m)}, \tau^{(m)})$. Here $\sigma^{(m)}$, $\upsilon^{(m)}$ and $\tau^{(m)}$ are the same as above, and $\sigma'^{(m)}={}^t \sigma^{(m)}$ is the transpose of $\sigma^{(m)}$.

Define $Y'=Y'_{p,(n_m)}=\bigsqcup_{m\in \mathbb{N}_{\mathrm{odd}}}\mathrm{Cay}(G_m;T_m)$.
\end{itemize}
\end{example}

For the proof of the fact that $S_m$ and $T_m$ are, respectively, markings of $G_m$, see \cite[Remark~5.5]{MimuraSako}.

For these $X'$ and $Y'$, we have the following contrast.

\begin{corollary}[See Corollary~$\ref{corollary=SpecialLinearGroups}$ for more detailed statements.]\label{corollary=SpecialLinear}
Let $X'$ and $Y'$ be as in Example~$\ref{example=SpecialLinear}$.
\begin{enumerate}[$(1)$]
  \item $($\cite[Remark~5.10]{MimuraSako}$)$ The space $X'$ enjoys property A of G. Yu \cite{Yu}. In particular, $X'$ admits a coarse embedding into every infinite dimensional Banach space; see the discussion below.
  \item The space $Y'$ does $\mathrm{not}$ admit a fibred coarse embedding as a disjoint union into $\mathcal{B}_{\mathrm{type}>1}$, the class of all Banach spaces with $($$\mathrm{linear}$, also known as $\mathrm{Rademacher}$$)$ $\mathrm{type}$ $>1$; see $(4)$ of Example~$\ref{example=BanachSpaces}$.
\end{enumerate}
\end{corollary}

For the first item, see also \cite[Corollary~B and  Proposition~2.22]{MimuraSako} in our Part I paper. In the current paper, we do not recall the definition of property A; see \cite{Yu} or \cite[Definition~2.19]{MimuraSako}. Yu \cite{Yu} showed that property A implies the coarse embeddability into a Hilbert space. By the Dvoretzky theorem \cite[Chapter~12]{BookBenyaminiLindenstrauss} and a theorem of M. I. Ostrovskii \cite{Ostrovskii}, it then follows that a locally finite generalized metric space with property A admits  a coarse embedding into every infinite dimensional Banach space. Thus we obtain the second assertion of $(1)$ above. At the other end of the spectrum, by $(2)$, the space $Y'$ above does \textit{not} admit a fibred coarse embedding as a disjoint union into a large class of Banach spaces, such as \textit{uniformly convex} Banach spaces (see $(7)$ of Example~\ref{example=BanachSpaces} for the definition). We refer the reader to \cite{BookTomczak-Jaegermann} and \cite{BookBenyaminiLindenstrauss}  as treatises on geometry of Banach spaces.

We investigate phenomena as in point $(c)$ to a greater extent by employing \textit{standard} (restricted) \textit{wreath products} $G\wr H$; see Subsection~\ref{subsection=WreathProducts} for the definition. By making use of the \textit{absorption trick}, observed by L.~Bartholdi and A.~Erschler \cite[Lemma~6.13]{BartholdiErschler} (we explain it in Subsection~\ref{subsection=AbsorptionLemma}), we obtain the following extreme example, which relates to \textit{non-exactness} of groups. See \cite{MimuraNonExact} and \cite{MimuraLW} for further developments in this direction.

\begin{theorem}[See Theorem~$\ref{mtheorem=NonExact}$ for the detailed statement.]\label{theorem=NonExact}
 There exist a sequence of finite groups $(\tilde{G}_n)_{n\in \mathbb{N}}$ with $\lim_{n\to \infty}\#(\tilde{G}_n)=\infty$ and $d\in \mathbb{N}$ such that the following holds true: There exist three systems $(S_n)_n$, $(T_n)_n$ and $(U_n)_n$ of $d$-markings of $\tilde{G}_n$ such that the following hold true:
\begin{enumerate}[$(1)$]
  \item The sequence $((\tilde{G}_n;S_n))_{n\in \mathbb{N}}$ converges in the Cayley topology to an $\mathrm{amenable}$ group. 
  \item The sequence $((\tilde{G}_n;T_n))_{n\in \mathbb{N}}$ converges in the Cayley topology to a group $\mathrm{without}$ property A. In other words, it is a $\mathrm{non}$-$\mathrm{exact}$ group. The Cayley limit group is, however, a-$\mathrm{T}$-menable. 
 \item The sequence $((\tilde{G}_n;U_n))_{n\in \mathbb{N}}$ converges in the Cayley topology to a $\mathrm{non}$-$\mathrm{exact}$ group; in addition, the Cayley limit group is $\mathrm{not}$ a-$\mathcal{M}$-menable for $\mathcal{M}=\mathcal{B}_{\mathrm{type}>1}$.
\end{enumerate}
\end{theorem}

In Theorem~\ref{theorem=NonExact}, we employ a constructions of D. Osajda \cite{OsajdaRF} of a RF non-exact group. More precisely, we use the LEF property for that non-exact group. This part of \cite{OsajdaRF} is built on earlier work of Osajda \cite{Osajda} and Arzhantseva--Osajda \cite{ArzhantsevaOsajda}; see the first part of Subsection~\ref{subsection=NonExact}. See also Remark~\ref{remark=Osajda} for item $(2)$ above. 

Three examples as in Theorem~\ref{theorem=NonExact} provide three disjoint unions
\[
\bigsqcup_{n}\mathrm{Cay}(\tilde{G}_n:S_n),\quad \bigsqcup_{n}\mathrm{Cay}(\tilde{G}_n:T_n)\quad \textrm{and}\quad \bigsqcup_{n}\mathrm{Cay}(\tilde{G}_n:U_n),
\]
whose coarse geometric properties are \textit{noteworthily different to each other}; see discussions below Theorem~\ref{mtheorem=NonExact}. It may indicate that, beyond the world of box spaces, it is \textit{no longer} reasonable to write disjoint unions as $\bigsqcup_{n}G_n$ without expressing markings. 

In the results above, we furthermore consider classes of \textit{non-linear} metric spaces, such as certain classes of $\mathrm{CAT}(0)$ spaces. See Section~\ref{section=Precise} for the precise statements in the full generality.

We, moreover, observe that \textit{point $(a)$ above is striking in the study of fibred coarse embeddings}: Unlike amenability and property $(\mathrm{T})$, \textit{uniformity is not automatic for a-$\mathcal{M}$-menability}; compare with \cite[Proposition~3.4]{MimuraSako} and \cite[Proposition~5.1]{MOSSPartIII}. Owing to this observation, we answer the question of Yu (in private communication) which asks whether the fibred coarse embeddability into a Hilbert space is closed under taking finite direct products. The answer is that it is \textit{almost never true} for (coarse) disjoint unions:

\begin{proposition}\label{proposition=Products}
Let $(\Gamma_m)_{m\in \mathbb{N}}$ and $(\Lambda_n)_{n\in \mathbb{N}}$ be two sequence of connected graphs of uniformly bounded degree. Let $X_1=\bigsqcup_{m\in \mathbb{N}}\Gamma_m$ and $X_2=\bigsqcup_{n\in \mathbb{N}}\Lambda_n$. Endow $X_1\times X_2$ with the structure of a disjoint union
\[
X_1\times X_2=\bigsqcup_{m,n\in \mathbb{N}}(\Gamma_m\times \Lambda_n),
\]
where $\Gamma_m\times \Lambda_n$ is equipped with the $\ell_1$-metric from $d_{\Gamma_m}$ and $d_{\Lambda_n}$. Let $\mathcal{M}$ be a non-empty class of metric spaces such that $\mathcal{UP}(\mathcal{M})\subseteq \mathcal{M}$; see Subsection~$\ref{subsection=OperationsOnMetricSpaces}$ for the symbol $\mathcal{UP}(\mathcal{M})$.

Then $X_1\times X_2$ admits a fibred coarse embedding as a disjoint union into $\mathcal{M}$ $\mathrm{only}$ $\mathrm{if}$ $X_1$ and $X_2$ both admit $($$\mathrm{genuine}$$)$ coarse embeddings into $\mathcal{M}$. In particular, this assertion applies to the case where $\mathcal{M}=\mathcal{H}\mathrm{ilbert}$, that means, the class of all Hilbert spaces.
\end{proposition}

If all $\Gamma_m$ and $\Lambda_n$ are finite, then we may replace disjoint unions above with coarse disjoint unions. In this case, the product above is equivalent to the product as metric spaces and fibred coarse embeddings may be taken in the original sense.

The argument for the proof of Proposition~\ref{proposition=Products} provides the following exotic example as well; see also Theorem~\ref{mtheorem=SymmetricGroups} for another example.

\begin{theorem}\label{theorem=Exotic}
There exists a sequence $(\Gamma_l)_{l\in \mathbb{N}}$ of finite graphs of uniformly bounded degree such that all of the following hold true.
\begin{enumerate}[$(1)$]
  \item The sequence $(\Gamma_l)_{l}$ forms an $\mathrm{expander}$ $\mathrm{family}$; see Definition~$\ref{definition=EmbeddedExpanders}$.
  \item The disjoint union $\bigsqcup_{l\in \mathbb{N}}\Gamma_l$ does not admit a fibred coarse embedding as a disjoint union into $\mathcal{CAT}(0)_{<1}$, that means, the class of complete $\mathrm{CAT}(0)$ space with $\mathrm{Izeki}$--$\mathrm{Nayatani}$ $\mathrm{invariant}$ strictly less than $1$; see Definition~$\ref{definition=IzekiNayataniInvariant}$.  Neither does it admit a fibred coarse embedding as a disjoint union to Banach spaces that are $\mathrm{sphere}$ $\mathrm{equivalent}$ $($see below$)$ to a Hilbert space.
  \item There exists a complete $\mathrm{CAT}(0)$ space $M$ such that $\bigsqcup_{l\in \mathbb{N}}\Gamma_l$ admits a $\mathrm{biLipschitz}$ embedding into $M$, namely, it admits a coarse embedding with control pair $(\rho,\omega)$, where $\rho$ and $\omega$ are both linear functions.
\end{enumerate}
\end{theorem}

Here two Banach spaces are said to be \textit{sphere equivalent} if there exists a bijection $\Phi$ between the unit spheres such that $\Phi$ and $\Phi^{-1}$ are both uniformly continuous; see \cite{MimuraSphereEquivalence} for details.  
Many reasonable $\mathrm{CAT}(0)$ spaces, including Hilbert spaces, all Hadamard manifolds and all Euclidean buildings associated with simple algebraic groups, belong to the class $\mathcal{CAT}(0)_{<1}$; see a paper of T. Toyoda \cite{ToyodaCAT(0)} for other examples of elements in $\mathcal{CAT}(0)_{<1}$.

In Section~\ref{section=Precise}, we present the precise statements of our main results. In the last part of Section~\ref{section=Precise}, we scratch the idea to prove our main result (Theorem~\ref{mtheorem=MainTheorem}); there we in addition explain relationships to relevant work by other researchers, and the organization of the current paper.

\ 

\paragraph{\bf Notation and Conventions:} We use $G$ for a (non-marked) group and $\mathbf{G}$ for a marked group. We write the group unit of $G$ as $e_G$. A finite generating set $S$ of $G$ is regarded as an ordered set (sometimes an ordered multi-set) $S=(s_1,s_2,\ldots ,s_k)$ so that $(G;S)$ is seen as a marked group. A marked group $\mathbf{G}=(G;S)$ is said to be finite (respectively, amenable, and  a-$\mathrm{T}$-menable) if so is $G$. For $k\in \mathbb{N}_{\geq 1}$, we denote by $\mathbf{F_k}$ the \textit{free $k$-marked group}, namely, $\mathbf{F_k}=(F_k;a_1,\ldots ,a_k)$. Here $(a_1,\ldots ,a_k)$ denotes a free basis of $F_k$. For $R\in \mathbb{R}_{\geq 0}$, let $\lfloor R\rfloor$ denote the integer part of $R$. For $m\in \mathbb{N}_{\geq 1}$, let $[m]=\{1,2,\ldots,m\}$. We use the terminology \textit{isometries} for surjective ones; we use \textit{geodesics} for minimal ones, namely, a \textit{geodesic} from $y\in M$ to $z\in M$ is an isometric embedding $c\colon [0,d(y,z)]\to M$. We always exclude the empty-set from metric spaces. For a metric space $M$, we write the class $\{M\}$ consisting only of $M$ as $M$ for short. As mentioned in the introduction, we use the symbol
\[
\mathbb{N}_{\mathrm{odd}}=\{3,5,7,\ldots\}
\]
for the set of \textit{odd} integers \textit{at least $3$.} (This is a non-standard notation; nevertheless, we use it for simplicity of description.)

\section{Precise statements of main results and the organization of this paper}\label{section=Precise}

In this section, we collect our main results for the reader's convenience. Some of them require several terminologies for the statements. We suggest the reader first cast a brief glance at this section to obtain a feel for the main theorems in the present paper, and then proceed to  subsequent sections. He/she may look back on this section to recall the precise statement of some result when diving into the proof. In the last part of this section, we describe the organization of the present paper.

To state Theorem~\ref{mtheorem=MainTheorem}, we need to formulate several operations on classes of metric spaces:
\[
\mathcal{M}\quad \mapsto \quad \mathcal{UP}(\mathcal{M}),\ \mathcal{F}_q(\mathcal{M}),\ \mathcal{FS}_q(\mathcal{M})\quad\textrm{and} \quad \ell_q(\mathcal{M}),\quad \textrm{for }q\in[1,\infty).
\]
See Subsection~\ref{subsection=OperationsOnMetricSpaces}, and Subsections~\ref{subsection=ClosedUnderOperations} and \ref{subsection=ClosedUnderOperationsMetric}, respectively, for the definitions and examples. See also Definitions~\ref{definition=GeneralizedFibredCoarseEmbeddings} and \ref{definition=a-M-menable}, respectively, for $\mathcal{CP}^{\mathrm{fib}}$ and $\mathcal{CP}^{\sharp}$.

\begin{mthm}[Main Theorem]\label{mtheorem=MainTheorem}
Let $\mathcal{M}$ be a non-empty class of metric spaces. Let $(\mathbf{G}_m)_{m\in \mathbb{N}}$ be a sequence in $\mathcal{G}(k)$ $($$k\in \mathbb{N}_{\geq 1}$$)$.
\begin{enumerate}[$(i)$]
\item
\begin{enumerate}[$(1)$]
  \item Assume that all $\mathbf{G}_m$, $m\in \mathbb{N}$, are $\mathrm{finite}$. Then, for every $q\in [1,\infty)$, the following holds true: If $\bigsqcup_{m\in \mathbb{N}}\mathrm{Cay}(\mathbf{G}_m)$ admits a fibred coarse embedding into $\mathcal{M}$ as a disjoint union, then $\partial_{\mathrm{Cay}}(\mathbf{G}_m)_{m\in \mathbb{N}}$ is uniformly a-$\mathcal{F}_q(\mathcal{M})$-menable.

Moreover, it holds that for every $q\in [1,\infty)$,
\[
\mathcal{CP}_{\mathcal{M}}^{\mathrm{fib}}\left(\bigsqcup_{m}\mathrm{Cay}(\mathbf{G}_m)\right)\subseteq \bigcap_{\mathbf{G}_{\infty}\in \partial_{\mathrm{Cay}}(\mathbf{G}_m)_m}\mathcal{CP}^{\sharp}_{\mathcal{F}_q(\mathcal{M})}(\mathbf{G}_{\infty}).
\]
  \item Assume that all $\mathbf{G}_m$, $m\in \mathbb{N}$, are $\mathrm{amenable}$. Assume moreover that 
\begin{itemize}
  \item either there exists  $q\in (1,\infty)$  such that for every $M\in \mathcal{M}$, there exists an element $L$ in $\mathcal{F}_q(M)$ that is $\mathrm{uniquely}$ $\mathrm{geodesic}$, see Definition~$\ref{definition=GeodesicSpaces}$, or
  \item the class $\mathcal{M}$ consists only of Banach spaces $($with no restriction on $q\in [1,\infty)$$)$.
\end{itemize}
Then for every such $q$ in the first case $($respectively, for every $q\in [1,\infty)$ in the second case$)$ the following holds true: If $\bigsqcup_{m\in \mathbb{N}}\mathrm{Cay}(\mathbf{G}_m)$ admits a fibred coarse embedding into $\mathcal{M}$ as a disjoint union, then $\partial_{\mathrm{Cay}}(\mathbf{G}_m)_{m\in \mathbb{N}}$ is uniformly a-$\mathcal{FS}_q(\mathcal{M})$-menable. 

Moreover, it holds that for every such $q$ above,
\[
\mathcal{CP}_{\mathcal{M}}^{\mathrm{fib}}\left(\bigsqcup_{m}\mathrm{Cay}(\mathbf{G}_m)\right)\subseteq \bigcap_{\mathbf{G}_{\infty}\in \partial_{\mathrm{Cay}}(\mathbf{G}_m)_m}\mathcal{CP}^{\sharp}_{\mathcal{FS}_q(\mathcal{M})}(\mathbf{G}_{\infty}).
\]
\end{enumerate}
  \item For every $q\in [1,\infty)$, the following holds true: If the Cayley boundary $\partial_{\mathrm{Cay}}(\mathbf{G}_m)_{m\in \mathbb{N}}$ is uniformly a-$\mathcal{M}$-menable, then the disjoint union $\bigsqcup_{m\in \mathbb{N}}\mathrm{Cay}(\mathbf{G}_m)$ admits a fibred coarse embedding into $\ell_q(\mathcal{M})$ as a disjoint union.

Moreover, it holds that for every $q\in [1,\infty)$,
\[
\bigcap_{\mathbf{G}_{\infty}\in \partial_{\mathrm{Cay}}(\mathbf{G}_m)_m}\mathcal{CP}^{\sharp}_{\mathcal{M}}(\mathbf{G}_{\infty})\subseteq \mathcal{CP}_{\ell_q(\mathcal{M})}^{\mathrm{fib}}\left(\bigsqcup_{m}\mathrm{Cay}(\mathbf{G}_m)\right).\]

If $(\mathbf{G}_{m})_{m\in \mathbb{N}}$ is a convergent sequence, then we may replace $\ell_q(\mathcal{M})$ with the original class $\mathcal{M}$ in the assertions above; in that case, it holds that
\[
\mathcal{CP}^{\sharp}_{\mathcal{M}}(\mathbf{G}_{\infty})\subseteq \mathcal{CP}_{\mathcal{M}}^{\mathrm{fib}}\left(\bigsqcup_{m}\mathrm{Cay}(\mathbf{G}_m)\right),\]
where $\mathbf{G}_{\infty}$ is the Cayley limit group of $(\mathbf{G}_m)_m$.
\end{enumerate}
\end{mthm}

Theorem~\ref{mtheorem=MainTheorem}, in particular, applies to the case where $\mathcal{M}=\mathcal{H}\mathrm{ilbert}$, the class of all Hilbert spaces. Thus we obtain the former half of Theorem~\ref{theorem=Corollary} from Theorem~\ref{mtheorem=MainTheorem}. More generally, we have the following corollary. See Examples~\ref{example=BanachSpaces} and \ref{example=MetricSpaces} for main examples of the class $\mathcal{M}$ in the current paper.

\begin{mcor}\label{mcorollary=a-T-menable}
Let $(\mathbf{G}_m)_{m\in \mathbb{N}}$ be a sequence of $\mathrm{amenable}$ marked groups  in $\mathcal{G}(k)$. 
\begin{enumerate}[$(1)$]
  \item The disjoint union $\bigsqcup_{m\in \mathbb{N}}\mathrm{Cay}(\mathbf{G}_m)$ admits a fibred coarse embedding into a Hilbert space as a disjoint union if and only if $\partial_{\mathrm{Cay}}(\mathbf{G}_m)_{m\in \mathbb{N}}$ is uniformly a-$\mathrm{T}$-menable. 
  \item Let $\mathcal{M}$ be either of the following classes. Then, $\bigsqcup_{m\in \mathbb{N}}\mathrm{Cay}(\mathbf{G}_m)$ admits a fibred coarse embedding into $\mathcal{M}$ as a disjoint union if and only if $\partial_{\mathrm{Cay}}(\mathbf{G}_m)_{m\in \mathbb{N}}$ is uniformly a-$\mathcal{M}$-menable. 
\begin{enumerate}
 \item[$(2$-$1)$] For $q\in [1,\infty)$, $L_q$ denoting $L_q([0,1],\mathbb{R})$. \item[$(2$-$2)$] For $r\in (1,2]$ and for $C>0$, $\mathcal{B}_{r,C}^{\mathrm{type}}$ being the class defined in $(4)$ of Example~$\ref{example=BanachSpaces}$. 
 \item[$(2$-$3)$] For $\delta_0\in [0,1]$, $\mathcal{CAT}(0)_{\leq \delta_0}$ denoting the class of all complete $\mathrm{CAT}(0)$ spaces with Izeki--Nayatani invariant at most $\delta_0$; see Definition~$\ref{definition=IzekiNayataniInvariant}$. 
\end{enumerate}
\end{enumerate}
Furthermore, for $\mathcal{M}$ being the class $\mathcal{H}\mathrm{ilbert}$ or being as in $(2)$, we have that
\[
\mathcal{CP}_{\mathcal{M}}^{\mathrm{fib}}\left(\bigsqcup_{m}\mathrm{Cay}(\mathbf{G}_m)\right)=\bigcap_{\mathbf{G}_{\infty}\in \partial_{\mathrm{Cay}}(\mathbf{G}_m)_m}\mathcal{CP}^{\sharp}_{\mathcal{M}}(\mathbf{G}_{\infty}).
\]
\end{mcor}
Item $(1)$ in Corollary~\ref{mcorollary=a-T-menable} is essentially a special case of $(2$-$1)$ with $q=2$.

We provide a similar example to one in Example~\ref{example=SpecialLinear}.

\begin{example}\label{example=SpecialLinearZ}
Let $(l_m)_{m\in \mathbb{N}_{\mathrm{odd}}}$ be a sequence of integers at least $2$ such that $\lim_{m\to \infty}l_m=\infty$. For  $m\in \mathbb{N}_{\mathrm{odd}}$, set $H_m=\mathrm{SL}(m,\mathbb{Z}/l_m\mathbb{Z})$ and take two markings $P_m$, $Q_m$ as follows:
\begin{itemize}
  \item Set $P_m=(\sigma^{(m)},\tau^{(m)})$, where $\sigma^{(m)}$ and $\tau^{(m)}$ are the matrices with exactly the same entries of $0$ and $1$ as in, respectively, $\sigma^{(m)}$ and $\tau^{(m)}$ as in $(1)$ above. 

Define $V'=V'_{(l_m)}=\bigsqcup_{m\in \mathbb{N}_{\mathrm{odd}}}\mathrm{Cay}(H_m;P_m)$.
  \item Set $Q_m=(\sigma^{(m)},\sigma'^{(m)},\tau^{(m)})$, where $\sigma^{(m)}$  and $\tau^{(m)}$ are the same as $P_m$, and $\sigma'^{(m)}={}^t\sigma^{(m)}$.

Define $W'=W'_{(l_m)}=\bigsqcup_{m\in \mathbb{N}_{\mathrm{odd}}}\mathrm{Cay}(H_m;Q_m)$.
\end{itemize}
\end{example}
In a similar argument to one in \cite[Remark~5.5]{MimuraSako}, it follows that $P_m$ and $Q_m$ are both markings of $H_m$.

To state Corollary~\ref{corollary=SpecialLinearGroups}, for every prime $p$, set
\[
\delta(p)=1- \frac{1}{2\left(1-\frac{\sqrt{p}}{p+1}\right)} \quad (\in (0,\frac{1}{2}));
\]
see the discussion above Example~\ref{example=r-UniformlyConvex} and Remark~\ref{remark=IzekiNayatani} for the background of this constant. For $\delta_0\in (0,1]$, let $\mathcal{CAT}(0)_{<\delta_0}$ denote the class of all complete $\mathrm{CAT}(0)$ spaces whose Izeki--Nayatani invariants are strictly less than $\delta_0$.

\begin{corollary}\label{corollary=SpecialLinearGroups}
Let $X'$ and $Y'$, and $V'$ and $W'$ be, respectively as in Examples~$\ref{example=SpecialLinear}$ and $\ref{example=SpecialLinearZ}$.
\begin{enumerate}[$(1)$]
  \item The spaces $X'$ and $V'$ both have  property A.
  \item The spaces $Y',V',W'$ do not admit a fibred coarse embedding as disjoint unions into $(\prod_{<\aleph_0}\mathcal{QT})_{\ell_1}$ or into $(\prod_{<\aleph_0}\mathcal{M})_{\ell_2}$, where $\mathcal{QT}$ denotes the class of all quasi-trees and $\mathcal{M}$ is the class of all finite dimensional, complete, connected and simply connected Riemannian manifolds with strictly negative sectional curvature that are cocompact; see Remark~$\ref{remark=QuasiTree}$ and Definition~\ref{definition=FiniteProducts} for the definitions.
  \item The space $Y'=Y'_{p,(n_m)_m}$ does not admit a fibred coarse embedding as a disjoint union into $\mathcal{B}_{\mathrm{type}>1}$. Neither does $Y'$ admit a fibred coarse embedding as a disjoint union into $\mathcal{CAT}(0)_{<\delta(p)}$. 
  \item The space $W'$ does not admit a fibred coarse embedding as a disjoint union into $\mathcal{B}_{\beta<1/2}$; see $(5)$ of Example~$\ref{example=BanachSpaces}$.
\end{enumerate}
\end{corollary}

For every prime $p$, the class $\mathcal{CAT}(0)_{<\delta(p)}$ as in $(3)$ above  includes $\mathcal{CAT}(0)_{\leq 0}$; the subclass $\mathcal{CAT}(0)_{\leq 0}$ contains all $($possibly infinite dimensional$)$ complete, connected and simply connected Riemannian manifolds with non-positive sectional curvature. Hence,  such spaces. After work \cite{BestvinaBrombergFujiwara} of Bestvina--Bromberg--Fujiwara, study of actions on finite products of quasi-trees has been paid an intensive attention.

The precise from of Theorem~\ref{theorem=NonExact} is stated in the following manner. To deduce Theorem~\ref{theorem=NonExact} from Theorem~\ref{mtheorem=NonExact}, fix $p$ and $(l_n)$, and let $\tilde{G}_n=G_n\wr \mathrm{SL}(2n+3,\mathbb{F}_{p^{l_n}})$.

\begin{mthm}\label{mtheorem=NonExact}
There exist a sequence of finite groups $(G_n)_{n\in \mathbb{N}}$ and $d\in \mathbb{N}$ such that the following holds true: For every prime $p$ and for every sequence $(l_n)_{n\in \mathbb{N}}$ of integers at least $2$ such that $\lim_{n\to \infty}l_n=\infty$, there exist three systems $(S_n)_n$, $(T_n)_n$ and $(U_n)_n$ of $d$-markings 
\begin{eqnarray*}
 S_n&=&(s_1^{(n)},s_2^{(n)},\ldots ,s_d^{(n)}), \\
 T_n&=&(t_1^{(n)},t_2^{(n)},\ldots ,t_d^{(n)}), \\
 U_n&=&(u_1^{(n)},u_2^{(n)},\ldots ,u_d^{(n)}), 
\end{eqnarray*}
of $(H_{n,p}(=H_{n,p,(l_n)_n})=G_n\wr \mathrm{SL}(2n+3,\mathbb{F}_{p^{l_n}}))_{n\in \mathbb{N}}$, such that the following hold true:
\begin{enumerate}[$(1)$]
  \item For every $n\in \mathbb{N}$ and for every $i \in [d]$, there exist $h_i=h_{n,p,i}\in H_{n,p}$ and $k_i=k_{n,p,i}\in H_{n,p}$ such that 
\[
h_i^{-1}s_i^{(n)} h_i=t_i^{(n)} \quad \textrm{and}\quad k_i^{-1}s_i^{(n)} k_i=u_i^{(n)}.
\]
  \item The sequence $((H_{n,p};S_n))_{n\in \mathbb{N}}$ converges in the Cayley topology to an amenable group. 
  \item The sequence $((H_{n,p};T_n))_{n\in \mathbb{N}}$ converges in the Cayley topology to a $\mathrm{non}$-$\mathrm{exact}$ group, but the Cayley limit group is a-$\mathrm{T}$-menable. 
 \item The sequence $((H_{n,p};U_n))_{n\in \mathbb{N}}$ converges in the Cayley topology to a $\mathrm{non}$-$\mathrm{exact}$ group. Moreover, the Cayley limit group is not a-$\mathcal{M}$-menable for $\mathcal{M}=\mathcal{B}_{\mathrm{type}>1}$ or $\mathcal{M}=\mathcal{CAT}(0)_{< \delta(p)}$. Here $\delta(p)$ is as in Corollary~$\ref{corollary=SpecialLinearGroups}$.
\end{enumerate}
\end{mthm}

By the main result of our Part I paper \cite[Theorem~A]{MimuraSako}, the disjoint union $\bigsqcup_{n\in \mathbb{N}}\mathrm{Cay}(H_{n,p};S_n)$ has property A. By Theorem~\ref{theorem=Corollary}, $\bigsqcup_{n\in \mathbb{N}}\mathrm{Cay}(H_{n,p};T_n)$ admits a fibred coarse embedding into a Hilbert space. At the other end of the spectrum, by $(i)$ of Theorem~\ref{mtheorem=MainTheorem} (in the current paper), $\bigsqcup_{n\in \mathbb{N}}\mathrm{Cay}(H_{n,p};U_n)$ does $\mathrm{not}$ admit a fibred coarse embedding into $\mathcal{B}_{\mathrm{type}>1}$ or $\mathcal{CAT}(0)_{< \delta(p)}$. 

T. Pillon introduced a notion of \textit{fibred coarse amenability} \cite{Pillon} and showed that a box space of a group has this property if and only if the group has property A. In this aspect, it is furthermore plausible that $\bigsqcup_{n\in \mathbb{N}}\mathrm{Cay}(H_{n,p};T_n)$ and $\bigsqcup_{n\in \mathbb{N}}\mathrm{Cay}(H_{n,p};U_n)$ both  \textit{fail} to enjoy fibred property A. D. Sawicki \cite[Proposition~7.4]{SawickiPropertyA} also introduced a notion of \textit{piecewise property $A$} in the context of \textit{warped cones}, and showed a similar statement in that framework under certain conditions.

The method of constructing $(\Gamma_l)$ as in Theorem~\ref{theorem=Exotic} produces the following exotic example, which concerns markings of finite symmetric groups.

\begin{mthm}\label{mtheorem=SymmetricGroups}
There exist $(k_l)_{l\in \mathbb{N}}$ of a sequence of natural numbers at least $2$ with $\lim_{l\to \infty}k_l=\infty$ and two $($ordered$)$ systems of generators $(\Xi_l)_{l\in \mathbb{N}}$, $(\Omega_l)_{l\in \mathbb{N}}$ of symmetric groups $(\mathrm{Sym}(k_l))_{l\in \mathbb{N}}$ that satisfy all of the following.
\begin{enumerate}[$(1)$]
    \item For all $l\in \mathbb{N}$, $\sharp(\Xi_l)=8$ and $\sharp(\Omega_l)=9$. For each $l\in \mathbb{N}$, $\Omega_l$ is constructed by adding one extra element to $\Xi_l$.  
\item The disjoint union $\bigsqcup_{l\in \mathbb{N}}\mathrm{Cay}(\mathrm{Sym}(k_l);\Xi_l)$ has property $\mathrm{A}$.
  \item The disjoint union $\bigsqcup_{l\in \mathbb{N}}\mathrm{Cay}(\mathrm{Sym}(k_l);\Omega_l)$ does $\mathrm{not}$ admit a fibred coarse embedding as a disjoint union into any of these spaces: 
  \begin{itemize}
    \item Banach spaces of non-trivial type, and Banach spaces that are sphere equivalent to Banach spaces of non-trivial type.    
\item Elements in $\mathcal{CAT}(0)_{<1}$.
\end{itemize}
\end{enumerate}
\end{mthm}

The construction as in Theorem~\ref{mtheorem=SymmetricGroups} is done in a completely explicit manner; see Subsection~\ref{subsection=ProofOfExotic} for details. For the proofs of Theorems~\ref{theorem=Exotic} and \ref{mtheorem=SymmetricGroups}, we utilize the notion of \textit{embedded expanders}; see Definition~\ref{definition=EmbeddedExpanders} and Proposition~\ref{proposition=Embedded}.

Our proof of Theorem~\ref{mtheorem=MainTheorem} is inspired by a trick by Gromov, \cite[Proposition~4.4]{deCornulierTesseraValette} for Hilbert spaces and \cite[Section~9]{NaorPeres} for general Banach spaces, as we will explain in Sections~\ref{section=Idea} and \ref{section=TheoremA(i)}. Independently to our results, S. Arnt \cite{Arnt} applied this trick in a special situation where the coarse disjoint union is a box space (in particular, all $G_m$, $m\in \mathbb{N}$, are finite) and the target class consists only of Banach spaces. For the case where $\mathcal{M}=\mathcal{H}\mathrm{ilbert}$, V. Alekseev and Finn-Sell \cite{AlekseevFinnSell} extended the framework of Theorem~\ref{mtheorem=MainTheorem} for the case where $(\mathbf{G}_m)_m$ is a LEF approximation of $\mathbf{G}_{\infty}$, see Definition~\ref{definition=RFLEFLEA}, to a sofic approximation of a sofic group. However, in that generality, only one direction (the direction of $(i)$ in Theorem~\ref{mtheorem=MainTheorem}) can be deduced; see the construction of a counterexample to the other direction by T. Kaiser \cite{Kaiser}, which is explained below Theorem~5.3 in the concerning reference \cite{Kaiser}. Compare also with our points $(a)$, $(b)$ with LEA approximations, and the case where $\mathcal{M}$ is general.

\ 

\paragraph{\bf Organization of the paper:} In Section~\ref{section=Preliminaries}, we briefly explain the space of marked groups and the Cayley topology, and the definition of fibred coarse embeddings (as a disjoint union). In Section~\ref{section=OperationsOnMetricSpaces}, we formulate several operations to classes of metric spaces and provide examples of our interest. We also provide a model example in Subsection~\ref{subsection=CAT(0)} to prove closeness properties under formation of these operations. In section~\ref{section=Idea}, we explain the key idea to non-linear version of Gromov's trick in relation to (pointed) metric ultraproducts. Section~\ref{section=TheoremA(i)} is devoted to the proof of  $(i)$ of Theorem~\ref{mtheorem=MainTheorem}. It is done by the non-linear version of Gromov's trick. In Section~\ref{section=TheoremA(ii)}, we prove $(ii)$ of Theorem~\ref{mtheorem=MainTheorem} and Corollary~\ref{mcorollary=a-T-menable} (and hence Theorem~\ref{theorem=Corollary} as well). Section~\ref{section=Gadgets} is for description of \textit{the absorption trick}, which plays a key role in the proof of Theorem~\ref{mtheorem=NonExact}. In Section~\ref{section=Examples}, we discuss various examples to apply Theorem~\ref{mtheorem=MainTheorem} (and Proposition~\ref{proposition=RootedGraphUniformCoarseEmbedding}), including the proofs of Corollary~\ref{corollary=SpecialLinearGroups} (and hence Corollary~\ref{corollary=SpecialLinear} as well) and Proposition~\ref{proposition=Products}. Theorem~\ref{mtheorem=NonExact} (and hence Theorem~\ref{theorem=NonExact} as well) is proved in Subsection~\ref{subsection=NonExact}; Theorem~\ref{theorem=Exotic} and Theorem~\ref{mtheorem=SymmetricGroups} are verified, respectively, in Subsections~\ref{subsection=UpperTriangular} and \ref{subsection=ProofOfExotic}.

\section{Preliminaries}\label{section=Preliminaries}
\subsection{Space of $k$-marked groups and Cayley topology}\label{subsection=CayleyTopology}
We recall basic facts of the Cayley topology from our Part I paper \cite{MimuraSako}; see Subsection~2.1 there for more details. Fix $k\in \mathbb{N}_{\geq 1}$. A $k$-marked group $\mathbf{G}=(G;S)=(G; s_1, s_2, \ldots, s_k)$ is a pair of a finitely generated group $G$ and an \textit{ordered} $k$-tuple $S=(s_1,\ldots ,s_k)$ of generators of $G$ (as a group). From a $k$-marked group $\mathbf{G}$, we construct two combinatorial objects, the \textit{Cayley diagram} $\mathrm{CayD}(\mathbf{G})$ and the \textit{Cayley graph} $\mathrm{Cay}(\mathbf{G})$ of $\mathbf{G}$ as follows. The former is defined as a diagram (edge-colored and edge-oriented graph), with the edge coloring set $[k]$, by setting the vertex set as $G$ and by putting edges of the form $(g,s_jg)$ with orientation from $g$ to $s_jg$ in color $j(\in [k])$ for every $j\in [k]$ and every $g\in G$. The latter is the graph (with no edge colorings or no edge orientations) constructed by forgetting the edge-colorings/orientations of $\mathrm{CayD}(\mathbf{G})$. Both of them are endowed with the shortest path metric $d_{\mathbf{G}}$ (in $\mathrm{CayD}(\mathbf{G})$, we ignore the edge-orientation to consider $d_{\mathbf{G}}$) on the vertex set $G$. In this way, we regard $\mathrm{CayD}(\mathbf{G})$ and $\mathrm{Cay}(\mathbf{G})$ as geometric objects. We also consider $\mathbf{G}$ itself as a metric space with this metric $d_{\mathbf{G}}$; in other words, $d_{\mathbf{G}}$ on $G$ is the right-invariant word metric with respect to $S$. 

For $\emptyset \ne Y\subseteq G$ and for $R\in \mathbb{N}_{\geq 1}$, denote by $\partial_{\mathbf{G}}(Y,R)$ the $R$-neighborhood of $Y$ in $d_{\mathbf{G}}$, namely, the set of all $h\in G$ such that there exists $g\in Y$ such that $d_{\mathbf{G}}(g,h)\leq R$. If $Y=\{g\}$, then we simply write $\partial_{\mathbf{G}}(\{g\},R)$ as $B_{\mathbf{G}}(g,R)$ (closed ball of radius $R$ centered at $g$). In this setting, we define $B_{\mathrm{CayD}(\mathbf{G})}(g,R)$ by restricting the vertex set of $\mathrm{CayD}(\mathbf{G})$ to $B_{\mathbf{G}}(g,R)$ and by taking the induced sub-diagram (more precisely, we collect all edges connecting vertices in $B_{\mathbf{G}}(g,R)$ with remembering its edge-colorings/orientations). By declaring $g$ to be the root, $B_{\mathrm{CayD}(\mathbf{G})}(g,R)$ has the  structure of a \textit{rooted diagram}. Note that $B_{\mathrm{CayD}(\mathbf{G})}(e_{G},R)$ \textit{completely remembers the multiplication table of $G$  up to word length $\lfloor R/2\rfloor$}. 

Denote by $\mathcal{G}(k)$ the set of all $k$-marked groups (up to marked group isomorphisms). This space is equipped with a natural topology, the \textit{Cayley topology},
which is metrizable and \textit{compact}. One definition of that topology is the induced topology of the product topology on $\{0,1\}^{F_k}$ to the set of all normal subgroups in $F_k$; there is a natural one-to-one correspondence between that subset of $\{0,1\}^{F_k}$ and $\mathcal{G}(k)$ by the standard marked quotient map $\mathbf{F_k}\twoheadrightarrow \mathbf{G}$. Another characterization of this topology is the topology of \textit{local} convergence (also known as the \textit{Gromov--Hausdorff} convergence in this setting) among rooted diagrams, as stated in the following lemma (Lemma~2.4 in \cite{MimuraSako}). Here for two groups $G,H$ and for two subsets $e_G\in K_1\subseteq G$ and $e_H\in K_2\subseteq H$, a map $\beta \colon K_1\to K_2$ is called a \textit{partial homomorphism} if 
\[
\textrm{for all $g_1,g_2 \in K_1$ \textit{such that} $g_1g_2\in K_1$,}\quad \beta(g_1g_2)=\beta(g_1)\beta(g_2)
\]
holds true. The map $\beta$ is called a \textit{partial isomorphism} if it is furthermore bijective.

The Cayley topology on $\mathcal{G}(k)$ is identical to the relative topology of the Chabauty topology for $F_k$; see \cite[Remark~2.5]{MimuraSako}.

\begin{lemma}\label{lemma=Neighborhood}
In $\mathcal{G}(k)$, $(\mathbf{G}_m)_{m\in \mathbb{N}}$ converges to $\mathbf{G}_{\infty}$ if and only if the following holds true:
\begin{align*}
\textrm{``}&\textrm{For every $m\in \mathbb{N}$, there exists $R_m\in \mathbb{N}$ such that $\lim_{m\to \infty}R_m=+\infty$ and } \tag{$\star$}\\
&\textrm{ $B_{\mathrm{CayD}(\mathbf{G}_m)}(e_{G_m},R_m)\cong B_{\mathrm{CayD}(\mathbf{G}_{\infty})}(e_{G_{\infty}},R_m)$ \textit{as rooted diagrams.''}} 
\end{align*}
Here an isomorphism of rooted diagrams means a graph automorphism that preserves edge-colorings $($in $[k]$$)$ and edge-orientations and that sends the root of the former diagram to the root of the latter. 

In other words, for $\mathbf{G}=(G;s_1,\ldots ,s_k)\in \mathcal{G}(k)$, if we define for each $R\in \mathbb{N}$,
\begin{align*}
N(\mathbf{G},R)=\{&\mathbf{H}=(H;t_1,\ldots ,t_k)\in \mathcal{G}(k): 
\textrm{the map $t_j\mapsto s_j$ induces}\\ 
&\textrm{a partial isomorphism} \quad \beta_{\mathbf{H},\mathbf{G},R}\colon B_{\mathbf{H}}(e_{H},R) \to B_{\mathbf{G}}(e_{{G}},R).\},
\end{align*}
then $\{N(\mathbf{G},R)\}_{R\in \mathbb{N}}$ forms an $($open$)$ neighborhood system of $\mathbf{G}$. 
\end{lemma}

\begin{proof}
For every $m\in \mathbb{N}$, set $R_m$ to be the largest $R$ such that $R\leq m$ and that $m\geq m_R$, where $m_R$ is as in \cite[Lemma~2.4]{MimuraSako}. Then, it follows that $\lim_{m\to \infty}R_m=+\infty$.
\end{proof}

We write the convergence in the Cayley topology as $\lim_{m\to \infty}\mathbf{G}_m=\mathbf{G}_{\infty}$ or $\mathbf{G}_m\stackrel{\mathrm{Cay}}{\to} \mathbf{G}_{\infty}$. The readers who are not familiar with the Cayley topology may consult Section~5 in our Part I paper \cite{MimuraSako}, specially Lemma~5.1 therein, for pedagogical examples of the Cayley convergence.

We also recall the definitions of RF/LEF/LEA groups; recall these abbreviations from the introduction.
\begin{definition}\label{definition=RFLEFLEA}
Let $G$ be a finitely generated group.
\begin{enumerate}[$(1)$]
  \item The group $G$ is said to be \textit{RF} if there exists a sequence $(N_m)_{m\in \mathbb{N}}$ of finite index normal subgroups of $G$ such that 
\[
\mathrm{liminf}_{m\to \infty}N_m(=\bigcup_{m\in \mathbb{N}}\bigcap_{n\in \mathbb{N}_{\geq m}}N_n)=\{e_G\}
\]
holds true.
  \item The group $G$ is said to be \textit{LEF} if for some (equivalently, every) marking $\mathbf{G}$ of $G$, there exists a Cayley convergent sequence consisting of finite marked groups that converges to $\mathbf{G}$ .
  \item The group $G$ is said to be \textit{LEA} if for some (equivalently, every) marking $\mathbf{G}$ of $G$, there exists a Cayley convergent sequence consisting of amenable marked groups that converges to $\mathbf{G}$ .
\end{enumerate}
We say a sequence $(\mathbf{G}_m)_m$ is a \textit{LEF} (respectively, \textit{LEA}) \textit{approximation} of $\mathbf{G}$ if it consists of finite (respectively, amenable) marked groups converging to $\mathbf{G}$ in the Cayley topology. A LEF approximation is moreover called an \textit{RF approximation} if it consists of \textit{marked group quotients}; namely, for every $m$, there exists a group quotient map $\varphi_m\colon G\twoheadrightarrow G_m$ that sends the marking $S=(s_1,\ldots ,s_k)$ of $\mathbf{G}$ to that $S_m=(s_1^{(m)},\ldots ,s_k^{(m)})$ of $\mathbf{G}_m$ with preserving the orders on them: $\varphi_m(s_j)=s_j^{(m)}$ for every $j\in [k]$. An RF approximation of $(G;S)$ is of the form $((G/N_m;S\ \mathrm{mod}\ N_m))_{m\in \mathbb{N}}$, where $(N_m)_{m\in \mathbb{N}}$ satisfies the conditions of $(1)$ of Definition~\ref{definition=RFLEFLEA}.
\end{definition}
In $(1)$ in the definition above, we may relax the condition of $\mathrm{liminf}_{m\to \infty}N_m=\{e_G\}$ to $\bigcap_{m}N_m=\{e_G\}$; indeed, if we set new $(N_m')_m$ as $N'_m=\bigcap_{n\in \mathbb{N}_{\leq m}}N_n$, then $\mathrm{liminf}_{m\to \infty}N'_m=\{e_G\}$ is equivalent to $\bigcap_{m}N'_m=\{e_G\}$. However, if we hope to have a RF approximation out of $(N_m)_m$ by taking marked group quotients, then the right condition is the former one, \textit{not} the latter.

\begin{remark}\label{remark=FinitelyPresented}
If a marked group $\mathbf{G}$ is \textit{finitely presented} (this is independent of the choice of markings), then the set of all marked group quotients of $\mathbf{G}$ forms an \textit{open} set. Hence, in that case, every LEF approximation  eventually is a RF approximation; see \cite{VershikGordon} and \cite[Subsection~2.1]{MimuraSako}.
\end{remark}

\subsection{Fibred coarse embeddings}\label{subsection=FibredCoarseEmbeddings}
Recall from the introduction the construction of the disjoint union $\bigsqcup_{m\in \mathbb{N}}X_m$ out of a sequence of metric spaces $(X_m,d_m)_{m\in \mathbb{N}}$. If every $X_m$ has finite diameter (the \textit{diameter} is defined as the supremum of the distances between two points in the metric space), then we may construct a \textit{coarse disjoint union} $\coprod_{m\in \mathbb{N}} X_m$, which is a (genuine) metric space. However, we do not go into details in this paper; instead, we refer the readers to \cite[Definition~2.17.$(2)$]{MimuraSako} on this notion.

In this paper, we study fibred coarse embeddings from the disjoint union constructed above. For this purpose, we relax the definition of the fibred coarse embeddings as follows. For a generalized metric space $X$, we say that $X$ is \textit{uniformly locally finite} if for every $R\in \mathbb{R}_{\geq 0}$, there exists $C\in \mathbb{N}$ such that every closed $R$-ball (for every center $x\in X$) has cardinality at most $C$. For a sequence of metric spaces $(X_m)_{m\in \mathbb{N}}$, we say that it is \textit{equi-uniformly locally finite} if every $X_m$ is uniformly locally finite and if moreover $C=C(R)$ is taken uniformly on $m\in \mathbb{N}$ for every $R\in \mathbb{R}_{\geq 0}$. If $(X_m)_{m\in \mathbb{N}}$ is equi-uniformly locally finite, then the disjoint union $X=\bigsqcup_{m\in \mathbb{N}}X_m$ is uniformly locally finite.

\begin{definition}\label{definition=GeneralizedFibredCoarseEmbeddings}
Let $\mathcal{M}$ be a non-empty class of metric spaces. Let $(X,d)=\bigsqcup_{m\in \mathbb{N}}X_m$ be the disjoint union of a  sequence of metric spaces $(X_m)_{m\in \mathbb{N}}$ that is equi-uniformly locally finite. Let $\rho,\omega \colon[0,\infty)\to [0,\infty)$ be two non-decreasing proper functions. 
\begin{enumerate}[$(i)$]
  \item We say that $X$ admits a \textit{$(\rho,\omega)$-fibred coarse embedding into $\mathcal{M}$ as a disjoint union} if there exists $M\in \mathcal{M}$ such that the following holds true: There exist 
\begin{itemize} 
\item a field of metric spaces $(M_x)_{x \in X}$ over $X$ such that 
each $M_x$ is isometric to $M$,
\item a section 
$s \colon X \to \bigsqcup_{x \in X} M_x$, (namely, $s(x) \in M_x$ for every $x\in X$),
\end{itemize}
such that for every $R \in \mathbb{R}_{\geq 0}$ there exists $m^{(R)}\in \mathbb{N}_{\geq 1}$ such that for each non-empty subset $C \subseteq X \setminus  \left(\bigsqcup_{\mathbb{N}_{<m^{(R)}}}X_m\right)$ of diameter at most $R$, there  
exists a ``trivialization''
$t_{C,R} \colon (M_x)_{x \in C} \to C \times M$
 such that the following holds. The restriction of $t_{C,R}$ to the fibre $M_x$, $x \in C$, is an isometry $t_{C,R}(x) \colon M_x \to M$ that satisfies
\begin{enumerate}[$(1)$]
\item
for every $x_1, x_2 \in C$, 
\[\rho(d(x_1, x_2)) \le d_M(t_{C,R}(x_1)(s(x_1)), t_{C,R}(x_2)(s(x_2))
\le \omega(d(x_1, x_2));
\]
\item
for every two subsets $C_1,C_2 \subseteq X \setminus  \left(\bigsqcup_{m\in \mathbb{N}_{<m^{(R)}}}X_m\right)$ of diameter at most $R$ with $C_1 \cap C_2 \neq \emptyset$,
there exists an isometry 
$t_{C_1, C_2,R} \colon M \rightarrow M$ such that 
$t_{C_1,R} (x) \circ t_{C_2,R}(x)^{-1} = t_{C_1,C_2,R}$
for all $x \in C_1 \cap C_2$.
\end{enumerate}
  \item We say $(\rho,\omega)$ is a \textit{control pair for fibred coarse embeddings as a disjoint union} for $X$ into $\mathcal{M}$ if there exists a $(\rho,\omega)$-fibred coarse embedding from $X$ to $\mathcal{M}$ as a disjoint union. Denote by $\mathcal{CP}_{\mathcal{M}}^{\mathrm{fib}}(X)$ the set of all control pairs above. The functions $\rho$ and $\omega$ are, respectively, called a \textit{compression function} and an \textit{expansion function} in the setting above.
\end{enumerate}
\end{definition}

We say that $X$ \textit{admits a fibred coarse embedding into} $M$ as a disjoint union if for some  pair $(\rho,\omega)$ of non-decreasing and  proper functions, the condition of $(i)$ is satisfied. This is equivalent to saying that
\[
\mathcal{CP}_{\mathcal{M}}^{\mathrm{fib}}(X)\ne \emptyset.
\]

Note that  if a non-empty subset $C$ of $X=\bigsqcup_{m\in \mathbb{N}}X_m$ is of bounded diameter, then there exists a unique $m\in \mathbb{N}$ such that $C\subseteq X_m$. 

\begin{remark}\label{remark=OriginalDefinition}
In the original formulation in \cite[Definition~2.1]{ChenWangYu} (for the case $\mathcal{M}$ being the class of all Hilbert spaces), for each $R\in \mathbb{N}$, we are allowed to choose a \textit{bounded} exceptional set $K$, and consider $C$ of diameter at most $R$ from $X\setminus K$. In our definition of fibred corase embeddability \textit{as a disjoint union}, we relax this process and allow to take $K=\bigsqcup_{m\in \mathbb{N}_{<m^{(R)}}}X_m$, the disjoint union of \textit{finitely many} components in $X=\bigsqcup_{m\in \mathbb{N}} X_m$.

Therefore, in Definition~\ref{definition=GeneralizedFibredCoarseEmbeddings}, if all $X_m$, $m\in \mathbb{N}$, are finite, then our notion of the fibred coarse embeddability as a disjoint union coincides with that of the fibred coarse embeddability in the original sense from a coarse disjoint union $\coprod_{m\in \mathbb{N}}X_m$. 

In this paper, we discuss quantitative aspects (control pairs) for fibred coarse embeddings (as a disjoint union) as well as qualitative aspects (the property itself). For this purpose, disjoint unions are more suited than coarse ones.\end{remark}

\begin{remark}\label{remark=FromFibredToGenuine}
The fibred coarse embeddability into $M$ (as a disjoint union) is weaker than the the (genuine) coarse embeddability. Indeed, if $f\colon X=\bigsqcup_{m}X_m \to M$ is a coarse embedding with control pair $(\rho,\omega)$, then set $m^{(R)}=0$ for all $R$ and $M_x=M$ for all $x\in X$. Let $s\colon X\to  \bigsqcup_{x\in X}M$ be $s(x)=f(x)$, and $t_{C,R}=\mathrm{id}_M$ for all $R$ and for all $C$ of diameter at most $R$. This gives rise to a $(\rho,\omega)$-fibred coarse embedding as a disjoint union into $M$.
\end{remark}

If $\mathcal{M}$ consists of Banach spaces, then we furthermore assume that all isometries in the conditions as in Definition~\ref{definition=GeneralizedFibredCoarseEmbeddings} are affine. However, by the Mazur--Ulam theorem \cite[Chapter~14.1]{BookBenyaminiLindenstrauss} and by formation of the Taylor complexification, this issue is not essential in many cases. Therefore, in the present paper, hereafter we do not discuss this matter.

\begin{remark}\label{remark=DependenceOnR}
Though it was implicit in the original formulation \cite[Definition~2.1]{ChenWangYu}, the ``trivialization'' $t_C=t_{C,R}$ in Definition~\ref{definition=GeneralizedFibredCoarseEmbeddings} \textit{is allowed to be incompatible on changing $R$}. More precisely, for $0\leq R_1<R_2$ and for $C\subseteq X \setminus  \left(\bigsqcup_{\mathbb{N}_{<m^{(R_2)}}}X_m\right)$ of diameter at most $R_1$, we do \textit{not} require that $t_{C,R_1}=t_{C,R_2}$. This observation is important in our proof of $(ii)$ of Theorem~\ref{mtheorem=MainTheorem}.
\end{remark}

We observe the following two lemmata. Here for a metric space $X$, $x\in X$ and $R\in \mathbb{R}_{> 0}$, denote by $B_X(x,R)$ the closed ball of radius $R$ centered at $x$.

\begin{lemma}\label{lemma=BallsFibredCoarseEmbedding}
Let $(X_m)_{m\in \mathbb{N}}$ be a sequence of metric spaces  that is equi-uniformly locally finite. Let $X = \bigsqcup_{m \in \mathbb{N}} X_m$. Let $\mathcal{M}$ be a non-empty class of metric spaces. Let $\rho,\omega\colon [0,\infty) \to [0,\infty)$ be two non-decreasing proper functions. Then, $X$ admits a $(\rho,\omega)$-fibred coarse embedding into $\mathcal{M}$ as a disjoint union if and only if there exists $M\in \mathcal{M}$ such that the following holds true: There exist 
\begin{itemize}
\item
a field of metric spaces $(M_x)_{x \in X}$ which are all isometric to $M$,
\item
there exists a section $s \colon X \to \bigsqcup_{x \in X} M_x$,
\item
a  sequence of non-negative real numbers $(R'_m)_{m\in \mathbb{N}}$ such that $\lim_{m\to \infty} R'_m=+\infty$,
\item
a local trivialization $t_{g,R'_m} \colon \bigsqcup_{x \in B_{X_m}(g, R'_m)} M_x \to B_{X_m}(g, R'_m) \times M$, for each $m \in \mathbb{N}$ and each 
$g \in X_m$,
\end{itemize}
such that the following hold.
\begin{enumerate}[$(1)$]
\item
For every $n\in \mathbb{N}$, for  every  $g \in X_m$ and every $x \in B_{X_m}(g, R'_m)$,
the restriction $t_{g,R'_m}(x) \colon M_x \to M$ of $t_{g,R'_m}$ is isometry;
\item
for every $x_1, x_2 \in B_{X_m}(g, R'_m)$, 
\[
\rho(d(x_1, x_2)) \le d_M(t_{g,R'_m}(x_1)(s(x_1)), t_{g,R'_m}(x_2)(s(x_2))
\le \omega(d(x_1, x_2));
\]
\item
if $B_{X_m}(g_1, R'_m) \cap B_{X_m}(g_2, R'_m) \neq \emptyset$, there exists an isometry 
$t_{g_1, g_2,R'_m} \colon M \rightarrow M$ such that 
$t_{g_1,R'_m} (x) \circ t_{g_2,R'_m}(x)^{-1} = t_{g_1,g_2,R}$
for all $x \in B_{X_m}(g_1, R'_m) \cap B_{X_m}(g_2, R'_m)$.
\end{enumerate} 
\end{lemma}

\begin{lemma}\label{lemma=Subsets}
In the setting of  Lemma~$\ref{lemma=BallsFibredCoarseEmbedding}$, let $Y=\bigsqcup_{n\in \mathbb{N}}Y_{m_n}$ be such that $(m_n)_{n\in \mathbb{N}}$ is a subsequence of $(m)_{m\in \mathbb{N}}$ and for each $n\in \mathbb{N}$, $Y_{m_n}$ is a non-empty subset of $X_{m_n}$ equipped with the induced metric. Then, if $X$ admits a fibred coarse embedding into $\mathcal{M}$ as a disjoint union with control pair $(\rho,\omega)$, then so does $Y$.
\end{lemma}

\begin{proof}[Proofs of Lemma~$\ref{lemma=BallsFibredCoarseEmbedding}$ and Lemma~$\ref{lemma=Subsets}$]
Lemma~$\ref{lemma=Subsets}$ is obvious. 

To show that the $(\rho,\omega)$-fibred coarse embeddability as a disjoint union implies the conditions as in Lemma~\ref{lemma=BallsFibredCoarseEmbedding}, take $R\mapsto m^{(R)}$ as in Definition~\ref{definition=GeneralizedFibredCoarseEmbeddings}. set $R'_m=\min\{\sup\{r\in \mathbb{R}_{\geq 0}: m\geq m^{(r)}\},m\}$ for every $m\in \mathbb{N}$, where we set $m_0=0$. By construction, $\lim_{m\to \infty}R'_m=+\infty$. For $g\in X_m$, set $t_{g,R'_m}$ as $t_{B_{X_m}(g,R_m'),R_m'}$.

To show the converse direction, for each $R\in \mathbb{R}_{>0}$, set $m^{(R)}=\max\{m\in \mathbb{N}: R'_m< R\}+1$. Since $\lim_{m\to \infty}R'_m=+\infty$, it holds that $m_R\in \mathbb{N}$. For each $C\subseteq X\setminus \left(\bigsqcup_{m<m^{(R)}}X_m\right)$ of diameter at most $R$, there exist (unique) $m\in \mathbb{N}_{\geq m^{(R)}}$ and (non-unique) $g\in X_m$ such that $C\subseteq B_{X_m}(g,R)$. Since $R'_m\geq R$ for all $m\in \mathbb{N}_{\geq m^{(R)}}$, we may define $t_{C,R}$ as the restriction of $t_{g,R'_m}$ on $C$. There is an ambiguity on the choice of $g$; however, if we fix the choices for all $C$, then condition~$(3)$ as in Lemma~\ref{lemma=BallsFibredCoarseEmbedding} ensures condition~$(2)$ as in Definition~\ref{definition=GeneralizedFibredCoarseEmbeddings}. Recall also Remark~\ref{remark=DependenceOnR}.
\end{proof}

\subsection{Equivariant coarse embeddings and a-$\mathcal{M}$-menability}
In Section~4 in our Part I paper \cite{MimuraSako}, we recall the definition of a-$\mathrm{T}$-menability for finitely generated groups. Here we generalize this concept in terms of other target spaces. The following property should be stated as a-$\mathrm{F}_{\mathcal{M}}$-menability in the strict sense. However, through communications with Arnt, we have agreed to use the terminology of a-$\mathcal{M}$-menability to avoid messes on notation. In the following definition, recall that a marked group $\mathbf{G}$ is naturally equipped with the metric $d_{\mathbf{G}}$; see Subsection~\ref{subsection=CayleyTopology}.

\begin{definition}\label{definition=a-M-menable}
Let $\mathbf{G}$ be a marked group and $\mathcal{M}$ be a non-empty class of metric spaces. 
\begin{enumerate}[$(1)$]
  \item The marked group $\mathbf{G}$ is said to be \textit{a-$\mathcal{M}$-menable} if there exist $M\in \mathcal{M}$ and a coarse embedding $f\colon (G,d_{\mathbf{G}})\to M$ such that the following condition is satisfied: The map $f$ is of the form
\[
f(g)=y \cdot \alpha(g),
\]
where $\alpha\colon M\curvearrowleft G$ is a \textit{right} action by isometries and $y\in M$. We say that a coarse embedding $f$ is \textit{$($G-$)$equivariant} if it satisfies the condition above. 
 \item We say a finitely generated group $G$ is \textit{a-$\mathcal{M}$-menable} if for some (equivalently, all) marking $\mathbf{G}=(G;S)$ of $G$, $\mathbf{G}$ is a-$\mathcal{M}$-menable.
  \item The pair $(\rho,\omega)$ of two non-decreasing proper functions $[0,\infty)\to [0,\infty)$ is called an \textit{equivariant control pair} for $\mathbf{G}$ into $\mathcal{M}$ if there exist $M\in \mathcal{M}$ and an $G$-equivariant coarse embedding $f\colon G\to M$ such that $(\rho,\omega)$ is a control pair for $f$. In this case, we call $\rho$ and $\omega$, respectively, an \textit{equivariant compression function} and an \textit{equivariant expansion function} from $\mathbf{G}$ into $\mathcal{M}$.
 \item We denote by $\mathcal{CP}^{\sharp}_{\mathcal{M}}(\mathbf{G})$ be the set of all equivariant control pairs for $\mathbf{G}$ into $\mathcal{M}$.
\end{enumerate}
\end{definition}
In the definition above, we take a right action, not a left action, because we equip  marked groups with right-invariant metrics.

Let $\mathcal{H}\mathrm{ilbert}$ denote the class of all Hilbert spaces. Then the notion of a-$\mathcal{H}\mathrm{ilbert}$-menablity coincides with that of a-$\mathrm{T}$-menability.

\begin{remark}\label{remark=ControlledPair}
We warn that, unlike some other literature, the control pair $(\rho,\omega)$ is regarded as the pair of \textit{concrete} functions, not only as growth orders. In particular, if $(\rho,\omega)\in \mathcal{CP}^{\sharp}_{\mathcal{M}}(\mathbf{G})$ and if $C_1,C_2>0$, it does \textit{not} necessarily hold that $(C_1\rho,C_2\omega)\in \mathcal{CP}^{\sharp}_{\mathcal{M}}(\mathbf{G})$. If we consider a class $\mathcal{M}$ that is not necessarily closed under rescaling, this remark applies even when $C_1=C_2$. A similar issue to above applies to $\mathcal{CP}_{\mathcal{M}}(X)$ and  $\mathcal{CP}^{\mathrm{fib}}_{\mathcal{M}}(X)$. 

For a fixed finitely generated group $G$ and  for a fixed equivariant coarse embedding $f\colon G\to M$, equivariant compression functions for $f$ depends on markings of $\mathbf{G}$ up to constant multiplication. Therefore, the set $\mathcal{CP}^{\sharp}_{\mathcal{M}}(\mathbf{G})$ \textit{does} depend on the choice of markings; recall our discussion above. This observation is important because the Cayley boundary $\partial_{\mathrm{Cay}}(\mathbf{G}_m)_m$ of a sequence $(\mathbf{G}_m)_m$ may possibly consist of \textit{infinitely many} marked groups.
\end{remark}

\section{Several operations on (pointed) metric spaces}\label{section=OperationsOnMetricSpaces}

A \textit{pointed metric} space $(M,y)$ is a (genuine) metric space $M=(M,d_M)$ with a base point $y\in M$. We define certain operations on a class of metric spaces
\[
\mathcal{M}\quad \mapsto \quad \mathcal{UP}(\mathcal{M}),\ \mathcal{F}_q(\mathcal{M}),\ \mathcal{FS}_q(\mathcal{M})\quad \textrm{and} \quad \ell_q(\mathcal{M}),\quad \textrm{for }q\in[1,\infty),
\]
which appears in the statement of our main theorem, Theorem~\ref{mtheorem=MainTheorem}. In this section, we first give formulations of these operations; then we explain several classes of metric spaces that are closed under formation of these operations.

\subsection{Direct $\ell_q$-products and metric ultraproducts}\label{subsection=Ultraproduct}

The direct $\ell_q$-product of pointed metric spaces is defined as follows.

\begin{definition}\label{definition=ell_qSum}
Let $q\in [1,\infty)$. Let $B$ be a non-empty set that is at most countable. Let $(r_j)_{j\in B}$ be such that $r_j\in (0,\infty)$ for all $j\in B$. For a sequence $(M_j,d_j,y_j)_{j\in B}$ of pointed metric spaces, define the \textit{$($pointed$)$ $\ell_q$-product with scaling $(r_j)_j$}, denoted by $(\prod_{j\in B}(M_j,y_j,r_j))_{\ell_q}$, by 
\[
\left(\prod_{j\in B}(M_j,y_j,r_j)\right)_{\ell_q}=\left\{ (z_j)_{j\in B}: \left(\sum_{j\in B} (r_jd_j(z_j,y_j))^q\right)^{1/q}<\infty\right\}
\]
with the metric 
\[
d_{q,(r_j)_j}((z_j)_j,(w_j)_j)=\left(\sum_{j\in B} (r_jd_j(z_j,w_j))^q\right)^{1/q}, \quad (z_j)_j, (w_j)_j \in \left(\prod_{j\in B}(M_j,y_j,r_j)\right)_{\ell_q}\]
and with the base point $(y_j)_j$. 

If the scaling factor $(r_j)_j$ is all $1$ ($r_j=1$ for all $j\in B$), then we simply write $(\prod_{j\in B}(M_j,y_j,1))_{\ell_q}$ as $(\prod_{j\in B}(M_j,y_j))_{\ell_q}$. This space is called the \textit{$($pointed$)$ $\ell_q$-product} of $(M_j,y_j)_j$. (If $M_j$ are Banach spaces, then it is usually called the pointed $\ell_q$-sum.)

If $\sharp(B)<\infty$, then (the isometry type of) the resulting space $(\prod_{j\in B}(M_j,y_j,r_j))_{\ell_q}$ does not depend on the choice $(y_j)_j$ of base points. In that case, we write it as $(\prod_{j\in B}(M_j,r_j))_{\ell_q}$ for short.
\end{definition}

We now switch our subject to \textit{$($pointed$)$ metric ultraproducts}. An \textit{ultrafilter $\mathcal{U}$ over} $\mathbb{N}$ has a one-to-one correspondence to a probability mean $\nu$ (\textit{finitely additive} measure with $\nu(\mathbb{N})=1$) on $\mathbb{N}$ that is $\{0,1\}$-valued and is defined over all subsets of $\mathbb{N}$. The correspondence is given by setting that $A\in \mathcal{U}$ if and only if $\nu(A)=1$. The cofinite filter $\mathcal{U}_{\mathrm{cofin}}=\{A\subseteq \mathbb{N}:\sharp(\mathbb{N}\setminus A)<\infty\}$ is a filter, but not an ultrafilter. A \textit{non-principal} ultrafilter $\mathcal{U}$ is an ultrafilter that includes $\mathcal{U}_{\mathrm{cofin}}$ (as a subfilter). In what follows, fix a non-principal ultrafilter $\mathcal{U}$ over $\mathbb{N}$. 

For a sequence $(r_m)_{m\in \mathbb{N}}$ in $\mathbb{R}$ and for $r_{\infty} \in \mathbb{R}$, we say that $\lim_{\mathcal{U}}r_m=r_{\infty}$ if it holds that
\[
\textrm{for every $\epsilon>0$,}\quad \{m\in \mathbb{N}:|r_{\infty}-r_m|<\epsilon\} \in \mathcal{U}.
\]
By local compactness and Hausdorff property of $\mathbb{R}$, it is standard to show that every \textit{bounded} real sequence $(r_m)_{m\in \mathbb{N}}$ has a unique $\mathcal{U}$-limit. The limit in general depends on the choice of a non-principal ultrafilter $\mathcal{U}$. However, if $\lim_{m\to  \infty}r_m$ exists, then $\lim_{\mathcal{U}}r_m$ coincides with the limit above.

We now consider a sequence $((M_m,d_m,y_m))_{m\in \mathbb{N}}$ of pointed metric spaces. Set 
\[
\left(\prod_{m\in \mathbb{N}}(M_m,y_m)\right)_{\ell_{\infty}}=\{(z_m)_{m\in \mathbb{N}}: \sup_{m\in \mathbb{N}}d_m(z_m,y_m)<\infty\}
\]
and define  $d_{\mathcal{U}}$ by setting for $(z_m)_m,(w_m)_m\in (\prod_{m\in \mathbb{N}}(M_m,y_m))_{\ell_{\infty}}$,
\[
d_{\mathcal{U}}((z_m)_m,(w_m)_m)=\lim_{\mathcal{U}}d_m(z_m,w_m).
\]
This is a \textit{pseudo}-metric, namely, $d_{\mathcal{U}}$ does not separate points in general. To obtain a genuine metric space,introduce an equivalence relation $\sim_{d_{\mathcal{U}}\equiv 0}$ on $(\prod_{m\in \mathbb{N}}(M_m,y_m))_{\ell_{\infty}}$ by defining $(z_m)_m\sim_{d_{\mathcal{U}}\equiv 0} (w_m)_m$ by $d_{\mathcal{U}}((z_m)_m,(w_m)_m)=0$. Finally, the quotient space 
\[
\lim_{\mathcal{U}}(M_m,y_m)=\left(\prod_{m\in \mathbb{N}}(M_m,y_m)\right)_{\ell_{\infty}}/\sim_{d_{\mathcal{U}}\equiv 0}
\]
is equipped with a genuine metric $d_{\mathcal{U}}$. We call the resulting space the \textit{$($pointed$)$ metric ultraproduct} of $(M_m,y_m)$ with respect to $\mathcal{U}$. We write the equivalence class with respect to $\sim_{d_{\mathcal{U}}\equiv 0}$ of $(z_m)_m$ as $[(z_m)_m]_{\mathcal{U}}$.

\subsection{The classes $\mathcal{UP}(\mathcal{M})$,  $\mathcal{F}_q(\mathcal{M})$, $\mathcal{FS}_q(\mathcal{M})$ and $\ell_q(\mathcal{M})$}\label{subsection=OperationsOnMetricSpaces}

\begin{definition}\label{definition=GeodesicSpaces}
A metric space $M$ is called a \textit{geodesic space} if for every $x,y\in M$, there exists a geodesic $c\colon [0,d(x,y)] \to M$ connecting $x$ and $y$.

The space $M$ is moreover said to be \textit{uniquely geodesic} if for every $x, y\in M$, there exists a $\mathrm{unique}$ geodesic $c\colon [0,d(x,y)]\to M$ from $x$ to $y$. For a uniquely geodesic space $M$ and for $x,y\in M$, let $[x,y]$ be the (uniquely determined) geodesic from $x$ to $y$.
\end{definition}
Here, recall from the introduction that by geodesics, we mean minimal ones.

We give the definitions of $\mathcal{UP}(\mathcal{M})$,  $\mathcal{F}_q(\mathcal{M})$, $\mathcal{FS}_q(\mathcal{M})$ and $\ell_q(\mathcal{M})$ out of a given class $\mathcal{M}$. Here $q\in [1,\infty)$ is a fixed exponent.

\begin{definition}\label{definition=MetricUltraProduct}
Let $\mathcal{M}$ be a non-empty class of metric spaces. We define $\mathcal{UP}(\mathcal{M})$ to be the class of all pointed metric ultraproducts (after forgetting metric ultraproducts) of a single  space in $\mathcal{M}$. More precisely, it is the class of all spaces (isometric to those) of the form $\lim_{\mathcal{U}}(M,y_m)$. Here $M\in \mathcal{M}$ and for every $m\in \mathbb{N}$, $y_m\in M$; $\mathcal{U}$ runs over all non-principal ultrafilters on $\mathbb{N}$.
\end{definition}

\begin{definition}\label{definition=OperationOnMetricSpaces}
Let $\mathcal{M}$ be a non-empty class of metric spaces. Fix $q\in [1,\infty)$. We define the following two new classes, $\mathcal{F}_q(\mathcal{M})$ and $\mathcal{FS}_q(\mathcal{M})$, of metric spaces constructed from $\mathcal{M}$.
\begin{enumerate}[$(1)$]
  \item We define $\mathcal{F}_q(\mathcal{M})$ as the class of all metric spaces (that is isometric to ones) that are constructed by the following three steps.
\begin{itemize}
  \item (Step~1.) Take $M\in \mathcal{M}$.
  \item (Step~2.) Consider all metric spaces of the form $\left(\prod_{f\in F}(M, \frac{1}{(\sharp (F))^{1/q}})\right)_{\ell_q}$ for non-empty finite sets $F$. Here $(\frac{1}{(\sharp (F))^{1/q}})_{f\in F}$ means that we take the constant scaling factor $\frac{1}{(\sharp (F))^{1/q}}$. 
  \item (Step~3.) Take an arbitrary sequence $((N_m,y_m))_{m\in \mathbb{N}}$, where for all $m\in \mathbb{N}$, $N_m=N_m(F^{(m)})$ lies in the class of  all metric spaces constructed in Step~2 that is associated with a finite set $F^{(m)}$ such that $\lim_{m\to \infty}\sharp(F^{(m)})=\infty$ and $y_m\in N_m$. Construct all metric spaces of the form $\lim_{\mathcal{U}}(N_m,y_m)$ (after forgetting the base points) for non-principal ultrafilters $\mathcal{U}$ of $\mathbb{N}$.
\end{itemize}
  \item The new class $\mathcal{FS}_q(\mathcal{M})$ is defined if every element $L$ in $\mathcal{F}_q(\mathcal{M})$ is a geodesic space. If this is the case, then we construct $\mathcal{FS}_q(\mathcal{M})$ in the following way.
\begin{itemize}
  \item If $\mathcal{M}$ consists only of Banach spaces, then every element $L$ in $\mathcal{F}_q(\mathcal{M})$ has a structure of affine Banach spaces. Then set $\mathcal{FS}_q(\mathcal{M})$ as the class of all Banach spaces isometrically affinely isomorphic to non-empty closed affine subspaces of $L$ for all $L\in \mathcal{F}_q(\mathcal{M})$.
  \item Otherwise, define $\mathcal{FS}_q(\mathcal{M})$ to be the class of all metric spaces isometric to non-empty closed and geodesically convex subsets $L_0$ of $L$ (equipped with the induced metric from $L$) for all $L\in \mathcal{F}_q(\mathcal{M})$. Here a non-empty subset $L_0\subseteq L$ is said to be \textit{geodesically convex} if for every $z,w \in L_0$ and for every geodesic $c\colon [0,d(z,w)]\to L$ from $z$ to $w$ in $L$, $c$ (more precisely, the image $c([0,d(z,w)])$) is included in $L_0$.
\end{itemize}
\end{enumerate}
\end{definition}

\begin{definition}\label{definition=ell_q}
Let $\mathcal{M}$ and $q$ be as in Definition~$\ref{definition=OperationOnMetricSpaces}$.  Then, we define $\ell_q(\mathcal{M})$ as the class of all metric spaces (that is isometric to ones) of the form $(\prod_{j\in B}(M_j,y_j))_{\ell_q}$ (after forgetting the base point) for a non-empty at most countable sets $B$ and for  $M_j\in \mathcal{M}$ and  $y_j \in M_j$ for $j\in B$.
\end{definition}

Note that unlike the construction of $\ell_q(\mathcal{M})$, in Step~1 of the construction of $\mathcal{F}_q(\mathcal{M})$, we use a \textit{single} $M\in \mathcal{M}$ to take the $\ell_q$-product with scaling. (Similarly for $\mathcal{UP}(\mathcal{M})$.) The symbol $\mathcal{F}$ in $(1)$ of Deinfition~\ref{definition=OperationOnMetricSpaces} stands for \textit{finite} and \textit{F{\o}lner}.  The symbol $\mathcal{FS}$ in $(2)$ of Definition~\ref{definition=OperationOnMetricSpaces} stands for \textit{F{\o}lner} and \textit{subspaces} (or \textit{subsets}).

\begin{remark}\label{remark=ScalingFactor}
The scaling factor $(\frac{1}{(\sharp (F))^{1/q}})_{f\in F}$ in Step~2 as in $(1)$ of Definition~\ref{definition=OperationOnMetricSpaces} is chosen exactly in order to ensure that the diagonal embedding $M\hookrightarrow \left(\prod_{f\in F}(M, \frac{1}{(\sharp (F))^{1/q}})\right)_{\ell_q}$; $z\mapsto (z,z,\ldots ,z)$ is isometric.
\end{remark}

\subsection{A model example of closeness under formation of several operations}\label{subsection=CAT(0)}

In the next subsection, we discuss several examples of classes $\mathcal{M}$ of metric spaces which are closed under formation of (some of) operations
\[
\mathcal{M}\quad \mapsto \quad \mathcal{UP}(\mathcal{M}),\ \mathcal{F}_q(\mathcal{M}),\ \mathcal{FS}_q(\mathcal{M}) \quad \textrm{and} \quad \ell_q(\mathcal{M}),
\]
for appropriate $q\in [1,\infty)$. The goal of this subsection is to provide a model example to verify closeness above; in Subsections~\ref{subsection=ClosedUnderOperations} and \ref{subsection=ClosedUnderOperationsMetric}, we will omit the arguments for it because the basic idea is exactly the same as one in this subsection. Our pedagogical example in this subsection is the class of all \text{complete $\mathrm{CAT}(0)$ spaces}. The reader who is familiar with the argument for closeness may skip this subsection.

\begin{definition}\label{definition=CAT(0)}
A metric space $M$ is said to be $\mathrm{CAT}(0)$ if it is a geodesic space (see Definition~\ref{definition=GeodesicSpaces}) and if for every $x\in M$ and for every geodesic $c\colon [0,d(y,z)] \to M$ with $c(0)=y$ and $c(d(y,z))=z$ and for every $0\leq t\leq 1$, the following inequality
\[
d(x,c_{t})^2\leq (1-t)d(x,y)^2+td(x,z)^2-t(1-t)d(y,z)^2
\]
holds true, where $c_{t}$ denotes $c(t d(y,z))$.
\end{definition}
See \cite[Chapter~II.1]{BookBridsonHaefliger} for more details and for different characterizations.

As we mentioned above, the goal of this subsection is to prove the following.

\begin{lemma}\label{lemma=CAT(0)}
Let $\mathcal{CAT}(0)$ denote the class of all complete $\mathrm{CAT}(0)$ spaces. Then for $\mathcal{M}=\mathcal{CAT}(0)$, we have that
\[
\mathcal{UP}(\mathcal{M})\subseteq\mathcal{M},\quad \mathcal{FS}_2(\mathcal{M})\subseteq \mathcal{M}\quad \textrm{and}\quad \ell_2(\mathcal{M})\subseteq \mathcal{M}.
\]
\end{lemma}

Since $\mathcal{F}_2(\mathcal{M})\subseteq \mathcal{FS}_2(\mathcal{M})$ in general, Lemma~\ref{lemma=CAT(0)} also implies that $\mathcal{F}_2(\mathcal{CAT}(0))\subseteq \mathcal{CAT}(0)$. A complete $\mathrm{CAT}(0)$ space is also called a \textit{Hadamard space}, but we do not use this terminology in the current paper.

The following lemma is a key to the proof of Lemma~\ref{lemma=CAT(0)}.

\begin{lemma}\label{lemma=CAT(0)UniquelyGeodesic}
Let $M\in \mathcal{CAT}(0)$. Then $M$ is uniquely geodesic.

Furthermore, for every $D>0$ and  $t\in [0,1]$ and for every $\epsilon>0$, there exists $\kappa=\kappa(D,t,\epsilon)>0$ with $\lim_{\epsilon\downarrow 0}\kappa=0$ $($for all fixed $D$ and for all fixed $t$$)$ such that the following holds true: Let $x,y\in M$ with $|d(x,y)-D|<\epsilon$. Let $z=c_t\in M$ for $c=[x,y]$. Then, for every $w\in M$ with
\[
|d(x,w)-td(x,y)|<\epsilon \quad \textrm{and} \quad |d(y,w)-(1-t)d(x,y)|<\epsilon,
\]
it holds that $d(z,w)\leq \kappa(D,t,\epsilon)$.
\end{lemma}

The latter part of the assertions of Lemma~\ref{lemma=CAT(0)UniquelyGeodesic} roughly states that, not only $M\in \mathcal{CAT}(0)$ is uniquely geodesic, but also for $x,y\in M$, all points $w\in M$ that satisfy
\[
d(x,w)\approx td(x,y) \quad \textrm{and} \quad d(y,w)\approx (1-t)d(x,y)
\]
are \textit{uniformly close} (in terms of $d(x,y)$ and $t$) to the point $z$ which divides $[x,y]$ internally in the ratio $t:(1-t)$.

\begin{proof}
We give the proof which can be generalized to the case of $r$-\textit{uniformly convex} metric spaces; see $(2)$ of Example~\ref{example=MetricSpaces}.

To prove unique geodesic property, let $x,y\in M$ and let $c^{(1)},c^{(2)}\colon [0,d(x,y)]\to M$ be two geodesic from $x$ to $y$. Fix $t\in [0,1]$. Take a geodesic $c'\colon [0,d(c^{(1)}_t,c^{(2)}_t)]\to M$ from $c^{(1)}_t$ to $c^{(2)}_t$. Apply the inequality as in Definition~\ref{definition=CAT(0)} with $t=1/2$ and $c=c'$ respectively for $(x,y,z)=(x,c^{(1)}_t,c^{(2)}_t)$ and for $(x,y,z)=(y,c^{(1)}_t,c^{(2)}_t)$. Then we have that
\begin{eqnarray*}
d(x,c'_{1/2})^2&\leq& t^2 d(x,y)^2 - \frac{1}{4} d(c^{(1)}_t,c^{(2)}_t)^2, \quad \textrm{and}\\
d(y,c'_{1/2})^2&\leq& (1-t)^2 d(x,y)^2 - \frac{1}{4} d(c^{(1)}_t,c^{(2)}_t)^2,
\end{eqnarray*}
If $d(c^{(1)}_t,c^{(2)}_t)>0$, then it would imply that
\begin{eqnarray*}
d(x,y)&\leq& d(x,c'_{1/2})+d(y,c'_{1/2})\\
&< & (t^2 d(x,y)^2)^{1/2}+((1-t)^2 d(x,y)^2)^{1/2}\\
&=& d(x,y);
\end{eqnarray*}
a contradiction. Therefore, $c^{(1)}\equiv c^{(2)}$, and we are done.

Next, we prove the latter assertion. Take a geodesic $[z,w]$ and let $u$ be the midpoint of it. Then in a similar way to one above, we have that
\begin{eqnarray*}
d(x,u)^2&\leq& \frac{1}{2}\left\{t^2 d(x,y)^2+(td(x,y)+\epsilon)^2\right\} - \frac{1}{4} d(z,w)^2, \quad \textrm{and}\\
d(y,u)^2&\leq& \frac{1}{2}\left\{(1-t)^2 d(x,y)^2+((1-t)d(x,y)+\epsilon)^2\right\} - \frac{1}{4} d(z,w)^2.
\end{eqnarray*}
Hence, we have that
\begin{eqnarray*}
d(x,y) &\leq& d(x,u)+d(y,u)\\
&\leq & \sqrt{\frac{1}{2} \left\{t^2d(x,y)^2+(td(x,y)+\epsilon)^2\right\} - \frac{1}{4} d(z,w)^2}\\
& &+\sqrt{\frac{1}{2}\left\{(1-t)^2d(x,y)^2+((1-t)d(x,y)+\epsilon)^2\right\} - \frac{1}{4} d(z,w)^2}.
\end{eqnarray*}
From the inequalities above, we may conclude the existence of $\kappa=\kappa(D,t,\epsilon)$  such that 
\[
d(z,w)<\kappa(=\kappa(D,t,\epsilon))
\] 
and that it satisfies
\[
\lim_{\epsilon\downarrow 0}\kappa(D,t,\epsilon)=0.
\]
Here our initial estimate of $\kappa$ depends on $d(x,y)$, $t$ and $\epsilon$; since $D-\epsilon\leq d(x,y)\leq D+\epsilon$, $\kappa$ may be expressed as a function on $D$, $t$ and $\epsilon$.
\end{proof}

Note that the function $\kappa=\kappa(D,t,\epsilon)$ above is \textit{universal}: It can be determined only from $\mathrm{CAT}(0)$ geometry (the inequality as in Definition~\ref{definition=CAT(0)}), and it does \textit{not} depend on the choices of the pair $(x,y)$.

\begin{proof}[Proof of Lemma~$\ref{lemma=CAT(0)}$]
Let $\mathcal{M}=\mathcal{CAT}(0)$. It is easy to see by Lemma~\ref{lemma=CAT(0)UniquelyGeodesic} that $\ell_2(\mathcal{M})\subseteq \mathcal{M}$. Indeed, every pointed $\ell_2$-product of complete and uniquely geodesic spaces is complete and uniquely geodesic. Moreover, since the inequality
\[
d(x,c_{t})^2\leq (1-t)d(x,y)^2+td(x,z)^2-t(1-t)d(y,z)^2
\]
as in Definition~\ref{definition=CAT(0)} is expressed \textit{only in terms of square sums} and since validity of it is \textit{stable under formation of rescalings}, we may confirm that this inequality remains valid for every resulting $\ell_2$-product space (possibly with rescalings). Indeed, take the square sum of inequalities which are obtained coordinatewise (recall that the resulting $\ell_2$-product space is uniquely geodesic as well).

Secondly, we will show that $\mathcal{UP}(\mathcal{M})\subseteq \mathcal{M}$. Standard arguments on metric ultraproducts show that every (pointed) metric ultraproduct of a geodesic metric space is geodesic, and a metric ultraproduct is always complete. Hence, what remains is to show the inequality as in Definition~\ref{definition=CAT(0)}. A basic philosophy to study metric ultraproducts is the following: An inequality with uniform constants \textit{on uniformly finitely many points} in metric spaces passes to metric ultraproducts. We will explain this philosophy in our example of the inequality for $\mathrm{CAT}(0)$ spaces, as in Definition~\ref{definition=CAT(0)}. Strictly speaking, this inequality is \textit{not} on uniformly finitely many points (because it involves a geodesic); however, Lemma~\ref{lemma=CAT(0)UniquelyGeodesic} enables us to reduce the inequality to two inequalities on \textit{there} (or four) \textit{points}.

Let $M=(M,d)\in \mathcal{M}$ and take a pointed metric ultraproduct $(M_{\mathcal{U}},d_{\mathcal{U}})=\lim_{\mathcal{U}}(M,y_m)$. Let $x^{(1)}_{\mathcal{U}}=[(x^{(1)}_m)_m]$ and $x^{(2)}_{\mathcal{U}}=[(x^{(2)}_m)_m]$ be two points in $M_{\mathcal{U}}$. Let $c_{\mathcal{U}}\colon [0,d_{\mathcal{U}}(x^{(1)}_{\mathcal{U}},x^{(2)}_{\mathcal{U}})]\to M_{\mathcal{U}}$ be a geodesic from $x^{(1)}_{\mathcal{U}}$ and $x^{(2)}_{\mathcal{U}}$. Let $u_{\mathcal{U}}=[(u_m)_m]$ be in $M_{\mathcal{U}}$. Let $t\in [0,1]$. Then, our goal is to prove that
\[
d_{\mathcal{U}}(u_{\mathcal{U}},c_{\mathcal{U},t})^2\leq (1-t)d_{\mathcal{U}}(u_{\mathcal{U}},x^{(1)}_{\mathcal{U}})^2+td_{\mathcal{U}}(u_{\mathcal{U}},x^{(2)}_{\mathcal{U}})^2-t(1-t)d_{\mathcal{U}}(x^{(1)}_{\mathcal{U}},x^{(2)}_{\mathcal{U}})^2,
\]
where $c_{\mathcal{U},t}=[(w_m)_m]$ is taken as $c_{\mathcal{U}}(t d_{\mathcal{U}}(x^{(1)}_{\mathcal{U}},x^{(2)}_{\mathcal{U}}))$. Let $D_{\mathcal{U}}=d_{\mathcal{U}}(x^{(1)}_{\mathcal{U}},x^{(2)}_{\mathcal{U}})$. 

We claim that  for every $\epsilon>0$, there exists $U_{\epsilon}\in \mathcal{U}$ such that the following holds: For every $m\in U_{\epsilon}$, 
\begin{eqnarray*}
& &|d(x_m^{(1)},x_m^{(2)})-D_{\mathcal{U}}|<\epsilon,\quad |d(x_m^{(1)},w_m)-td(x_m^{(1)},x_m^{(2)})|<\epsilon \\
\textrm{and} \quad& & |d(x_m^{(2)},w_m)-(1-t)d(x_m^{(1)},x_m^{(2)})|<\epsilon.
\end{eqnarray*}
Indeed, by definition of metric ultraproducts, there exists $V^{(1)}_{\epsilon}\in \mathcal{U}$ such that for every $m\in V^{(1)}_{\epsilon}$, it holds that $|d(x_m^{(1)},x_m^{(2)})-D_{\mathcal{U}}|<\epsilon$. Similarly, there exist $V^{(2)}_{\epsilon}\in \mathcal{U}$ and $V^{(3)}_{\epsilon}\in \mathcal{U}$ such that for every $m\in V^{(2)}_{\epsilon}\in \mathcal{U}$, $|d(x_m^{(1)},w_m)-tD_{\mathcal{U}}|<\epsilon$ holds and for every $m\in V^{(3)}_{\epsilon}\in \mathcal{U}$, $|d(x_m^{(2)},w_m)-(1-t)D_{\mathcal{U}}|<\epsilon$ is satisfied. Then set
\[
V_{\epsilon}=V^{(1)}_{\epsilon}\cap V^{(2)}_{\epsilon}\cap V^{(3)}_{\epsilon}.
\]
Since an ultrafilter corresponds to \textit{finitely additive} $\{0,1\}$-valued probability measures on $\mathbb{N}$, the membership of it is closed under formation of \textit{finitely many} intersections. Therefore, it follows that
\[
V_{\epsilon}\in \mathcal{U}.
\]
Finally, set
\[
U_{\epsilon}=V_{\frac{\epsilon}{2}}\quad (\in \mathcal{U});
\]
it is now easy to see that this $U_{\epsilon}$ satisfies all of the conditions of the claim above. This argument explains importance in the philosophy above to restirct ourselves to an inequality on \textit{uniformly finitely many} points in metric spaces.

Let $m\in U_{\epsilon}$. Apply Lemma~\ref{lemma=CAT(0)UniquelyGeodesic}. Then we have that $d(z_m,w_m)<\kappa(D_{\mathcal{U}},t,\epsilon)$ with 
\[
\lim_{\epsilon\downarrow 0}\kappa(D_{\mathcal{U}},t,\epsilon)=0,
\]
where $z_m$ is the point that divides $[x^{(1)}_m,x^{(2)}_m]$ internally in the ratio $t:(1-t)$. The key here is the function $\kappa$ does \textit{not} depend on the choice of $m\in U_{\epsilon}$; recall that $\kappa$ was determined only by $\mathrm{CAT}(0)$ geometry. We may apply the inequality as in Definition~\ref{definition=CAT(0)} for the triple $(z_m,x_m^{(1)},x_m^{(2)})$ (because the original space $M$ is $\mathrm{CAT}(0)$); hence we have that for every $m\in U_{\epsilon}$,
\[
d(u_m,z_m)^2\leq (1-t)d(u_m,x_m^{(1)})^2+td(u_m,x_m^{(2)})^2-t(1-t)d(x_m^{(1)},x_m^{(2)})^2.
\]
Since $d(z_m,w_m)\leq \kappa(D_{\mathcal{U}},t,\epsilon)$, we in addition have that
\[
d(u_m,w_m)\leq d(u_m,z_m)+\kappa(D_{\mathcal{U}},t,\epsilon).
\]
Finally, we let $\epsilon\downarrow 0$. Then by the two inequalities above, it follows from the definition of the metric ultraproduct that
\[
d_{\mathcal{U}}(u_{\mathcal{U}},c_{\mathcal{U},t})^2\leq (1-t)d_{\mathcal{U}}(u_{\mathcal{U}},x^{(1)}_{\mathcal{U}})^2+td_{\mathcal{U}}(u_{\mathcal{U}},x^{(2)}_{\mathcal{U}})^2-t(1-t)d_{\mathcal{U}}(x^{(1)}_{\mathcal{U}},x^{(2)}_{\mathcal{U}})^2.
\]
Therefore, we obtain our goal; it proves that $M_{\mathcal{U}}\in \mathcal{CAT}(0)$.

Once we showed $\ell_2(\mathcal{M})\subseteq \mathcal{M}$ and $\mathcal{UP}(\mathcal{M})\subseteq \mathcal{M}$, it is straightforward to verify that $\mathcal{F}_2(\mathcal{M})\subseteq \mathcal{M}$. Indeed, in general,
\[
\mathcal{F}_q(\mathcal{M})\subseteq \mathcal{UP}(\ell_q(\mathcal{M})).
\]
holds true for every $q\in [1,\infty)$.
Finally, to see that $\mathcal{FS}_2(\mathcal{M})\subseteq \mathcal{M}$, observe that all conditions for complete $\mathrm{CAT}(0)$ spaces (completeness, geodesic property and the inequality as in Definition~\ref{definition=CAT(0)}) pass to closed and geodesically convex subsets. Hence it is deduced from the inclusion $\mathcal{F}_2(\mathcal{M})\subseteq \mathcal{M}$.
\end{proof}

\begin{remark}\label{remark=FinitelyRepresentable}
If our metric space $M=E$ is a Banach space, then the philosophy as in the proof above on metric ultraproducts is stated in the following way: An inequality with uniform constants \textit{on uniformly finite dimensional subspaces} in Banach spaces passes to metric ultraproducts. 

This statement above for Banach spaces is formulated in a rigorous manner as follows: For every Banach space $E$, each metric ultrapower $E_{\mathcal{U}}=\lim_{\mathcal{U}}(E,0)$ of $E$ is \textit{finitely representable} in $E$; see the definition and discussions in \cite[Chapter~F]{BookBenyaminiLindenstrauss}. It morally means that, quantitative information of \textit{finite dimensional subspaces} of $E_{\mathcal{U}}$ may be approximated as  accurate as we hope by that of $E$.
\end{remark}

\subsection{Examples of classes of Banach spaces}\label{subsection=ClosedUnderOperations}
We discuss several examples of classes of metric spaces of our interest. They are main examples of $\mathcal{M}$ which is closed under formation of (some of) the following operations
\[
\mathcal{M}\quad \mapsto \quad \mathcal{UP}(\mathcal{M}),\ \mathcal{F}_q(\mathcal{M}),\ \mathcal{FS}_q(\mathcal{M}) \quad \textrm{and}\quad \ell_q(\mathcal{M}),
\]
for an appropriate exponent $q\in [1,\infty)$. In this subsection, we discuss certain classes of Banach space; in the next subsection, we deal with those of non-linear metric spaces.

The reader may consult the proof of Lemma~\ref{lemma=CAT(0)} for basic ideas behind the proofs of the closeness.

\begin{example}\label{example=BanachSpaces}
First we consider classes of Banach spaces.
\begin{enumerate}[$(1)$]
 \item Let $r\in [1,\infty)$. Then $\mathcal{M}=\ell_r(=\{\ell_r\})$ satisfies $\ell_q(\mathcal{M})\subseteq \mathcal{M}$ for $q=r$.
 \item More generally to $(1)$, let $\mathcal{M}=\mathcal{B}_{L_r}$ denote the class of all $L_r$-spaces (over all measure spaces). (We fix $\mathbb{R}$ or $\mathbb{C}$, and construct the class above over the fixed coefficient field.) Then, $\ell_q(\mathcal{M})\subseteq \mathcal{M}$ for $q=r$. Furthermore, Krivine showed that $\mathcal{UP}(\mathcal{M})\subseteq \mathcal{M}$; see the survey \cite{Heinrich}. It implies that $\mathcal{F}_r(\mathcal{M})\subseteq  \mathcal{M}$.

In  particular, by letting $r=2$, we observe that for $\mathcal{M}=\mathcal{H}\mathrm{ilbert}$ (the class of all Hilbert spaces), $\ell_2(\mathcal{M})\subseteq \mathcal{M}$ and $\mathcal{F}_2(\mathcal{M})\subseteq \mathcal{M}$; the proof of the latter item is much easier than that for the general $L_r$-space case. In that case, moreover, $\mathcal{FS}_2(\mathcal{M})\subseteq \mathcal{M}$ holds.
 \item More generally to $(2)$, let $\mathcal{M}=\mathcal{B}_{\mathrm{NC}L_r}$ denote the class of all non-commutative $L_r$-spaces (associated with all von-Neumann algebras). Then, $\ell_r(\mathcal{M})\subseteq \mathcal{M}$ and $\mathcal{UP}(\mathcal{M})\subseteq \mathcal{M}$; the latter follows from the work of Raynaud \cite{Raynaud}. It also holds that $\mathcal{F}_r(\mathcal{M})\subseteq \mathcal{M}$.
 \item A Banach space $E$ is said to be of \textit{non-trivial $($linear or Rademacher$)$ type} if there exists $r\in (1,2]$ and a constant $C>0$ such that the following holds true: For every $m\in \mathbb{N}_{\geq 1}$ and for every $(\xi_i)_{i\in [m]}$ in $E$, 
\[
\mathbb{E}_{(\epsilon_i)_i}\left[\|\sum_{i\in [m]} \epsilon_i \xi_i \|^r\right] \leq C^r \sum_{i\in [m]} \|\xi_i\|^r.
\]
If the inequality above is satisfied for fixed $r$ and $C$, we say that $E$ has a \textit{type $r$ with constant $C$}.
Here $\mathbb{E}_{(\epsilon_i)_i}[\cdot]$ means the expected value (average) over the uniform distribution of $(\epsilon_i)_{i\in \mathbb{N}}$ over $\{-1,1\}^m$.  Let $\mathcal{M}=\mathcal{B}_{\mathrm{type}>1}$ denote the class of all complex Banach spaces of non-trivial type. Then, $\mathcal{UP}(\mathcal{M})\subseteq \mathcal{M}$ and $\mathcal{F}_2(\mathcal{M})\subseteq \mathcal{FS}_2(\mathcal{M})\subseteq \mathcal{M}$ hold true. Indeed, having type $>1$ is stated in terms of \textit{conditions on finite dimensional subspaces}; recall Remark~\ref{remark=FinitelyRepresentable}.

Moreover, for a fixed $r\in (1,2]$ and $C>0$, if we consider the subclass $\mathcal{B}_{r,C}^{\mathrm{type}}$ of all complex Banach spaces that have type $r$ with constants $C$, then $\ell_r(\mathcal{B}_{r,C}^{\mathrm{type}})\subseteq \mathcal{B}_{r,C}^{\mathrm{type}}$. Indeed, the condition of the membership of $\mathcal{B}_{r,C}^{\mathrm{type}}$ is stated \textit{only in terms of $\ell_r$-sums}; recall the first part of the proof of Lemma~\ref{lemma=CAT(0)}. See \cite{BookTomczak-Jaegermann} for details of types of Banach spaces.

Celebrated work of V. Lafforgue \cite{Lafforgue1}, \cite{Lafforgue2} yield fixed point properties with respect to $\mathcal{B}_{\mathrm{type}>1}$; see $(1)$ of Theorem~\ref{theorem=StrongPropertyT}.
 \item N. Tomczak-Jaegermann \cite[Chapter~6]{BookTomczak-Jaegermann}, and T. de Laat and M. de la Salle \cite{deLaatdelaSalle} studied quantities $(d_k(E))_{k\in \mathbb{N}_{\geq 1}}$ and the class $\mathcal{B}_{\beta<1/2}$ (see also \cite[Formula $(1.1)$]{deLaatMimuradelaSalle}), which are defined as follows: For two  isomorphic (but not necessarily isometrically) Banach spaces $E_1$ and $E_2$, the \textit{Banach--Mazur distance} $d_{\mathrm{BM}}(E_1,E_2)$ is defined to be the infimum of $\|T\|\|T^{-1}\|$, where $T\colon E_1 \stackrel{\simeq}{\to} E_2$ runs over all isomorphisms between $E_1$ and $E_2$, and $\|\cdot\|$ denotes the operator norm. For a complex Banach space $E$, for each $k\in \mathbb{N}_{\geq 1}$, we define $d_k(E)$ by
\[
d_k(E)=\sup\left\{d_{\mathrm{BM}}(E',\ell_{2,\mathbb{C}}^{\mathrm{dim}_{\mathbb{C}} (E')}): \mathrm{dim}_{\mathbb{C}}(E')\leq k\right\},
\]
where $E'$ runs over all (complex) linear subspaces of $E$ with the condition above. Here $\ell_{2,\mathbb{C}}^m$ denotes the $m$-dimensional complex $\ell_2$-space for $m\in \mathbb{N}$. The class $\mathcal{B}_{\beta<1/2}$ is defined as the class of all complex Banach spaces $E$ for which  there exist $0<\beta<1/2$ and $C>0$ such that the condition
\[
\textrm{for all $k\in \mathbb{N}_{\geq 1}$,}\quad d_k(E)\leq Ck^{\beta} 
\]
is satisfied. Then, it follows that $\mathcal{UP}(\mathcal{B}_{\beta<1/2})\subseteq \mathcal{B}_{\beta<1/2}$ and that $\mathcal{F}_2(\mathcal{B}_{\beta<1/2})\subseteq \mathcal{FS}_2(\mathcal{B}_{\beta<1/2})\subseteq \mathcal{B}_{\beta<1/2}$. Moreover, for fixed $\beta \in (0,1/2)$ and $C>0$, if we denote by $\mathcal{B}_{\beta,C}$ the class of all complex Banach spaces such that the condition above holds for that pair $(\beta,C)$, then $\ell_2(\mathcal{B}_{\beta,C})\subseteq \mathcal{B}_{\beta,C}$ holds. The proofs of these inclusions above go along the same line as ones in $(4)$.

A fact states that a complex Banach space $E$ is of non-trivial type if and only if $\lim_{k\to \infty}k^{-1/2}d_k(E) =0$; see \cite{BookTomczak-Jaegermann}. In particular, $\mathcal{B}_{\beta<1/2} \subseteq \mathcal{B}_{\mathrm{type}>1}$. It is not known whether the inclusion above is strict.

de Laat--Mimura--de la Salle \cite{deLaatMimuradelaSalle} studied fixed point properties with respect to $\mathcal{B}_{\beta<1/2}$; see $(3)$ of Theorem~\ref{theorem=StrongPropertyT}.
\item Similar to $(4)$, for each $r\in [2,\infty)$ and each $C>0$, we define the class $\mathcal{B}_{r,C}^{\mathrm{cotype}}$ as that of all Banach spaces that satisfy the cotype $r$ inequality with constant $C$:
\[
\mathbb{E}_{(\epsilon_i)_i}\left[\|\sum_{i\in [m]} \epsilon_i \xi_i \|^r\right] \geq C^{-r} \sum_{i\in [m]} \|\xi_i\|^r.
\]
Here the expected value in the left-hand side is defined as one in $(4)$. Then for $\mathcal{M}=\mathcal{B}_{r,C}^{\mathrm{cotype}}$, in a similar way to one in $(4)$, it holds that $\ell_r(\mathcal{M})\subseteq \mathcal{M}$, $\mathcal{UP}(\mathcal{M})\subseteq \mathcal{M}$  and $\mathcal{F}_r(\mathcal{M})\subseteq \mathcal{FS}_r(\mathcal{M})\subseteq \mathcal{M}$. The union $\mathcal{B}_{\mathrm{cotype}<\infty}=\bigcup_{r,C}\mathcal{B}_{r,C}^{\mathrm{cotype}}$ equals the class of all Banach spaces of non-trivial cotype.
 \item A Banach space $E$ is said to be \textit{uniformly convex} if there exists a \textit{strictly positive} real-valued function $\Delta\colon (0,2]\to \mathbb{R}_{>0}$ such that the following holds: For every $\xi,\eta\in S(E)$ with $\xi\ne \eta$,
\[
1-\left\|\frac{\xi+\eta}{2}\right\|\geq \Delta(\|\xi-\eta\|).
\]
Here $S(E)$ denotes the unit sphere of $E$. For a fixed $r\in[2,\infty)$, if there exists $C>0$ such that $\Delta$ above satisfies that $\Delta(\epsilon)\geq C^r\epsilon^r$ for all $\epsilon\in (0,2]$, then we say that $E$ is uniformly convex with modulus of convexity of \textit{power type $r$}. Ball--Carlen--Lieb \cite[Proposition~7]{BallCarlenLieb} showed that the condition above is equivalent to saying that there exists $C'>0$ such that for all $\xi,\eta\in X$ and for all $t\in [0,1]$, the following inequality holds true:
\[
\left\|(1-t)\xi+t\eta\right\|^r \leq (1-t)\|\xi\|^r+t\|\eta\|^r-(C')^rt(1-t)\|\xi-\eta\|^r.
\]
They also made estimate between $C$ and $C'$ above. In this paper, we say a Banach space $E$ is \textit{$r$-uniformly convex with constant $C'$} if the inequality above is satisfied; this terminology is compatible with that of \textit{$r$-uniformly convex metric spaces} in a more general framework; see $(2)$ of Example~\ref{example=MetricSpaces}. 

A Banach space $E$ is said to be \textit{superreflexive} if every (equivalently, some) metric ultrapower $\lim_{\mathcal{U}}(E,0)$ is reflexive. Enflo's characterization states that $E$ is superreflexive if and only if $E$ is isomorphic to a uniformly convex Banach space. A theorem of G. Pisier \cite{PisierMartingale} shows that, moreover, for every superreflexive Banach $E$, there exists $r\in [2,\infty)$ such that $E$ is isomorphic to a uniformly convex Banach space with modulus of convexity of power type $r$. For $r\in [2,\infty)$, for $C'>0$ and for $D\geq 1$, we define the class $\mathcal{B}^{\mathrm{sr}}_{r,C',D}$ as that of all Banach spaces whose Banach--Mazur distance at most $D$ to $r$-uniformly convex Banach spaces with constant $C'$. Then, for $\mathcal{M}=\mathcal{B}^{\mathrm{sr}}_{r,C',D}$, it holds that $\ell_r(\mathcal{M})\subseteq \mathcal{M}$, $\mathcal{UP}(\mathcal{M})\subseteq \mathcal{M}$  and $\mathcal{F}_r(\mathcal{M})\subseteq \mathcal{FS}_r(\mathcal{M})\subseteq \mathcal{M}$. Indeed, without the condition of Banach--Mazur distances, they follow from a similar argument to one in $(4)$. It may be easily verified that the extra condition in terms of Banach--Mazur distances does not affect the closeness properties above.

By aforementioned theorems in \cite{BallCarlenLieb} and \cite{PisierMartingale}, the union $\mathcal{B}_{\mathrm{sr}}=\bigcup_{r,C',D}\mathcal{B}^{\mathrm{sr}}_{r,C',D}$ coincides with the class of all superreflexive Banach spaces. It is known that $\mathcal{B}_{\mathrm{sr}}\subseteq \mathcal{B}_{\mathrm{type>1}}\subseteq \mathcal{B}_{\mathrm{cotype}<\infty}$ and that both of the inclusions are strict; see \cite{BookTomczak-Jaegermann} and \cite{BookBenyaminiLindenstrauss}.
\end{enumerate}
\end{example}

We make a remark that  if $\mathcal{UP}(\mathcal{M})\subseteq \mathcal{M}$ holds, then in many cases this inclusion happens to be strict. For instance, let $\mathcal{M}=\mathcal{H}\mathrm{ilbert}$. Then as we argued in Example~\ref{example=BanachSpaces}, the inclusion above holds. It is a standard fact that a metric ultrapower of an infinite dimensional Banach space is always non-separable; see \cite{BookBenyaminiLindenstrauss}. Hence, the class $\mathcal{UP}(\mathcal{H}\mathrm{ilbert})$ does \textit{not} contain an infinite dimensional separable Hilbert space.

\subsection{Examples of classes of non-linear metric spaces}\label{subsection=ClosedUnderOperationsMetric}

In this subsection, we discuss classes of \textit{non-linear} metric spaces. Our main examples are subclasses of $\mathcal{CAT}(0)$ as in Lemma~\ref{lemma=CAT(0)}, as the class $\mathcal{CAT}(0)$ itself is too enormous. For instance, to the best knowledge of the authors, it might \textit{not} be known whether there exists an infinite RF (finitely generated) group that has the fixed point property with respect to the class $\mathcal{CAT}(0)$. In order to restrict to subclasses of $\mathcal{CAT}(0)$, we employ the following numerical invariant of a complete $\mathrm{CAT}(0)$ space.

\begin{definition}[Izeki--Nayatani invariant; \cite{IzekiNayatani}]\label{definition=IzekiNayataniInvariant}
Let $M\in \mathcal{CAT}(0)$. 
Let $\mathcal{P}_{<\aleph_0}(M)$ denote the set of all finitely supported probability measures on $M$ supported on more than one point. In other words, each $\mu\in \mathcal{P}_{<\aleph_0}(M)$ is of the form $\sum_{i=1}^k t_i\mathrm{Dirac}_{p_i}$ with $t_i>0$ for $i\in [k]$, $\sum_{i=1}^kt_i=1$ and $k\in \mathbb{N}_{\geq 2}$. Here $\mathrm{Dirac}_{p}$ means the Dirac mass at $p$. For such $\mu$, there exists a unique point $\overline{\mu}\in M$ that minimizes the function
\[
M\ni x \mapsto \sum_{i=1}^k t_i d_M(p_i,x)^2 \in \mathbb{R}_{\geq 0};
\]
this point $\overline{\mu}$ is called the \textit{barycenter} of $\mu$. For such $\mu$, define
\[
\delta(\mu)=\inf\left\{\frac{\left\|\sum_{i=1}^k t_i f(p_i)\right\|^2}{\sum_{i=1}^k t_i \|f(p_i)\|^2}: \|f(p_i)\|=d_M(p_i,\overline{\mu}),\  \|f(p_i)-f(p_j)\|\leq d_M(p_i,p_j) \right\}.
\]
Here $f$ runs over all maps from $\mathrm{supp}(\mu)$ to $L_2=L_2([0,1])$ that satisfies the two conditions indicated above, and $i$ and $j$ there vary all indices in $[k]$. 

The \textit{Izeki--Nayatani invariant} $\delta (M)$ is defined as
\[
\delta(M)=\sup_{\mu\in \mathcal{P}_{<\aleph_0}(M)}\delta(\mu).
\]
\end{definition}

The Izeki--Nayatani invariant takes values in $[0,1]$. For instance, if $M$ is a tree (or $\mathbb{R}$-tree), a Hilbert space or a (possibly infinite dimensional) \textit{Hadamard manifold} (that is, a complete, connected and  simply-connected Riemannian manifold with non-positive sectional curvature), then $\delta(M)$ is computed to be $0$; see \cite{IzekiNayatani}. As we mentioned in the introduction, many reasonable $\mathrm{CAT}(0)$ spaces $M$, such as all Euclidean buildings associated with simple algebraic groups, have $\delta(M)<1$; see \cite{ToyodaCAT(0)}. For this reason, we may regard $M\in \mathcal{CAT}(0)$ with $\delta(M)=1$ as a \textit{singular} $\mathrm{CAT}(0)$ space.

\begin{example}\label{example=MetricSpaces}
Here we discuss certain classes of \textit{non-linear} metric spaces. 
\begin{enumerate}[$(1)$]
\item Fix $\delta_0\in [0,1]$. We define a class
\[
\mathcal{CAT}(0)_{\leq \delta_0}=\{M:M \textrm{ is complete and $\mathrm{CAT}(0)$,\ $\delta(M)\leq \delta_0$}\}.
\]
Then for each $\delta_0$, the class $\mathcal{M}=\mathcal{CAT}(0)_{\leq \delta_0}$ satisfies
\[
\ell_2(\mathcal{M})\subseteq \mathcal{M},\ \mathcal{F}_2(\mathcal{M})\subseteq \mathcal{FS}_2(\mathcal{M})\subseteq \mathcal{M} \quad 
\textrm{and} \quad  \mathcal{UP}(\mathcal{M})\subseteq \mathcal{M}. 
\]
Indeed, it is known from \cite{IzekiNayatani} and \cite{Toyoda} that the Izeki--Nayatani invariant does \textit{not} increase by formation of metric ultraproducts or by taking closed and geodesically convex subsets. Also they showed that if for every $M_m\in \mathcal{CAT}(0)$, $m\in \mathbb{N}$, satisfies that $\delta(M)\leq \delta_0$, then for every choice $(y_m)_m$ of base points, the space $\left(\prod_{m\in \mathbb{N}}(M_m,y_m)\right)_{\ell_2}$ belongs to $\mathcal{CAT}(0)_{\leq \delta_0}$. Hence, the assertions above follow from Lemma~\ref{lemma=CAT(0)}.

For $\delta_0\in (0,1]$, we define the following class:
\[
\mathcal{CAT}(0)_{<\delta_0}=\bigcup_{\delta'<\delta_0}\mathcal{CAT}(0)_{\leq\delta'}.
\]
Then for $\mathcal{M}=\mathcal{CAT}(0)_{<\delta_0}$, it holds that
\[
\mathcal{F}_2(\mathcal{M})\subseteq \mathcal{FS}_2(\mathcal{M})\subseteq \mathcal{M} \quad 
\textrm{and} \quad  \mathcal{UP}(\mathcal{M})\subseteq \mathcal{M}.
\]
However, it does \textit{not} hold that $\ell_2(\mathcal{M})\subseteq \mathcal{M}$; recall the discussion below Definition~\ref{definition=ell_q}.

Izeki and Nayatani \cite{IzekiNayatani} studied fixed point properties with respect to $\mathcal{CAT}(0)_{<\delta_0}$ for certain $\delta_0$; see $(2)$ of Theorem~\ref{theorem=StrongPropertyT}.

\item Fix $r\in [2,\infty)$. Then, some $\ell_r$-analogue of item $(1)$ may be defined as follows: Let $C \in (0,1]$. A geodesic space $M$ is said to be \textit{$r$-uniformly convex with constant $C$} if for every $x\in M$ and for every geodesic $c\colon [0,d(y,z)] \to M$ with $c(0)=y$ and $c(d(y,z))=z$ and for every $0\leq t\leq 1$, the following inequality
\[
d(x,c_{t})^r\leq (1-t)d(x,y)^r+td(x,z)^r-C^rt(1-t)d(y,z)^r
\]
holds true, where $c_{t}$ denotes $c(t d(y,z))$. See also \cite{NaorSilberman}; compare with the inequality of $r$-uniformly convex Banach spaces in $(7)$ of Example~\ref{example=BanachSpaces}. The Clarkson inequality (see for instance \cite{BookBenyaminiLindenstrauss}) shows that $L_r$ is $r$-uniformly convex with a certain constant $C_r$. For fixed $C\in (0,C_r]$, we write the class of all complete $r$-uniformly convex (geodesic) spaces with constant $C$ as $\mathcal{UC}_{r,C}$. Note that $\mathcal{UC}_{2,1}=\mathcal{CAT}(0)$.

Note that we may modify the proof of Lemma~\ref{lemma=CAT(0)} to the current case. Indeed, the inequality above is stated only in terms of $\ell_r$-sums and that validity of it is stable under formation of rescalings. Furthermore, the proof of Lemma~\ref{lemma=CAT(0)UniquelyGeodesic} may be adapted to the current setting. Thus, we conclude that  for every $C\in (0,C_r]$, the class $\mathcal{M}=\mathcal{UC}_{r,C}$ satisfies that
\[
\ell_r(\mathcal{M})\subseteq \mathcal{M},\ \mathcal{F}_r(\mathcal{M})\subseteq \mathcal{FS}_r(\mathcal{M})\subseteq \mathcal{M} \quad \textrm{and} \quad \mathcal{UP}(\mathcal{M})\subseteq \mathcal{M}.
\]
In addition, it follows that every $M\in \mathcal{UC}_{r,C}$ is uniquely geodesic.

However, similar to $\mathcal{CAT}(0)$, the class $\mathcal{M}=\mathcal{UC}_{r,C}$ itself is too huge. We discuss some subclass in Example~\ref{example=r-UniformlyConvex}.
\end{enumerate}
\end{example}

In the next example, we discuss certain subclasses of $\mathcal{UC}_{r,C}$ into which fibred coarse embeddability may be reasonable to study. Before proceeding to it, we explain importance of the Izeki--Nayatani invariant of a complete $\mathrm{CAT}(0)$ space in relation to fixed point properties. Let $M\in \mathcal{CAT}(0)$. Let $\Gamma=(V(\Gamma),E(\Gamma),\mathbf{m})$ be a weighted finite connected  graph(we consider $\Gamma$ as a directed graph by considering each unoriented edge as two oriented edges). It means, $\mathbf{m}\colon E(\Gamma)\to \mathbb{R}_{>0}$, satisfies that $\mathbf{m}(v,w)=\mathbf{m}(w,v)$ for all $(v,w)\in E(\Gamma)$. For $v\in V(\Gamma)$, let $\mathbf{m}(v)=\sum_{w\in V(\Gamma): (v,w)\in E(\Gamma)}\mathbf{m}(v,w)$. For $\Gamma=(V(\Gamma),E(\Gamma),\mathbf{m})$, define the \textit{Wang-type non-linear spectral gap} with target in $M$ by
\[
\lambda_1(\Gamma,M)=\frac{1}{2}\inf_{f\colon V(\Gamma)\to M} \frac{\sum_{e=(v,w)\in E(\Gamma)}\mathbf{m}(v,w)d(f(v),f(w))^2}{\sum_{v\in V(\Gamma)}\mathbf{m}(v)d(f(v),\overline{f})^2}.
\]
Here $f$ runs over all non-constant maps $V(\Gamma)\to M$; $\overline{f}$ is the ($2$-)barycenter of $f(V(\Gamma))$ (recall Definition~\ref{definition=IzekiNayataniInvariant}). Namely, $\overline{f}$ denotes the unique point in $M$ that minimizes
\[
M\ni x \mapsto \sum_{v\in V(\Gamma)}\mathbf{m}(v)d(f(v),x)^2 \in \mathbb{R}_{\geq 0}.
\]
It is known that if $M=\mathbb{R}$, then $\lambda_1(\Gamma,M)$ equals $\lambda_1(\Gamma)$, the first strictly positive eigenvalue of the normalized Laplacian of $\Gamma$. The key property of $\delta(M)$ to the fixed point property is that for every $\Gamma$, the following inequality
\[
\lambda_1(\Gamma,M)\geq (1-\delta(M))\lambda_1(\Gamma)
\]
holds; see \cite[Proposition~6.3]{IzekiNayatani}.

Now fix $r\in [2,\infty)$. Let $M$ be a complete and $r$-uniformly convex metric space. Then, in a similar way to one above, we may define a \textit{Wang-type non-linear $r$-spectral gap} with target in $M$ for a weighted finite connected $\Gamma=(V(\Gamma),E(\Gamma),\mathbf{m})$ by
\[
\lambda^{(r)}(\Gamma,M)=\frac{1}{2}\inf_{f\colon V(\Gamma)\to M} \frac{\sum_{e=(v,w)\in E(\Gamma)}\mathbf{m}(v,w)d(f(v),f(w))^r}{\sum_{v\in V(\Gamma)}\mathbf{m}(v)d(f(v),\overline{f}^{(r)})^r}.
\]
Here $f$ runs over all non-constant maps $V(\Gamma)\to M$; $\overline{f}^{(r)}$ is the $r$-barycenter of $f(V(\Gamma))$, that means, a point that minimizes
\[
M\ni x \mapsto \sum_{v\in V(\Gamma)}\mathbf{m}(v)d(f(v),x)^r \in \mathbb{R}_{\geq 0}.
\]
By   $r$-uniform convexity and completeness of $M$, $\overline{f}^{(r)}$ uniquely exists.
In this setting, it might not be reasonable to require $\lambda^{(r)}(\Gamma,M)$ to be bounded from below by some scalar multiple of $\lambda^{(r)}(\Gamma,\mathbb{R})$. Instead, we consider a function which controls the behavior of $\lambda^{(r)}(\Gamma,M)$ for each weighted graph $\Gamma$ in terms of $\lambda^{(r)}(\Gamma,\mathbb{R})$. This formulation yields the following example of subclasses of $\mathcal{UC}_{r,C}$.

\begin{example}\label{example=r-UniformlyConvex}
Fix $r\in [2,\infty)$. Fix a non-decreasing function 
\[
\Psi\colon [0,2] \to [0,2]
\]
such that for all $t\in [0,2]$, $\Psi(t)\leq t$. For $C\in (0,C_r]$, let $\mathcal{UC}_{r,C}^{\Psi}$ be the class of all complete $r$-uniformly convex metric spaces with constant $C$ such that the following holds true: For every weighted finite connected graph $\Gamma$, it holds that
\[
\lambda^{(r)}(\Gamma,M)\geq \Psi(\lambda^{(r)}(\Gamma,\mathbb{R})).
\]

Then it may be showed that for every $\Psi$ (and for every $r$ and $C$), the subclass $\mathcal{M}=\mathcal{UC}_{r,C}^{\Psi}$ of $\mathcal{UC}_{r,C}$ satisfies that
\[
\ell_r(\mathcal{M})\subseteq \mathcal{M},\ \mathcal{F}_r(\mathcal{M})\subseteq \mathcal{FS}_r(\mathcal{M})\subseteq \mathcal{M} \quad \textrm{and} \quad \mathcal{UP}(\mathcal{M})\subseteq \mathcal{M}.
\]
Indeed, in a similar argument to one in the proof of Lemma~\ref{lemma=CAT(0)UniquelyGeodesic}, we may show some stability (with respect to approximations) of $r$-barycenters of maps. Then, for a fixed graph $\Gamma$ and for a fixed $\epsilon >0$, the following condition on $M$,
\[
\lambda^{(r)}(\Gamma,M)\geq  \epsilon
\]
can be essentially written as some inequalities with uniform constants on uniformly finitely many elements (the number is estimated in terms of $\# (V(\Gamma))$) on $M$. It then follows that $\mathcal{UP}(\mathcal{M})\subseteq \mathcal{M}$. To show that $\ell_r(\mathcal{M})\subseteq \mathcal{M}$, observe that the condition on $M$ above is stated only in terms of $\ell_r$-sums.
\end{example}

For instance, if $(r,C)=(2,1)$ and if $\Psi=\Psi_{\delta}$ is of the form 
\[
\Psi_{\delta}(t)=(1-\delta)t,
\]
for some $\delta \in [0,1]$, then it follows that
\[
\mathcal{UC}_{2,1}^{\Psi_{\delta}}\supseteq \mathcal{CAT}(0)_{\leq \delta}.
\]

As we mentioned above, non-linear spectral gaps relate to the study of fixed point properties; see Remarks~\ref{remark=IzekiNayatani} and \ref{remark=Bourdon}.

We make a remark that this construction of $\mathcal{UC}_{r,C}^{\Psi}$ is not new: It has been studied by Naor \cite{Naor} and other researchers. See also \cite{NaorSilberman}. In \cite{Naor}, Naor regarded a weighted finite graph as a symmetric stochastic matrix via the associated weighted adjacency matrix, and he considered \textit{$r$-Poincar\'{e} modulus}. Although the formulation may look different, it is essentially identical to ours.

\begin{remark}\label{remark=QuasiTree}
The main difference between $\mathcal{F}_q(\mathcal{M})$ and $\mathcal{UP}(\mathcal{M})$ is that the latter does not take (finite) $\ell_q$-products (or rescaling) before taking metric ultraproducts. Therefore, the latter procedure may preserve some ``dimension'' under certain conditions. First we consider the class  $\mathbb{R}\mathcal{T}$ of all $\mathbb{R}$-trees (namely, geodesic $0$-hyperbolic metric spaces). By the four-point condition of Gromov-hyperbolicity \cite[Chapter~III. Remark~1.21]{BookBridsonHaefliger}, it follows that $\mathcal{UP}(\mathbb{R}\mathcal{T})\subseteq \mathbb{R}\mathcal{T}$. Even if we consider a smaller class $\mathcal{T}$ of all simplicial trees (considered as geodesic spaces, possibly with uncountably many vertices), then $\mathcal{UP}(\mathcal{T})\subseteq \mathcal{T}$. This is because we may endow a metric ultraproduct with a simplicial structure by declaring vertices to be (equivalence classes of) bounded sequence of vertices; we draw edges between those with the limit distance $1$.

We consider the class $\mathcal{QT}$ of \textit{quasi-trees}, namely, graphs (considered as geodesic spaces, possibly with uncountably many vertices) that are quasi-isometric  (\cite[Chapter~I. Definition~8.14]{BookBridsonHaefliger}) to simplicial trees. By the argument above, we see that $\mathcal{UP}(\mathcal{QT})\subseteq \mathcal{QT}$; recall that we fix a single element of $\mathcal{M}$ and take pointed metric ultraproducts of it to construct $\mathcal{UP}(\mathcal{M})$.
\begin{definition}\label{definition=FiniteProducts}
Let $\mathcal{M}$ be a non-empty class of metric spaces and $q\in [1,\infty)$. Denote $(\prod_{<\aleph_0} \mathcal{M})_{\ell_q}$ by the class of all metric spaces (isometric to) \textit{finite} $\ell_q$-products $(\prod_{j\in F}M_j)_{\ell_q}$, where $1\leq \sharp(F)<\infty$ and  $M_j \in \mathcal{M}$ for all $j\in F$.
\end{definition}

Since for a fixed $m\in \mathbb{N}_{\geq 2}$, taking an $m$-fold product is compatible with taking a metric ultraproduct, we conclude that $\mathcal{UP}((\prod_{<\aleph_0} \mathcal{QT})_{\ell_1})\subseteq (\prod_{<\aleph_0} \mathcal{QT})_{\ell_1}$. (We may replace $\ell_1$ simultaneously with $\ell_q$ for each $q\in (1,\infty)$.)

Another construction is the following. Let $\mathcal{M}=M$ be a proper metric space that is cocompact. Here the properness means that all closed bounded balls are compact; $M$ is said to be \textit{cocompact} if the full isometry group of $M$ acts on $M$ cocompactly. Then, $\mathcal{UP}(M)=M$. Here the cocompactness assumption is needed in order to make control on choices of base points $(y_m)_m$ to take a pointed metric ultraproduct. 
\end{remark}


\begin{remark}\label{remark=AsymptoticDimension}
Here we make more precise on what ``dimension'' means in examples in Remark~\ref{remark=QuasiTree}. Gromov \cite{GromovAsymptotic} introduced an analogue of covering dimension in coarse geometry for a generalized metric space $M$. This concept  is called the \textit{asymptotic dimension}, written as $\mathrm{asdim}$; see \cite[Chapter~2.2]{BookNowakYu} for the definition. This is an invariant under coarse equivalence. Moreover, it is showed that if $f\colon X\to M$ is a coarse embedding, then it holds that
\[
\mathrm{asdim}(X)\leq \mathrm{asdim}(M).
\]
Also, a finite product of spaces with finite asymptotic dimension has finite asymptotic dimension. See \cite[Proposition~2.2.4, Theorem~2.2.5 and Example 2.4.1]{BookNowakYu} for details of these facts. Every tree has asymptotic dimension at most $1$ (\cite[Proposition~2.3.1]{BookNowakYu}; see also \cite[Proposition~10.2.1]{BookBuyaloSchroeder} for $\mathbb{R}$-trees). Hence,  every element in $(\prod_{<\aleph_0} \mathcal{QT})_{\ell_1}$ has finite asymptotic dimension. For $M$ a complete, connected and simply connected finite dimensional Riemannian manifold with sectional curvature  \textit{strictly negative}, results in 1.$\mathrm{E}'_1$ in \cite{GromovAsymptotic} implies the following: If such $M$ is cocompact, then $\mathrm{asdim}(M)<\infty$.
\end{remark}

\section{Idea of the proof: Metric ultraproducts and the key to non-linear version of Gromov's trick}\label{section=Idea}
We explain how metric ultraproducts play a role in the proofs of Proposition~\ref{proposition=UniformCoarseEmbedding} and $(i)$ of Theorem~\ref{mtheorem=MainTheorem}.
\subsection{Metric ultraproducts and proof of Proposition~\ref{proposition=UniformCoarseEmbedding}}\label{subsection=UniformCoarseEmbedding}
To illustrate the ideas, we first prove the following result.

\begin{lemma}\label{lemma=CoarseEmbedding}
Let $\mathcal{M}$ be a non-empty class of metric spaces. Let $(\mathbf{G}_m)_{m\in \mathbb{N}}$ be a convergent sequence in $\mathcal{G}(k)$ $($$k\in \mathbb{N}_{\geq 1}$$)$ and let $\mathbf{G}_{\infty}$ be the limit. Let $\rho,\omega\colon [0,\infty)\to [0,\infty)$ be two non-decreasing proper functions. Then, the following holds true: If $\bigsqcup_{m\in \mathbb{N}}\mathrm{Cay}(\mathbf{G}_m)$ admits a $($genuine$)$ coarse embedding into $\mathcal{M}$ with control pair $(\rho,\omega)$, then $\mathbf{G}_{\infty}$ admits a coarse embeddings into $\mathcal{UP}(\mathcal{M})$ with the same control pair $(\rho,\omega)$. 

If $\mathcal{M}$ consists only of Banach spaces, then the following holds true: If $\bigsqcup_{m\in \mathbb{N}}\mathrm{Cay}(\mathbf{G}_m)$ admits a coarse embedding into $\mathcal{M}$, then $\mathbf{G}_{\infty}$ admits a coarse embedding into the $\mathrm{original}$ $\mathrm{class}$ $\mathcal{M}$.
\end{lemma}

One key to the proof of Lemma~\ref{lemma=CoarseEmbedding} is $(\star)$ in Lemma~\ref{lemma=Neighborhood}: For each $m\in \mathbb{N}$, take $R_m \in \mathbb{N}$  as in there. Hence, $\beta_{\mathbf{G}_m,\mathbf{G}_{\infty},R_m}$ gives a complete identification between $R_m$-balls $B_{\mathrm{CayD}(\mathbf{G}_m)}(e_{G_m},R_m)$ and $B_{\mathrm{CayD}(\mathbf{G}_{\infty})}(e_{G_{\infty}},R_m)$; also, $R_m\to +\infty$ as $m\to \infty$. 
By employing this identification and by taking a pointed metric ultraproduct associated with a well-chosen sequence of base points $(y_m)_{m\in \mathbb{N}}$, we construct a coarse embedding $\mathrm{Cay}(\mathbf{G}_{\infty})\to \lim_{\mathcal{U}}(M,y_m)$ out of the original coarse embedding $\bigsqcup_{m\in \mathbb{N}}\mathrm{Cay}(\mathbf{G}_m) \to M$. The precise argument goes as follows.

\begin{proof}[Proof of Lemma~$\ref{lemma=CoarseEmbedding}$]
Let $M\in \mathcal{M}$. Suppose there exists a  coarse embedding 
\[
f\colon \bigsqcup_{m\in \mathbb{N}}\mathrm{Cay}(\mathbf{G}_m) \to M
\]
with control pair $(\rho,\omega)$. For every $m\in \mathbb{N}$, take $R_m$ as above. 

Now, for each $g\in \mathbf{G}_{\infty}$, we associate the following sequence $(y(g)_m)_{m\in \mathbb{N}}$ of points in $M$:
\[
y(g)_m=\left\{
\begin{array}{cl}
f((\beta_{\mathbf{G}_{m},\mathbf{G}_{\infty},R_m})^{-1}(g)), & \textrm{if $g\in B_{\mathrm{Cay}(\mathbf{G}_{\infty})}(e_{G_{\infty}},R_m)$,} \\
f(e_{G_{m}}), & \textrm{otherwise}.
\end{array}\right.
\]
By $(\star)$, we observe the following:
\begin{itemize}
  \item For every $g\in G_{\infty}$,
\[
\sup_{m\in \mathbb{N}} d_M(y(g)_m, f(e_{G_{m}})) \leq \omega (d_{\mathbf{G}_{\infty}}(e_{G_{\infty}},g)) (<\infty).
\]
  \item For every $g_1,g_2\in G_{\infty}$, let $m_{g_1,g_2}$ be the smallest $m$ such that for every $n\geq m$, it holds that $g_1,g_2 \in B_{\mathrm{Cay}(\mathbf{G}_{\infty})}(e_{G_{\infty}},R_{n})$. (Since $\lim_{m\to \infty}R_m=+\infty$, such $m$ exists.) Then, for all $m\geq m_{g_1,g_2}$, it holds that
\[
\rho(d_{\mathbf{G}_{\infty}}(g_1,g_2)) \leq d_M(y(g_1)_m, y(g_2)_m)\leq \omega (d_{\mathbf{G}_{\infty}}(g_1,g_2)).
\]
\end{itemize}
Finally, fix a non-principal ultrafilter $\mathcal{U}$ over $\mathbb{N}$ and  take the pointed metric ultraproduct $M_{\mathcal{U}}=\lim_{\mathcal{U}}(M,d_M,f(e_{G_{m}}))$; we define the following map
\[
f_{\infty}\colon (\mathrm{Cay}(\mathbf{G}_{\infty}),d_{\mathbf{G}_{\infty}})\to M_{\mathcal{U}};\quad g \mapsto [(y(g)_m)_{m\in \mathbb{N}}]_{\mathcal{U}}.
\]
By the two observations above, we conclude that this $f_{\infty}$ is well-defined, and that it is a coarse embedding with the same control pair $(\rho,\omega)$ as one for the original $f$.

If $M=E$ is a Banach space, then the arguments in the paper of Ostrovskii \cite{Ostrovskii} indicate a way to construct a coarse embedding from $\mathrm{Cay}(\mathbf{G}_{\infty})$ to \textit{the original} $E$ out of the metric ultraproduct construction above; this procedure will affect the control pair by some  multiplicative and additive factors.
\end{proof}

Lemma~\ref{lemma=CoarseEmbedding} can be generalized to the following proposition, which deals with the general case where the sequence of marked groups may not converge in the Cayley topology.

\begin{proposition}\label{proposition=UniformCoarseEmbedding}
Let $\mathcal{M}$ be a non-empty class of metric spaces. Let $(\mathbf{G}_m)_{m\in \mathbb{N}}$ be a sequence in $\mathcal{G}(k)$ $($$k\in \mathbb{N}_{\geq 1}$$)$. If $\bigsqcup_{m\in \mathbb{N}}\mathrm{Cay}(\mathbf{G}_m)$ admits a fibred coarse embedding into $\mathcal{M}$ as a disjoint union, then $\partial_{\mathrm{Cay}}(\mathbf{G}_m)_{m\in \mathbb{N}}$ $\mathrm{admits}$ $\mathrm{equi}$-$\mathrm{coarse}$ $\mathrm{embeddings}$ $\mathrm{into}$ $\mathcal{UP}(\mathcal{M})$; that means, 
\[
\bigcap_{\mathbf{G}_{\infty} \in \partial_{\mathrm{Cay}}(\mathbf{G}_m)_{m}} \mathcal{CP}^{\ast}_{\mathcal{UP}(\mathcal{M})}(\mathbf{G}_{\infty})\ne  \emptyset.\]

If $\mathcal{M}$ consists only of Banach spaces, then the following holds true: If $\bigsqcup_{m\in \mathbb{N}}\mathrm{Cay}(\mathbf{G}_m)$ admits a fibred coarse embedding into $\mathcal{M}$ as a disjoint union, then $\partial_{\mathrm{Cay}}(\mathbf{G}_m)_{m\in \mathbb{N}}$ admits equi-coarse embeddings into the $\mathrm{original}$ $\mathrm{class}$ $\mathcal{M}$.
\end{proposition}
See Proposition~\ref{proposition=RootedGraphUniformCoarseEmbedding} for a further generalization to disjoint unions of connected graphs with uniformly bounded degree, not necessarily those of Cayley graphs.

\begin{proof}
Let $M\in \mathcal{M}$. Suppose there exists a fibred coarse embedding from $\bigsqcup_{m\in \mathbb{N}}\mathrm{Cay}(\mathbf{G}_m)$ into $M$ with control pair $(\rho,\omega)$. Let $\mathbf{G}_{\infty}\in \partial_{\mathrm{Cay}}(\mathbf{G}_m)_{m\in \mathbf{N}}$. By definition, there exists a subsequence $(\mathbf{G}_{m_n})_n$ of  $(\mathbf{G}_m)$ that converges to $\mathbf{G}_{\infty}$ in $\mathcal{G}(k)$. By Lemma~\ref{lemma=Subsets}, there exists a fibred coarse embedding from $\bigsqcup_{n}\mathrm{Cay}(\mathbf{G}_{m_n})$ into $M$ with control pair $(\rho,\omega)$. Thus, we may assume that $(\mathbf{G}_m)_{m\in \mathbb{N}}$ itself converges to $\mathbf{G}_{\infty}$.

For every $m\in \mathbb{N}_{\geq 1}$, take $R_m$ as in $(\star)$ and take $R'_m$ as in Lemma~\ref{lemma=BallsFibredCoarseEmbedding}. Let $R_m''$ be the minimum of  $R_m$ and $R'_m$. By construction, $\lim_{m\to \infty} R_m''=+\infty$. Take the local trivialization 
\[
t_{e_{G_{m}}, R'_{m}}\colon \bigsqcup_{x\in B_{\mathrm{Cay}(\mathbf{G}_{m})}(e_{G_{m}}, R'_{m})}  M_x \to B_{\mathrm{Cay}(\mathbf{G}_{m})}(e_{G_{m}}, R'_{m}) \times M.
\]
Define a map
\[
f_{m}\colon B_{\mathrm{Cay}(\mathbf{G}_{m})}(e_{G_{m}}, R'_{m}) \to M;\quad x\mapsto t_{e_{G_{m}}, R'_{m}} (x)( s(x)) .
\]
By $(2)$ in Lemma~\ref{lemma=BallsFibredCoarseEmbedding}, this $f_{m}$ is a coarse embedding with compression pair $(\rho,\omega)$. 

Then, we may modify the construction of $(y(g)_m)$ as in the proof of Lemma~\ref{lemma=CoarseEmbedding} by setting for every $m\in \mathbb{N}$, 
\[
y(g)_m=\left\{
\begin{array}{cl}
f_m((\beta_{\mathbf{G}_{m},\mathbf{G}_{\infty},R''_m})^{-1}(g)), & \textrm{if $g\in B_{\mathrm{Cay}(\mathbf{G}_{\infty})}(e_{G_{\infty}},R''_m)$,} \\
f_m(e_{G_{m}}), & \textrm{otherwise}.
\end{array}\right.
\]
Then it will complete our proof of Proposition~\ref{proposition=UniformCoarseEmbedding}.
\end{proof}

\begin{remark}\label{remark=RootedGraphs}
To prove these lemma and proposition, we do not use the property that $\beta_{\mathbf{G}_m,\mathbf{G}_{\infty},}$ is an isomorphism as rooted \textit{diagrams}; what we needed above is this map is an isomorphism as rooted (non-labelled, non-oriented) \textit{graphs}. From this point of view, we consider \textit{the space of rooted graphs with bounded degree} and generalize Proposition~\ref{proposition=UniformCoarseEmbedding} in the following manner; see Proposition~\ref{proposition=RootedGraphUniformCoarseEmbedding} for the conclusion.

Fix $k\in \mathbb{N}_{\geq 2}$, We set $\mathcal{R}(k)$ as the space of all connected graphs (without labellings/orientations) $(\Gamma,r_{\Gamma})$ with roots $r_{\Gamma}(\in V(\Gamma))$ such that the degrees of all vertices are at most $k$. We say $\phi\colon (\Gamma_1,r_{\Gamma_1})\stackrel{\simeq}{\to} (\Gamma_2,r_{\Gamma_2})$ is an isomorphism as \textit{rooted graphs} if $\phi(r_{\Gamma_1})=r_{\Gamma_2}$ and if $\phi$ is a graph isomorphism. In $\mathcal{R}(k)$, we identify two rooted graphs that are isomorphic in the sense above. We endow $\mathcal{R}(k)$ with the topology of \textit{local convergence as rooted graphs}. This means, $((\Gamma_m,r_{\Gamma_m}))_{m\in \mathbb{N}}$ converges to $(\Gamma_{\infty},r_{\Gamma_{\infty}})$ if for every $R\in \mathbb{N}_{\geq 1}$, there exists $m_R\in \mathbb{N}$ such that for every $m\geq m_R$, the $R$-balls $B_{\Gamma_m}(r_{\Gamma_m},R)$ and $B_{\Gamma_{\infty}}(r_{\Gamma_{\infty}},R)$, \textit{centered at roots}, are isomorphic as rooted graphs. The space $\mathcal{R}(k)$, equipped with this topology, is a \textit{compact} metrizable space. 

Consider a sequence $(\Gamma_m)_{m\in \mathbb{N}}$ of connected graphs with all degrees at most $k$. Then, each $\Gamma_m$ forms a (possibly, non-singleton) subset $\widetilde{\Gamma_m}=\{(\Gamma_m,v):v\in V(\Gamma_m)\}$ of $\mathcal{R}(k)$; we define the \textit{rooted graph boundary} of $(\Gamma_m)_m$ by the set of all possible accumulation points of $\bigcup_{m\in \mathbb{N}}\widetilde{\Gamma_m}$ in $\mathcal{R}(k)$ as $m\to  \infty$. We write it as $\partial_r (\Gamma_m)_{m\in \mathbb{N}}$.

\begin{proposition}\label{proposition=RootedGraphUniformCoarseEmbedding}
Let $k\in \mathbb{N}_{\geq 2}$. Let $\mathcal{M}$ be a non-empty class of metric spaces. Let $(\Gamma_m)_{m\in \mathbb{N}}$ be a sequence of connected graphs with all degrees at most $k$. If $\bigsqcup_{m\in \mathbb{N}}\Gamma_m$ admits a fibred coarse embedding into $\mathcal{M}$ as a disjoint union, then the rooted graph boundary $\partial_{r}(\Gamma_m)_{m\in \mathbb{N}}$ admits equi-coarse embeddings into $\mathcal{UP}(\mathcal{M})$; that means, 
\[
\bigcap_{(\Gamma_{\infty},r_{\infty}) \in \partial_{r}(\Gamma_m)_m} \mathcal{CP}_{\mathcal{UP}(\mathcal{M})}(\Gamma_{\infty})\ne  \emptyset.\]

If $\mathcal{M}$ consists only of Banach spaces, then the following holds true: If $\bigsqcup_{m\in \mathbb{N}}\Gamma_m$ admits a fibred coarse embedding into $\mathcal{M}$ as a disjoint union, then $\partial_{r}(\Gamma_m)_{m\in \mathbb{N}}$ admits equi-coarse embeddings into $\mathcal{M}$.
\end{proposition}
\end{remark}

\subsection{Metric ultraproducts of fragmentary actions}\label{subsection=PartialActions}
In Subsection~\ref{subsection=UniformCoarseEmbedding}, we saw how to recover (non-equivariant) coarse embeddings from Cayley limit groups out of a (fibred) coarse embeddings of the disjoint union. 

In this subsection, we discuss some recovery procedure of \textit{equivariant} coarse embeddings. One important point here is that for this, what we need is \textit{not} the \textit{global} actions of the whole groups $G_m$, but \textit{local} actions of balls; compare with the proof of Lemma~\ref{lemma=CoarseEmbedding}. Here we give the definition of a \textit{fragmentary action} of a subset of a group, which is a local version of the action of the whole group.

\begin{definition}\label{definition=PartialAction}
Let $M$ be a metric space. Let $G$ be a group and $e_G\in K \subseteq G$ be a subset. A partial homomorphism from $K$ to the isometry group $\mathrm{Isom}(M)$ is called a \textit{fragmentary action} of $K$ on $M$. In other words, a right fragmentary action $\alpha \colon M\curvearrowleft K$ (where for all $g\in K$, $\alpha(g)$ is an isometry on $M$) satisfies the following property: For every $g_1,g_2\in K$ \textit{such that} $g_1g_2\in K$, 
\[
z\cdot \alpha(g_1g_2)=(z \cdot \alpha(g_1))\cdot \alpha(g_2)
\]
for all $z\in M$. 
\end{definition}
We use the word ``fragmentary'' because the terminology ``partial action'' is referred to a quite different concept in the literature.

\begin{proposition}\label{proposition=FromLocalToGlobal}
Let $\mathbf{G}_m\stackrel{\mathrm{Cay}}{\to} \mathbf{G}_{\infty}$. Let $\rho,\omega\colon [0,\infty)\to [0,\infty)$ be two non-decreasing proper functions. Assume that for every $m\in \mathbb{N}$, there exists $r_m \in \mathbb{N}_{\geq 1}$ with $\lim_{m\to \infty}r_m=+\infty$ such that the following holds: For every $m\in \mathbb{N}$, there exists an pointed isometric right $\mathrm{fragmentary}$ $\mathrm{action}$ $(\alpha_m,M_m,y_m)$ of $B_{\mathbf{G}_m}(e_{G_m},r_m)$ 
\[
\alpha\colon M_m \curvearrowleft B_{\mathbf{G}_m}(e_{G_m},r_m),\quad y_m \in M_m
\]
such that the orbit map of $y_m$ is an $($equivariant$)$ coarse embedding of $(B_{\mathbf{G}_m}(e_{G_m},r_m),d_{\mathbf{G}_m})$ with $($equivariant$)$ control pair $(\rho,\omega)$. 

Then, for every non-principal ultrafilter $\mathcal{U}$ over $\mathbb{N}$, there exists a pointed isometric right action $(\alpha_{\mathcal{U}},M_{\mathcal{U}},y_{\mathcal{U}})$ of $G_{\infty}$ such that the orbit map of $y_{\mathcal{U}}$ is an $($equivariant$)$ coarse embedding of $(\mathbf{G}_{\infty},d_{\mathbf{G}_{\infty}})$ with equivariant control pair $(\rho,\omega)$. Here $M_{\mathcal{U}}=\lim_{\mathcal{U}}(M_m,y_m)$ and $y_{\mathcal{U}}=[(y_m)_m]_{\mathcal{U}}$.
\end{proposition}

Compare the statement of Proposition~\ref{proposition=FromLocalToGlobal} with the standard argument, for instance, in a survey of Y. Stalder \cite[Theorem 3.12]{Stalder}.
\begin{proof}
For every $m\in \mathbb{N}$, take $R_m \in \mathbb{N}$ and $\beta_{\mathbf{G}_{m},\mathbf{G}_{\infty},R_m}$ as in $(\star)$. Set $R''_m=\min\{R_m,r_m\}$. For each $g\in G_{\infty}$, define $ \alpha'_m(g)\colon M_m \to M_m; \ z\mapsto z \cdot \alpha'_m(g)$ by 
\[
z \cdot \alpha'_m(g)=\left\{
\begin{array}{cl}
z \cdot \alpha_m((\beta_{\mathbf{G}_{m},\mathbf{G}_{\infty},R_m})^{-1}(g)), & \textrm{if $g\in B_{\mathbf{G}_{\infty}}(e_{G_{\infty}},R''_m)$,} \\
z, & \textrm{otherwise}.
\end{array}\right.
\]
By construction, the restriction of $\alpha'_m$ on $B_{\mathbf{G}_{\infty}}(e_{G_{\infty}},R''_m)$ is a fragmentary action.

Finally, for every $g\in G_{\infty}$, define $\alpha_{\mathcal{U}}(g)\colon M_{\mathcal{U}} \to M_{\mathcal{U}}$ by 
\[
[(z_m)_m]_{\mathcal{U}} \cdot \alpha_{\mathcal{U}}(g) =[(z_m \cdot \alpha'_m(g))_m]_{\mathcal{U}}\quad \textrm{for every $[(z_m)_m]_{\mathcal{U}} \in M_{\mathcal{U}}$.}
\]
It is straightforward to check that this is well-defined. Since $\lim_{m\to \infty}R''_m=+\infty$, this $\alpha_{\mathcal{U}}$ is a (\textit{global}) action of $G_{\infty}$ on $M_{\mathcal{U}}$ (by isometries). By assumption, it furthermore holds that for every $g_1,g_2\in G_{\infty}$,
\[
\rho(d_{\mathbf{G}_{\infty}}(g_1,g_2))\leq d_{M_{\mathcal{U}}}(y_{\mathcal{U}}\cdot \alpha_{\mathcal{U}}(g_1),y_{\mathcal{U}}\cdot \alpha_{\mathcal{U}}(g_2))\leq \omega(d_{\mathbf{G}_{\infty}}(g_1,g_2)),
\]
as desired; compare with the proofs of Lemma~\ref{lemma=CoarseEmbedding} and Proposition~\ref{proposition=UniformCoarseEmbedding}.
\end{proof}

\subsection{Key to the non-linear version of Gromov's trick}\label{subsection=QuasiAction}
Proposition~\ref{proposition=FromLocalToGlobal} will be used for the proof of $(i)$.$(1)$ of Theorem~\ref{mtheorem=MainTheorem}. To deal with $(i)$.$(2)$, we employ the following definition.

\begin{definition}\label{definition=AlmostPartialAction}
Let $G$ be a group and $e_G\in K\subseteq G$ be a subset. Let $M$ be a metric space and $y\in M$. Let $\epsilon\geq 0$. We say that a map $\alpha\colon K\to \mathrm{Isom}(M)$ is an \textit{$\epsilon$-almost fragmentary} (right) \textit{action at $y$} if the following condition is fulfilled: For every $g_1,g_2\in K$ such that $g_1g_2\in K$,
\[
d(y\cdot \alpha(g_1g_2), (y\cdot \alpha(g_1))\cdot \alpha(g_2))\leq \epsilon.
\]
\end{definition}
If $K=G$ and $\alpha$ is a $0$-almost fragmentary at $y$, then $\alpha\colon G\to \mathrm{Isom}(M)$ gives rise to a genuine action \textit{on the $G$-orbit $\{y\cdot \alpha(g):g\in G\}$ of $y$}.

\begin{proposition}\label{proposition=FromLocalToGlobalSubspace}
Let $\mathbf{G}_m\stackrel{\mathrm{Cay}}{\to} \mathbf{G}_{\infty}$. Let $\rho,\omega\colon [0,\infty)\to [0,\infty)$ be two non-decreasing proper functions. Assume that for every $m\in \mathbb{N}$, there exist $r_m \in \mathbb{N}_{\geq 1}$ with $\lim_{m\to \infty}r_m=+\infty$, a sequence $(\rho_m,\omega_m)_m$ of two non-decreasing proper functions $[0,\infty)\to [0,\infty)$, $\epsilon_m\geq 0$ with $\lim_{m\to \infty} \epsilon_m=0$ and $y_m\in M_m$ such that the following conditions hold: 
\begin{itemize}
  \item For every $m\in \mathbb{N}$, there exists an $\epsilon_m$-$\mathrm{almost}$ $($right$)$ $\mathrm{fragmentary}$ $\mathrm{action}$ $\alpha_m$ $\mathrm{at}$ $y_m$ of $B_{\mathbf{G}_m}(e_{G_m},r_m)$ $($by isometries$)$ on $M_m$.
  \item Two sequences $(\rho_m)_m$ and $(\omega_m)_m$, respectively, converge to $\rho$ and $\omega$ pointwise.
  \item For every $m\in \mathbb{N}$ and for every $g_1,g_2\in B_{\mathbf{G}_m}(e_{G_m},r_m)$, it holds that
\[
\rho_m(d_{\mathbf{G}_m}(g_1,g_2)) \leq d_{M_m}(y_m\cdot \alpha_m(g_1), y_m\cdot \alpha_m(g_2)) \leq \omega_m(d_{\mathbf{G}_m}(g_1,g_2)).
\]
\end{itemize}
Assume that there exists a non-principal ultrafilter $\mathcal{U}$ over $\mathbb{N}$ such that $M_{\mathcal{U}}=\lim_{\mathcal{U}}(M_m,y_m)$ is $\mathrm{uniquely}$ geodesic.

Then, for every such $\mathcal{U}$ over $\mathbb{N}$, there exist a closed and $\mathrm{geodesically}$ $\mathrm{convex}$ subset $L_0$ of $M_{\mathcal{U}}$ and an isometric right $($genuine$) $action $(\alpha_{\mathcal{U}},L_0)$ of $G_{\infty}$ that satisfy all of the following conditions: 
\begin{itemize}
  \item For $y_{\mathcal{U}}=[(y_m)_m]_{\mathcal{U}}$, it holds that $\{y_{\mathcal{U}}\cdot\alpha_{\mathcal{U}}(g):g\in G_{\infty}\}\subseteq L_0$.
  \item The orbit map of $y_{\mathcal{U}}$ by $\alpha_{\mathcal{U}}$ is an $($equivariant$)$ coarse embedding of $(\mathbf{G}_{\infty},d_{\mathbf{G}_{\infty}})$ $($into $L_0$$)$ with equivariant control pair $(\rho,\omega)$. 
\end{itemize}
Here we equip $L_0$ with the induced metric from that of $M_{\mathcal{U}}$.
\end{proposition}

\begin{proof}
For each $g\in G_{\infty}$, the construction of $\alpha_{\mathcal{U}}(g)\colon M_{\mathcal{U}}\to M_{\mathcal{U}}$ is exactly the same as one in the proof of Proposition~\ref{proposition=FromLocalToGlobal}. Indeed, since each $\alpha_m(h)$, for $h\in B_{\mathbf{G}_m}(e_{G_m},r_m)$, is isometric, $\alpha_m$ is $\epsilon$-almost fragmentary action at $y$ and the ``orbit map'' of $y_m$ by $\alpha_m$ is a coarse embedding with control pair $(\rho_m,\omega_m)$, it follows that for each $g\in G_{\infty}$,
\[
\sup_{m\in \mathbb{N}}d_{M_m}( z_m \cdot \alpha'_m(g),y_m) <\infty \quad\textrm{for every $(z_m)_m \in \left(\prod_m (M_m,y_m)\right)_{\ell_{\infty}}$;} 
\]
recall that $\rho_m$ and $\omega_m$ are non-decreasing.
The construction of $\alpha_{\mathcal{U}}(g)$ above is well-defined, and $\alpha_{\mathcal{U}}(g)$ is an isometry. We, however, warn that in general, $\alpha_{\mathcal{U}}(gh)$ may \textit{not} coincide with $\alpha_{\mathcal{U}}(g)\circ \alpha_{\mathcal{U}}(h)$ (the composition is from left to right) as maps $M_{\mathcal{U}}\to M_{\mathcal{U}}$.

Nevertheless, we observe that $\alpha_{\mathcal{U}} \colon G_{\infty}\to \mathrm{Isom}(M_{\mathcal{U}})$ is \textit{$0$-almost} fragmentary action \textit{at $y_{\mathcal{U}}$} because $\lim_{m\to \infty}\epsilon_m=0$. Therefore, it is a \textit{genuine} action on $L'=\{y_{\mathcal{U}}\cdot \alpha_{\mathcal{U}}(g):g\in G_{\infty}\}$. For every $g_1,g_2\in G_{\infty}$, define 
\[
L_{g_1,g_2}=\{z \in M_{\mathcal{U}}: z \cdot (\alpha(g_1)\circ \alpha(g_2)\circ \alpha(g_1g_2)^{-1})=z\}.
\]
Because $\alpha(g_1)\circ \alpha(g_2)\circ \alpha(g_1g_2)^{-1}$ is an isometry and we assume that $M_{\mathcal{U}}$ is \textit{uniquely} geodesic, each $L_{g_1,g_2}$ is a closed and \textit{geodecially  convex} subset of $M_{\mathcal{U}}$ with $L'\subseteq L_{g_1,g_2}$. (Observe that every isometry sends geodesics to geodesics.) Finally, take
\[
L_0=\bigcap_{g_1,g_2\in G_{\infty}}L_{g_1,g_2} \quad (\supseteq L').
\]
Then $L_0=L_0\cdot \alpha_{\mathcal{U}}(G_{\infty})$ holds, and $\alpha_{\mathcal{U}}$ gives rise to a \textit{genuine} action on $L_0$. We rewrite the restriction of $\alpha_{\mathcal{U}}$ on $L_0$ as $\alpha_{\mathcal{U}}\colon L_0\curvearrowleft G_{\infty}$; it satisfies the required two conditions.
\end{proof}

\begin{remark}\label{remark=AffineGeodesic}
We may remove the assumption of the unique geodesic property on $M_{\mathcal{U}}=E_{\mathcal{U}}$ if $\mathcal{M}=\mathcal{E}$ consists only of Banach spaces. Indeed, if we assume that all $\alpha_m$ are complex affine, then take $L_0$ as the closure of the algebraic complex affine span of $L'$; this $L_0$ is a non-empty complex affine subspace of $E_{\mathcal{U}}$. Even if we do not assume it, the Mazur--Ulam theorem states that all $\alpha_m$ are real affine. Then we can take a desired $L_0$ as a non-empty real affine subspace of $E_{\mathcal{U}}$. 
\end{remark}

\section{From fibred coarse embeddings to equivariant embeddings of groups in the Cayley boundary}\label{section=TheoremA(i)}

In this section, we prove item $(i)$ of Theorem~\ref{mtheorem=MainTheorem}. As mentioned in the introduction, our idea of the proof(s) is based on a trick of Gromov. We first demonstrate the proof of $(i)$.$(1)$ in Subsection~\ref{subsection=Finite}. Then we proceed to the proof of $(i)$.$(2)$ in Subsection~\ref{subsection=Amenable}.

\subsection{Proof for finite marked group sequences}\label{subsection=Finite}
We already know from Proposition~\ref{proposition=UniformCoarseEmbedding} the way to recover (non-equivariant) coarse embeddings of groups in the Cayley  boundary from local information from the fibred coarse embedding. The point in our proof is how to  recover moreover \textit{equivariant} coarse embeddings. The key tool here is Proposition~\ref{proposition=FromLocalToGlobal}.

\begin{proof}[Proof of $(i)$.$(1)$ of Theorem~$\ref{mtheorem=MainTheorem}$]

Similar to the proof of Proposition~\ref{proposition=UniformCoarseEmbedding}, we may assume that $(\mathbf{G}_m)_{m\in \mathbb{N}}$ is a convergent sequence. Let $\mathbf{G}_{\infty}$ be the Cayley limit of it. Assume that $\bigsqcup_{m\in \mathbb{N}}\mathrm{Cay}(\mathbf{G}_m)$ admits a fibred coarse embedding into $M$, $M\in \mathcal{M}$, with control pair $(\rho,\omega)$. Take as $s$, $R'_m$, $t_{g,R'_m}$, $t_{g_1,g_2,R_m'}$ as in Lemma~\ref{lemma=BallsFibredCoarseEmbedding}. Fix $q\in [1,\infty)$.

For each $m\in \mathbb{N}$, set 
\[
M_{m,q}=\left(\prod_{x\in G_m}(M,\left(\frac{1}{\sharp(G_m)}\right)^{1/q})\right)_{\ell_q}; 
\]
recall Definition~\ref{definition=ell_qSum}. 

For every $g\in B_{\mathbf{G}_{m}}(e_{G_m},R'_m)$, define $\alpha_m(g)\colon M_{m,q}\to M_{m,q}$ by
\[
(z_x)_{x\in G_m} \cdot \alpha_m(g)=(t_{x,gx,R'_m}(z_{gx}))_{x\in G_m}, \quad \textrm{for $(z_x)_x \in M_{m,q}$.}
\]
We claim the following.

\begin{lemma}\label{lemma=ClaimFinite}
\begin{enumerate}[$(1)$]
 \item This $\alpha_m$ is a fragmentary action $($by isometries$)$
\[
M_{m,q} \curvearrowleft B_{\mathbf{G}_m}(e_{G_m},R'_m).
\]
 \item Let $y_m=(y_{m,x})_{x\in G_m}\in M_{m,q}$, where $y_{m,x}=t_{x,R'_m}(x)(s(x))$ for every $x\in G_m$. Then the orbit map of $y_m$ by the fragmentary action $\alpha_m$ is an $($equivariant$)$ coarse embedding from $B_{\mathbf{G}_m}(e_{G_m},R'_m)$ into $M_{m,q}$ with control pair $(\rho,\omega)$.
\end{enumerate}
\end{lemma}

\begin{proof}[Proof of Lemma~$\ref{lemma=ClaimFinite}$]
Since all $t_{x,gx,R'_m}$ are isometries, $\alpha_m(g)$ is an isometry for all $g\in G_m$. Assume that $g_1,g_2,g_1g_2\in B_{\mathbf{G}_m}(e_{G_m},R'_m)$. Then since for each $x\in G_m$, $g_1g_2x \in B_{\mathbf{G}_m}(g_1g_2x,R'_m)\cap B_{\mathbf{G}_m}(g_2x,R'_m)\cap B_{\mathbf{G}_m}(x,R'_m)$, it holds that
\begin{eqnarray*}
&& t_{x, g_2 x,R'_m} \circ t_{g_2 x, g_1 g_2 x, R'_m} \\
&=& t_{x,R'_m}(g_1 g_2 x) \circ (t_{g_2 x,R'_m}(g_1 g_2 x))^{-1} \circ t_{g_2 x,R'_m}(g_1g_2 x) \circ (t_{g_1g_2 x,R'_m}(g_1g_2 x))^{-1}\\
&=& t_{x,R'_m}(g_1g_2 x) \circ (t_{g_1g_2 x,R'_m}(g_1 g_2 x))^{-1}\\
&=& t_{x, g_1 g_2 x,R'_m}.
\end{eqnarray*}
Therefore, we have that by setting $w_x=t_{x,g_1x,R'_m}(z_{g_1x})$,
\begin{eqnarray*} 
((z_x)_x \cdot \alpha(g_1)) \cdot \alpha(g_2)&=&
(t_{x, g_2 x,R'_m} (w_{g_2x}))_x \\
&=&
(t_{x, g_2 x,R'_m} \circ t_{g_2 x, g_1 g_2 x,R'_m}(z_{g_1g_2 x}))_x\\
&=&
(t_{x, g_1 g_2 x,R'_m}(z_{g_1 g_2 x}))_x\\
&=& (z_x)_x\cdot \alpha(g_1 g_2).
\end{eqnarray*}
This proves $(1)$. 

For $(2)$, observe that for every $g\in B_{\mathbf{G}_m}(x,R'_m)$ and every $x\in G_m$,
\[
d_M(y_{m,x},t_{x,gx,R'_m}(y_{m,gx}))=d_M(t_{x,R'_m}(x)(s(x)), t_{x,R'_m}(gx)(s(gx)).
\]
By assumption and by recalling Remark~\ref{remark=ScalingFactor}, we verify $(2)$.
\end{proof}

By applying Proposition~\ref{proposition=FromLocalToGlobal} with $r_m=R'_m$, we obtain from Lemma~\ref{lemma=ClaimFinite} an \textit{equivariant} coarse embedding from $\mathbf{G}_{\infty}$ into $\lim_{\mathcal{U}}(M_{m,q},y_m)$ with equivariant control pair $(\rho,\omega)$. Since $\lim_{\mathcal{U}}(M_{m,q},y_m) \in \mathcal{F}_q(\mathcal{M})$, it proves the desired assertions.
\end{proof}

\subsection{Non-linear version of Gromov's trick and proof for amenable group sequences}\label{subsection=Amenable}
In order to extend the argument as in Subsection~\ref{subsection=Finite} to the  case of amenable marked group sequences, we employ a \textit{F{\o}lner set} of $G_m$ instead of $G_m$ itself and utilize Proposition~\ref{proposition=FromLocalToGlobalSubspace}. This idea dates back to Gromov, and well known if $\mathcal{M}=\mathcal{H}\mathrm{ilbert}$. We extend this framework to possibly non-linear settings.

For $\epsilon>0$ and for $R\in \mathbb{N}$, an \textit{$(\epsilon,R)$-F{\o}lner set} $F$ of a marked group $\mathbf{G}$ is a non-empty \textit{finite} subset of $G$ that satisfies 
\[
\frac{\sharp(\partial_{\mathbf{G}}(F,R))}{\sharp (F)}<\epsilon.
\]
Amenability of $G$ is characterized by the existence of $(\epsilon,R)$-F{\o}lner sets for all $\epsilon(>0)$ and for all $R$ (this property does not depend on the choices of markings of $G$). 

\begin{proof}[Proof of $(i)$.$(2)$ of Theorem~$\ref{mtheorem=MainTheorem}$]
We describe the modifications needed from the proof $(i)$.$(1)$ of Theorem~\ref{mtheorem=MainTheorem}. Fix $q\in [1,\infty)$.
For each $m\in \mathbb{N}$, choose $\delta_m>0$ appropriately (we will specify later) and take an $(\delta_m,R'_m)$-F{\o}lner set $F^{(m)}$ of $\mathbf{G}_m$. Set
\[
M_{m,q}=\left(\prod_{x\in F^{(m)}}(M,\left(\frac{1}{\sharp(F^{(m)})}\right)^{1/q})\right)_{\ell_q}.
\]
For every $g\in B_{\mathbf{G}_m}(e_{G_m},R'_m)$, let $\alpha_m(g)\colon M_{m,q}\to M_{m,q}$ be $(z_x)_{x\in F^{(m)}} \cdot \alpha_m(g)=(w_x)_x$ for $(z_x)_x \in M_{m,q}$. Here $w_x \in M$ is defined by
\[
w_x=
\left\{\begin{array}{cl}
t_{x,gx,R'_m}(z_{gx}), & \quad \textrm{if $gx\in F^{(m)}$,} \\
z_{x}, & \quad \textrm{otherwise.} 
\end{array}\right.
\]
We claim the following.

\begin{lemma}\label{lemma=ClaimAmenable}
Let $y_m=(y_{m,x})_{x\in G_m}\in M_{m,q}$, where $y_{m,x}=t_{x,R'_m}(x)(s(x))$ for every $x\in G_m$.  Let $\delta'_{m,q}=\delta_m^{1/q}\omega(R_m')$.
\begin{enumerate}[$(1)$]
 \item For every $g\in B_{\mathbf{G}_m}(e_{G_m},R'_m)$, $\alpha_m(g)$ is an isometry. 
  \item This $\alpha_m$ is a $3\delta'_{m,q}$-almost fragmentary action of $B_{\mathbf{G}_m}(e_{G_m},\lfloor R'_m/2\rfloor)$ at $y_m$; recall Definition~$\ref{definition=AlmostPartialAction}$.
 \item  For every $g_1,g_2\in B_{\mathbf{G}_m}(e_{G_m},\lfloor R'_m/2\rfloor)$,  it holds that
\[
(1-2\delta_m)\rho(d_{\mathbf{G}_m}(g_1,g_2))\leq d_{M_m}(y_m\cdot \alpha_m(g_1), y_m\cdot \alpha_m(g_2)) \leq \omega(d_{\mathbf{G}_m}(g_1,g_2))+2\delta'_{m,q}.
\]
\end{enumerate}
\end{lemma}
\begin{proof}[Proof of Lemma~$\ref{lemma=ClaimAmenable}$]
Item $(1)$ is by construction. For $(2)$, let $g_1,g_2\in B_{\mathbf{G}_m}(e_{G_m},\lfloor R'_m/2\rfloor)$ such that $g_1g_2\in B_{\mathbf{G}_m}(e_{G_m},\lfloor R'_m/2\rfloor)$. Let $F_{\mathrm{good}}^{(m)}=F^{(m)}\cap (g_1^{-1}F^{(m)})\cap (g_2^{-1}F^{(m)})\cap ((g_1g_2)^{-1}F^{(m)})$ and $F_{\mathrm{bad}}^{(m)}=F^{(m)}\setminus F_{\mathrm{good}}^{(m)}$. Then, by the F{\o}lner property for $F$, $\sharp(F_{\mathrm{bad}}^{(m)})\leq 3\delta_m \sharp (F^{(m)})$. Note that by the proof of Lemma~\ref{lemma=ClaimFinite}, for all $x\in F_{\mathrm{good}}^{(m)}$,
\[
((y_m \cdot \alpha(g_1))\cdot \alpha(g_2))(x)=(y_m \cdot \alpha(g_1g_2))(x),
\]
where $(\cdot)(x)$ indicates the $x$-th coordinate. 

Now let $x\in F_{\mathrm{bad}}^{(m)}$. Then, similar to one above, we have that 
\begin{eqnarray*}
& &d_M((y_m \cdot \alpha(g_1))\cdot \alpha(g_2))(x), (y_m \cdot \alpha(g_1g_2))(x))\\
&\leq& \max\{\omega(d_{\mathbf{G}_m}(\gamma,\gamma')):\gamma,\gamma'\in \{e_{G_m},g_1,g_2,g_1g_2\}\} \\
&=& \omega\left(\max\{d_{\mathbf{G}_m}(\gamma,\gamma'):\gamma,\gamma'\in \{e_{G_m},g_1,g_2,g_1g_2\}\}\right) \\
&\leq& \omega(R'_m).
\end{eqnarray*}
By recalling that we take the scaling factor $(1/\sharp(F^{(m)}))^{1/q}$ to construct $M_{m,q}$ from $M$, we conclude that
\begin{eqnarray*}
d_{M_{m,q}}((y_m \cdot \alpha(g_1))\cdot \alpha(g_2),y_m \cdot \alpha(g_1g_2)))
&\leq& \left(3\delta_m \sharp(F^{(m)})\cdot \frac{\omega(R'_m)^q}{\sharp(F^{(m)})}\right)^{\frac{1}{q}}\\
&\leq& 3\delta_m^{1/q}\omega(R'_m),
\end{eqnarray*}
as desired. Item $(3)$ will be proved in a manner quite similar to one above.
\end{proof}

For given $q\in [1,\infty)$, $\omega\colon [0,\infty)\to [0,\infty)$ and $(R_m')_m$, we can choose $(\delta_m)_m$, $\delta_m>0$, such that $\lim_{m\to \infty}\delta_m=0$ and $\lim_{m\to \infty}\delta'_{m,q}=0$. Finally, apply Proposition~\ref{proposition=FromLocalToGlobalSubspace} with $r_m=\lfloor R_m'/2\rfloor$, $\epsilon_m=\max\{2\delta_m, 3\delta'_{m,q}\}$ and $(\rho_m,\omega_m)=((1-\epsilon_m)\rho,\omega+\epsilon_m)$, and we thus obtain the conclusion. If $\mathcal{M}$ only consists of Banach spaces, then consult also Remark~\ref{remark=AffineGeodesic}.
\end{proof}

By setting $\mathbf{G}_m\equiv \mathbf{G}$ for a fixed amenable group and by restricting embeddings to genuine coarse embeddings (recall Remark~\ref{remark=FromFibredToGenuine}), we in particular have the following corollary. It may be regarded as a \textit{non-linear version of Gromov's trick}. Although this may have been previously observed by other researchers, we include it for the sake of convenience of the readers; compare with \cite[Theorem~9.1]{NaorPeres} for the case of Banach spaces.

\begin{corollary}\label{corollary=NonLinearGromov}
Let $\mathcal{M}$ be a non-empty class of metric spaces that satisfies the conditions as in Theorem~$\ref{mtheorem=MainTheorem}$.$(i)$.$(2)$. Assume that for some of such $q$, it holds that $\mathcal{FS}_q(\mathcal{M})\subseteq \mathcal{M}$. Then for every $\mathrm{amenable}$ marked group $\mathbf{G}$, it holds that
\[
\mathcal{CP}_{\mathcal{M}}^{\ast}(\mathbf{G})=\mathcal{CP}_{\mathcal{M}}^{\sharp}(\mathbf{G}).
\]
\end{corollary}
On the other hand, for non-amenable marked groups, $\mathcal{CP}_{\mathcal{M}}^{\sharp}(\mathbf{G})$ is much restrictive than $\mathcal{CP}_{\mathcal{M}}^{\ast}(\mathbf{G})$. For instance, E. Guentner and J. Kaminker \cite[proposition~4.2]{GuentnerKaminker} showed that for every $a\in (0,1)$, there exists $C>0$ such that $(Cr^a,r) \in \mathcal{CP}_{\mathcal{H}\mathrm{ilbert}}^{\ast}(\mathbf{F_2})$. However, they \cite[Theorem~5.3]{GuentnerKaminker} also proved that if there exist $a\in (1/2,1]$ and $C>0$ such that $(Cr^a,r) \in \mathcal{CP}_{\mathcal{H}\mathrm{ilbert}}^{\sharp}(\mathbf{G})$, then $\mathbf{G}$ must be amenable.

\section{From equivariant equi-coarse embeddings of the Cayley boundary to fibred coarse embeddings}\label{section=TheoremA(ii)}
Here we prove $(ii)$ of Theorem~\ref{mtheorem=MainTheorem}. Unlike the proofs in Section~\ref{section=TheoremA(i)}, we do not need to impose conditions on $G_m$, $m\in \mathbb{N}$. First, we provide the proof where $(\mathbf{G}_m)_m$ is a convergent sequence.

\begin{proof}[Proof of $(ii)$ of Theorem~$\ref{mtheorem=MainTheorem}$ for the case where $\sharp(\partial_{\mathrm{Cay}}(\mathbf{G}_m)_m)=1$]

Let $\mathbf{G}_{\infty}$ be the Cayley limit of $(\mathbf{G}_m)_m$. For each $m\in \mathbb{N}$, take $R_m$ and $\beta_{\mathbf{G}_m,\mathbf{G}_{\infty},R_m}$ as in Lemma~\ref{lemma=Neighborhood}. Assume that there exist $M\in \mathcal{M}$ and an equivariant coarse embedding from $\mathbf{G}_{\infty}$ into $M$ with equivariant control pair $(\rho,\omega)$, Let $\alpha\colon M\curvearrowleft G_{\infty}$ be an action by isometries and $y\in M$ such that the orbit map $G_{\infty}\ni g_{\infty}\mapsto y\cdot \alpha(g)\in M$ gives the (equivariant) coarse embedding above. Write as  $X=\bigsqcup_{m\in \mathbb{N}}\mathrm{Cay}(\mathbf{G}_m)$.

Let $R'_m=\lfloor R_m/2\rfloor$ for every $m\in \mathbb{N}$ and  $M_x=M$ for every $x\in X$. Define a section $s\colon X\to \bigsqcup_{x\in X}M$ by $s(x)=y (\in M=M_x)$ for every $x\in X$. Now for $m\in \mathbb{N}$, $g\in \mathbf{G}_m$, define a local trivialization $t_{g,R'_m}\colon \bigsqcup_{x\in B_{\mathbf{G}_m}(g,R'_m)}M\to B_{\mathbf{G}_m}(g,R'_m)\times M$ by
\[
(t_{g,R'_M}(x))(z)=z\cdot \alpha (\beta_{\mathbf{G}_m,\mathbf{G}_{\infty},R_m}(xg^{-1})),\quad \textrm{for $x\in B_{\mathbf{G}_m}(g,R'_m)$ and $z\in M$.}
\]
Here note that since $G_m$ acts on $\mathrm{CayD}(\mathbf{G}_m)$ by right, $B_{\mathbf{G}_m}(g,R'_m) g^{-1}=B_{\mathbf{G}_m}(e_{G_m},R'_m)$. 

In what follows, we will verify conditions $(1)$--$(3)$ of Lemma~\ref{lemma=BallsFibredCoarseEmbedding}. For $(1)$, for each $x\in B_{\mathbf{G}_m}(g,R'_m)$, the map $t_{g,R'_M}(x)\colon M\to M$ is an isometry. For $x_1,x_2\in B_{\mathbf{G}_m}(g,R'_m)$,
\begin{eqnarray*}
& &d_M((t_{g,R'_M}(x_1))(s(x_1)),(t_{g,R'_M}(x_2))(s(x_2)))\\
&=& d_M(y\cdot \alpha (\beta_{\mathbf{G}_m,\mathbf{G}_{\infty},R_m}(x_1g^{-1})),y\cdot \alpha (\beta_{\mathbf{G}_m,\mathbf{G}_{\infty},R_m}(x_2g^{-1}))
\end{eqnarray*}
Since 
\begin{eqnarray*}
& &d_{\mathbf{G}_{\infty}}(\beta_{\mathbf{G}_m,\mathbf{G}_{\infty},R_m}(x_1g^{-1}),\beta_{\mathbf{G}_m,\mathbf{G}_{\infty},R_m}(x_2g^{-1}))\\
&=& d_{\mathbf{G}_m}(x_1g^{-1},x_2g^{-1})\\
&=& d_{\mathbf{G}_m}(x_1,x_2),
\end{eqnarray*}
it follows $(2)$. Finally, we check $(3)$. Let $B_{\mathbf{G}_m}(g,R'_m)\cap B_{\mathbf{G}_m}(g',R'_m)\ne \emptyset$. For each $x\in B_{\mathbf{G}_m}(g_1,R'_m)\cap B_{\mathbf{G}_m}(g_2,R'_m)$, 
\begin{eqnarray*}
((t_{g_1,R'_M}(x))\circ (t_{g_2,R'_M}(x))^{-1})(z)
&=& (z\cdot \alpha((\beta_{\mathbf{G}_m,\mathbf{G}_{\infty},R_m}(xg_2^{-1}))^{-1})) \cdot \alpha(\beta_{\mathbf{G}_m,\mathbf{G}_{\infty},R_m}(xg_1^{-1}))\\
&=& (z\cdot \alpha((\beta_{\mathbf{G}_m,\mathbf{G}_{\infty},R_m}(g_2x^{-1}))) \cdot \alpha(\beta_{\mathbf{G}_m,\mathbf{G}_{\infty},R_m}(xg_1^{-1}))\\
&=& z \cdot \alpha(\beta_{\mathbf{G}_m,\mathbf{G}_{\infty},R_m}(g_2x^{-1})\beta_{\mathbf{G}_m,\mathbf{G}_{\infty},R_m}(xg_1^{-1}))\\
&=& z\cdot \alpha(\beta_{\mathbf{G}_m,\mathbf{G}_{\infty},R_m}(g_2x^{-1}xg_1^{-1}))\\
&=& z\cdot \alpha(\beta_{\mathbf{G}_m,\mathbf{G}_{\infty},R_m}(g_2g_1^{-1})).
\end{eqnarray*}
Indeed, here we observe that $\beta_{\mathbf{G}_m,\mathbf{G}_{\infty},R_m}$ is a partial isomorphism from $B_{\mathbf{G}_m}(e_{G_m},R_m)$ to $B_{\mathbf{G}_{\infty}}(e_{G_{\infty}},R_m)$ and that $g_2g_1^{-1} \in B_{\mathbf{G}_m}(e_{G_m},2R'_m)\subseteq B_{\mathbf{G}_m}(e_{G_m},R_m)$. The expression in the very below side of the equalities above is independent of $x\in B_{\mathbf{G}_m}(g_1,R'_m)\cap B_{\mathbf{G}_m}(g_2,R'_m)$. It proves $(3)$, and hence our proof completes. Moreover, it follows from our arguments that
\[
(\rho,\omega)\in \mathcal{CP}_{M}^{\mathrm{fib}}(X).
\]
\end{proof}

We proceed to the proof of the general case; we here employ the class $\ell_q(\mathcal{M})$. Recall the definition of an open neighborhood $N(\mathbf{G},R)$ of $\mathbf{G}$ from Lemma~\ref{lemma=Neighborhood}.

\begin{proof}[Proof of $(ii)$ of Theorem~$\ref{mtheorem=MainTheorem}$ in full generality]  For each $R\in \mathbb{N}$, $\{N(\mathbf{H},R):\mathbf{H}\in \partial_{\mathrm{Cay}}(\mathbf{G}_m)_m\}$ is an open cover of $\partial_{\mathrm{Cay}}(\mathbf{G}_m)_m$. By compactness of $\partial_{\mathrm{Cay}}(\mathbf{G}_m)_m$, there exist $i(R)\in \mathbb{N}$ and $\mathbf{H}_{0}^{(R)},\ldots ,\mathbf{H}_{i(R)}^{(R)}$ such that
\[
\bigcup_{i=0}^{i(R)}N(\mathbf{H}_{i}^{(R)},R)\supseteq \partial_{\mathrm{Cay}}(\mathbf{G}_m)_m.
\]
Let $H_{0}^{(R)},\ldots ,H_{i(R)}^{(R)}$ be, respectively, the underlying groups of $\mathbf{H}_{0}^{(R)},\ldots ,\mathbf{H}_{i(R)}^{(R)}$.
By definition of $\partial_{\mathrm{Cay}}(\mathbf{G})_m$, for each $R$, there exists $n_R\in \mathbb{N}_{\geq 1}$ such that,
\[
\bigcup_{i=0}^{i(R)}N(\mathbf{H}_{i}^{(R)},R)\supseteq \overline{(\mathbf{G}_m)_m}^{\mathrm{Cay}} \setminus \{\mathbf{G}_0,\ldots ,\mathbf{G}_{n_R-1}\}
\]
holds, where $\overline{\{\ \}}^{\mathrm{Cay}}$ denotes the closure in the Cayley topology. Note that for $R=0$, then the left-hand side above,  in fact, includes $\overline{(\mathbf{G}_m)_m}^{\mathrm{Cay}}$.
For each $R\in \mathbb{N}$ and  for every $m\in \mathbb{N}_{\geq n_R}$, choose $0\leq i\leq i(R)$ such that $\mathbf{G}_m\in N(\mathbf{H}_{i}^{(R)},R)$ (if there exist at least two such $i$, choose the smallest $i$). We write this $i$ as $i_m^{(R)}$.

Set a new disjoint union as
\[
X'=\bigsqcup_{R\in \mathbb{N}}\left(\bigsqcup_{0\leq i\leq i(R)} \left(\bigsqcup_{n_R\leq m \leq n_{R+1}+R\ :\ i_m^{(R)}=i} \mathrm{Cay}(\mathbf{G}_m) \right)\right).
\]

Now assume that $\partial_{\mathrm{Cay}}(\mathbf{G}_m)_m$ is uniformly a-$\mathcal{M}$-menable; there exists a pair $(\rho,\omega)$ of non-decreasing proper functions $[0,\infty)\to [0,\infty)$ such that
\[
(\rho,\omega)\in \bigcap_{\mathbf{H}\in \partial_{\mathrm{Cay}}(\mathbf{G}_m)_m} \mathcal{CP}^{\sharp}_{\mathcal{M}}(\mathbf{H}).
\]
In particular, for every $R\in \mathbb{N}$ and for every $0\leq i\leq i(R)$, there exist $M_{i}^{(R)}\in \mathcal{M}$, $y_{i}^{(R)}\in M_i^{(R)}$ and an action $\alpha_{i}^{(R)}\colon M_{i}^{(R)}\curvearrowleft H_i^{(R)}$ such that the orbit map $H_i^{(R)}\ni h \mapsto y_i^{(R)}\cdot \alpha_i^{(R)}(h) \in M_i^{(R)}$ is an (equivariant) coarse embedding with equivariant control pair $(\rho,\omega)$. Fix $q\in [1,\infty)$ and define
\[
M_q=\left(\prod_{R\in \mathbb{N}}\left(\prod_{0\leq i\leq i(R)} (M_{i}^{(R)},y_i^{(R)})\right)\right)_{\ell_q}.
\]
Note that this is an (at most) countable $\ell_q$-product; hence $M_q\in \ell_q(\mathcal{M})$. 

By Lemma~\ref{lemma=Subsets}, it suffices to construct a fibred coarse embedding as a disjoint union from $X'$ into $M_q$. Let $(M_q)_x=M_q$ for all $x\in X'$ and $s\colon X'\to \bigsqcup_{x\in X'}M_q $ be $s(x)=(y_{j}^{(r)})_{r,j}$. For each $n_R\leq m\leq n_{R+1}+R$ with $i_m^{(R)}=i$, consider the component $\mathrm{Cay}(\mathbf{G}_m)$ in $X'$ associated with these $R$ and $i$. Set $R'_m=\lfloor R/2 \rfloor$ and construct $t_{g,R_m'}$ by
\[
(t_{g,R_m'}(x))((z)_{r,j})=(w_{r,j})_{r,j}
\]
for $x\in B_{\mathbf{G}_m}(g,R_m')$ and for $(z)_{r,j}\in M_q$, where,
\[
w_{r,j}=
\left\{
\begin{array}{cl}
z_{R,i}\cdot \alpha_i^{(R)}(\beta_{\mathbf{G}_m,\mathbf{H}_{i}^{(R)},R}(xg^{-1})), & \textrm{if $(r,j)=(R,i)$},\\
z_{r,j}, & \textrm{otherwise}.
\end{array}\right.
\]
Then in a similar argument to one in the previous proof for the case where $\sharp(\partial_{\mathrm{Cay}}(\mathbf{G}_m))=1$, we may verify conditions $(1)$--$(3)$ in Lemma~\ref{lemma=BallsFibredCoarseEmbedding}; recall also Remark~\ref{remark=DependenceOnR}. Furthermore, we obtain that
\[
(\rho,\omega) \in \mathcal{CP}_{M_q}^{\mathrm{fib}}\left(\bigsqcup_{m\in \mathbb{N}}\mathrm{Cay}(\mathbf{G}_m)\right).
\]
\end{proof}

\begin{proof}[Proof of Corollary~$\ref{mcorollary=a-T-menable}$]
For the case where $\mathcal{M}=\mathcal{H}\mathrm{ilbert}$, $\mathcal{B}_{r,C}^{\mathrm{type}}$ or $\mathcal{CAT}(0)_{\leq \delta_0}$, the assertions immediate follow from Theorem~\ref{mtheorem=MainTheorem}; see also arguments in Examples~\ref{example=BanachSpaces} and \ref{example=MetricSpaces}. For $\mathcal{M}=L_q$,  Naor and Y. Peres employed the classification of separable closed subspaces of $L_q$-spaces and indicated a way to coming back to $L_q$ from $\mathcal{FS}_q(L_q)$; see the last assertion of \cite[Theorem~9.1]{NaorPeres} for more details.
\end{proof}

\section{The absorption trick}\label{section=Gadgets}
In this section, we explain the \textit{absorption trick}, which appeared in the work of Bartholdi and Erschler \cite{BartholdiErschler}. We employ this trick to prove Theorem~\ref{mtheorem=NonExact}.

\subsection{Standard (restricted) wreath products}\label{subsection=WreathProducts}
We recall the definition of \textit{standard} (restricted) \textit{wreath products}; see also \cite[Proposition~2.11]{MimuraSako}.  For two groups $G$ and $H$, $G\wr H$ is defined to be $\left(\bigoplus_{h\in H}G\right)\rtimes H$, where $H$ acts by permutation of coordinates by right. For $g\in G$ and $h\in H$, by $g\delta_h$ we denote the element in $\bigoplus_{h\in H}G$ whose $h$-entry is $g$ and all of the other entries are $e_G$. We use $\mathbf{e}$ for the group unit of $\bigoplus_{h\in H}G$. If $\mathbf{G}=(G;s_1,\ldots ,s_k)$ and $\mathbf{H}=(H;t_1,\ldots ,t_l)$ are two marked groups, then we endow $G\wr H$ with the standard $(k+l)$-marking as follows:
\[
((s_1\delta_{e_H},e_H),(s_2\delta_{e_H},e_H),\ldots ,(s_k\delta_{e_H},e_H),(\mathbf{e},t_1),(\mathbf{e},t_2),\ldots ,(\mathbf{e},t_l)).
\]
We write the marked group of $G\wr H$ with the standard marking above as $\mathbf{G}\wr\mathbf{H}$. Then, for $\mathbf{G}_m\to \mathbf{G}_{\infty}$ and $\mathbf{H}_n\to \mathbf{H}_{\infty}$ (respectively in $\mathcal{G}(k)$ and $\mathcal{G}(l)$) in the Cayley topology, we have that as $\min\{m,n\}\to \infty$,
\[
\mathbf{G}_m\wr \mathbf{H}_n\quad \stackrel{\mathrm{Cay}}{\longrightarrow} \quad \mathbf{G}_{\infty}\wr \mathbf{H}_{\infty} \quad \textrm{in $\mathcal{G}(k+l)$};
\]
see ${\S}$2.4. Theorem in \cite{VershikGordon} or  \cite[Lemma~4.6]{MimuraNonExact}. This may be clear to the reader who is familiar with a relationship between wreath products and random walks.

\subsection{The absorption trick}\label{subsection=AbsorptionLemma}

The following lemma enables us to \textit{absorb a group into some abelian group by taking the wreath product by an infinite group}. For this reason, we call the idea of it \textit{the absorption trick}. The original form in the paper of Bartholdi and Erschler \cite{BartholdiErschler} stated it in terms of permutational (restricted) wreath products; here we formulate it for a simpler case.

\begin{lemma}[A prototype of the \textit{absorption trick}; Special case of Lemma~$6.13$ in Bartholdi--Erschler \cite{BartholdiErschler}]\label{lemma=Absorption}
Let $G$ be a finitely generated group and fix $(g_1,\ldots ,g_k)$ a marking of $G$. For each $j\in [k]$, let $C_j$ be the cyclic group of the same order as for $g_j$. Then, for every infinite and finitely generated group $P$, there exists a system of marking $(S_m)_{m\in \mathbb{N}}$ of $G\wr P$ with fixed size such that
\[
(G\wr P;S_m) \quad \stackrel{\mathrm{Cay}}{\longrightarrow}\quad (C_1\times C_2\times \cdots \times C_k)\wr P,
\]
with a suitable marking of the Cayley limit group.
\end{lemma}
For the sake of completeness, we include (idea of) the proof. See \cite[Subsection~5.2]{MimuraNonExact} for a more detailed demonstration for $P=\mathbb{Z}$.

\begin{proof}
Fix a marking $T=(t_1,\ldots ,t_l)$ of $P$. Since $P$ is infinite, for every $m\in \mathbb{N}$, there exists $e_{P}=x_1^{(m)},x_2^{(m)},\ldots ,x_k^{(m)}\in P$ such that $B_{\mathrm{CayD}(P;T)}(x_j^{(m)},m)$, $j\in [k]$, are mutually disjoint. Now, define a system $(S_m)_{m\in \mathbb{N}}$ of markings of $G\wr P$ by 
\[
S_m=((g_1\delta_{e_P},e_P),(g_2\delta_{x_2^{(m)}},e_P),\ldots ,(g_k\delta_{x_k^{(m)}},e_P),(\mathbf{e},t_1),\ldots,(\mathbf{e},t_l)),
\]
where $\mathbf{e}$ means the group unit of $\bigoplus_{P}G$. Let $(S_m)_1$ be the set of the first $k$ elements in the marking $S_m$. Then the following holds true: \textit{For   $\gamma_1, \gamma_2$ elements in $G\wr P$ of the form $\tau^{-1} \sigma \tau$, $\sigma\in (S_m)_1$, $\tau\in P$, if $\gamma_1,\gamma_2$ and $\gamma_1\gamma_2$ are all contained in the ball $B_{(G\wr P;S_m)}(e_{G\wr P},m)$ of radius $m$, then $\gamma_1$ and $\gamma_2$ commute}. By a similar reasoning to one in the proof of \cite[Lemma~5.1]{MimuraSako}, we conclude that as $m\to \infty$,
\[
(G\wr P;S_m)\quad \stackrel{\mathrm{Cay}}{\longrightarrow}\quad  (C_1\times C_2 \times \cdots \times C_k)\wr P, 
\]
with a suitable marking of the Cayley limit group. See \cite[the proof of Lemma~5.3]{MimuraNonExact} for more details.
\end{proof}

Since the constant sequence of $G\wr P$ with a fixed standard marking converge to itself, Lemma~\ref{lemma=Absorption} can be utilized as a source of producing two systems of different markings of a group that produce Cayley limit groups of quite different nature. For instance, we have the following.

\begin{lemma}[A variant of the absorption trick]\label{lemma=AbsorptionZ}
Let $G$ be a LEF group. Let $(G_m)_{m\in \mathbb{N}}$ be a sequence of finite groups that is obtained from the underlying groups of a LEF approximation of $G$ $($with a fixed marking$)$. Then, there exist two different systems of markings $(S_m)_m$ and $(T_m)_m$ of $(G_m \wr (\mathbb{Z}/m\mathbb{Z}))_{m\in \mathbb{N}_{\geq 3}}$ such that the following two conditions hold true:
\begin{itemize}
  \item The sequence $(G_m \wr (\mathbb{Z}/m\mathbb{Z});S_m)_{m\in \mathbb{N}_{\geq 3}}$ converges in the Cayley topology  to a solvable marked group.
  \item The sequence $(G_m \wr (\mathbb{Z}/m\mathbb{Z});T_m)_{m\in \mathbb{N}_{\geq 3}}$ converges in the Cayley topology  to $G\wr \mathbb{Z}$ with a suitable marking.
\end{itemize}
\end{lemma}

We will make use of another variant of the absorption trick in the proof of Theorem~\ref{mtheorem=NonExact}.
\begin{proof}
Fix $(g_1,\ldots ,g_k)$ a marking of $G$. For every $m\in \mathbb{N}_{\geq 3}$, let $(g_1^{(m)},\ldots ,g_k^{(m)})$ be the corresponding marking of $G_m$ in the LEF approximation. Note that $((\mathbb{Z}/m\mathbb{Z};1))_{m\in \mathbb{N}_{\geq 3}}$ converges to $(\mathbb{Z};1)$ in the Cayley topology; here we can take $R_m=\lfloor (m-1)/2 \rfloor$ as in $(\star)$ in Lemma~\ref{lemma=Neighborhood}. For every $m\in \mathbb{N}_{\geq 3}$, set
\[
r_m=\min\{R_m, \left\lfloor\frac{\mathrm{diam}(\mathrm{CayD}(\mathbb{Z}/m\mathbb{Z});1)}{4k}\right\rfloor \}.
\]
Then, $\lim_{m\to \infty}r_m=+\infty$. There exist $0=x_1^{(m)}, x_2^{(m)},\ldots ,x_k^{(m)} \in \mathbb{Z}/m\mathbb{Z}$ such that $B_{\mathrm{CayD}(\mathbb{Z}/m\mathbb{Z};1)}(x_j^{(m)},r_m)$, $j\in [k]$, are mutually disjoint. Finally, define two systems $(S_m)_m$ and $(T_m)_m$ of markings of $(G_m \wr (\mathbb{Z}/m\mathbb{Z}))_{m\in \mathbb{N}_{\geq 3}}$ by
\begin{eqnarray*}
S_m&=&((g_1^{(m)}\delta_{e_{\mathbb{Z}/m\mathbb{Z}}},e_{\mathbb{Z}/m\mathbb{Z}}),(g_2^{(m)}\delta_{x_2^{(m)}},e_{\mathbb{Z}/m\mathbb{Z}}),\ldots ,(g_k^{(m)}\delta_{x_k^{(m)}},e_{\mathbb{Z}/m\mathbb{Z}}),(\mathbf{e},1)),\\
T_m&=&((g_1^{(m)}\delta_{e_{\mathbb{Z}/m\mathbb{Z}}},e_{\mathbb{Z}/m\mathbb{Z}}),(g_2^{(m)}\delta_{e_{\mathbb{Z}/m\mathbb{Z}}},e_{\mathbb{Z}/m\mathbb{Z}}),\ldots ,(g_k^{(m)}\delta_{e_{\mathbb{Z}/m\mathbb{Z}}},e_{\mathbb{Z}/m\mathbb{Z}}),(\mathbf{e},1)),
\end{eqnarray*}
where $\mathbf{e}$ means the group unit of $\bigoplus_{\mathbb{Z}/m\mathbb{Z}}G_m$. (Hence $T_m$ is the standard marking of $G_m\wr(\mathbb{Z}/m\mathbb{Z})$.) Then, we have that
\begin{eqnarray*}
(G_m\wr (\mathbb{Z}/m\mathbb{Z});S_m) &\stackrel{\mathrm{Cay}}{\longrightarrow}&  (C_1\times C_2 \times \cdots \times C_k)\wr \mathbb{Z}, \\
(G_m\wr (\mathbb{Z}/m\mathbb{Z});T_m) &\stackrel{\mathrm{Cay}}{\longrightarrow}&  G\wr \mathbb{Z}, 
\end{eqnarray*}
where for every $j\in [k]$, $C_j$ is the cyclic group of the same order as for $g_j$. 
\end{proof}

\begin{remark}\label{remark=Gruenberg}
K. W. Gruenberg \cite{Gruenberg} showed that a wreath product $G\wr H$ with an infinite $H$ is \textit{never} RF unless $G$ is abelian. Hence, our construction  as in Lemma~\ref{lemma=AbsorptionZ} may be available only after we extend the framework from RF approximations to LEF ones. In addition, if $G$ is not abelian, then the Cayley convergence of $(G_m\wr P_m;S_m)$ to the amenable marked group $(C_1\times  \cdots \times C_k)\wr P$ above is a LEF approximation, but \textit{not} a RF one. This is  because $C_1\times \cdots \times C_k$ is abelian but $G_m$ for large $m$ is not.
\end{remark}

We refer the reader to \cite{MimuraNonExact} for a further development on the absorption trick.

\section{Examples}\label{section=Examples}
\subsection{Special linear groups}\label{subsection=SpecialLinear}
Here we discuss coarse properties of $X',Y',V'$ and $W'$ as in Examples~\ref{example=SpecialLinear} and \ref{example=SpecialLinearZ}. In our Part I paper \cite[Remark~5.10]{MimuraSako}, we observed that
\begin{eqnarray*}
(G_m;S_m)\stackrel{\mathrm{Cay}}{\longrightarrow} \mathrm{N}_{>}(\mathbb{Z},\mathbb{F}_p[t])\rtimes \mathbb{Z}&,& (G_m;T_m)\stackrel{\mathrm{Cay}}{\longrightarrow} \mathrm{SL}(\mathbb{Z},\mathbb{F}_p[t])\rtimes \mathbb{Z},\\
(H_m;P_m)\stackrel{\mathrm{Cay}}{\longrightarrow} \mathrm{N}_{>}(\mathbb{Z},\mathbb{Z})\rtimes \mathbb{Z}&,& (H_m;Q_m)\stackrel{\mathrm{Cay}}{\longrightarrow} \mathrm{SL}(\mathbb{Z},\mathbb{Z})\rtimes \mathbb{Z},
\end{eqnarray*}
with respectively suitable markings of the Cayley limit groups. Here for a unital commutative ring $A$ (associative), the group $\mathrm{SL}(\mathbb{Z},A)$ denotes the union of $\mathrm{SL}(K,A)=\{g\in \mathrm{Mat}_{K\times K}(A):\mathrm{det}(g)=1\}$ over all finite non-empty sets $K\subseteq \mathbb{Z}$ (via the natural inclusion $\mathrm{SL}(K,A)\hookrightarrow \mathrm{SL}(\mathbb{Z},A)$). Similarly, $\mathrm{N}_{>}(\mathbb{Z},A)$ denotes the union of 
\[
\mathrm{N}_{>}(K,A)=\{g\in \mathrm{Mat}_{K\times K}(A):(g)_{i,i}=1 \textrm{ for all $i\in K$},\ (g)_{i,j}=0 \textrm{ for all $i>j$, $i,j\in K$}\}\]
over all finite non-empty sets $K\subseteq \mathbb{Z}$. Here $>$ is the standard total order on $\mathbb{Z}$. The actions of $\mathbb{Z}$ in the semi-direct products above are given by the right translation of $\mathbb{Z}$ on the coordinate set $\mathbb{Z}$. In the Part I paper \cite[Remark~5.10]{MimuraSako}, we deduced property A for $X'$ and $V'$ by amenability of the Cayley limit groups $\mathrm{N}_{>}(\mathbb{Z},\mathbb{F}_p[t])\rtimes \mathbb{Z}$ and $\mathrm{N}_{>}(\mathbb{Z},\mathbb{Z})\rtimes \mathbb{Z}$. 

\begin{definition}\label{definition=FixedPointProperty}
Let $\mathcal{M}$ be a non-empty class of metric spaces. We say that a group $G$ has property $(\mathrm{F}_{\mathcal{M}})$ if for every $M\in \mathcal{M}$, every action $\alpha\colon M\curvearrowleft G$ by isometries admits a global fixed point.
\end{definition}

The following are showed by several researchers.
\begin{theorem}\label{theorem=StrongPropertyT}
\begin{enumerate}[$(1)$]
\item $($V. Lafforgue \cite{Lafforgue1}, \cite{Lafforgue2}$)$ For every prime $p$ and for every $n\in \mathbb{N}_{\geq 3}$, the group $\mathrm{SL}(n,\mathbb{F}_p[t])$ has property $(\mathrm{F}_{\mathcal{B}_{\mathrm{type}>1}})$.
\item $($Izeki--Nayatani \cite{IzekiNayatani}$)$ For every prime $p$ and for every $n\in \mathbb{N}_{\geq 3}$, the group $\mathrm{SL}(n,\mathbb{F}_p[t])$ has property $(\mathrm{F}_{\mathcal{CAT}(0)_{\leq 0}})$.
\item $($de Laat--Mimura--de la Salle \cite{deLaatMimuradelaSalle}$)$ For every $E\in \mathcal{B}_{\beta<1/2}$, there exists $N_E\in \mathbb{N}_{\geq 3}$ such that for every $n\in \mathbb{N}_{\geq N_E}$, the group $\mathrm{SL}(n,\mathbb{Z})$ has property $(\mathrm{F}_E)$.
\end{enumerate}
\end{theorem}
Indeed, $(2)$ is deduced from the following argument: First, we consider a uniform lattice in $\mathrm{SL}(n,\mathbb{F}_p((t^{-1})))$. Consider the first strictly positive Laplace eigenvalue $\lambda_1$ for the link graph associated to it; recall the argument above Example~\ref{example=r-UniformlyConvex}. Then, exactly the same estimate as one for a uniform lattices in $\mathrm{SL}(n,\mathbb{Q}_p)$ applies. This is because local information is the same for buildings associated with $\mathrm{PGL}(n,\mathbb{Q}_p)$ and for those associated with $\mathrm{PGL}(n,\mathbb{F}_p((t^{-1})))$.  By \cite[Section~6, Example~1]{IzekiNayatani}, the estimate is given as
\[
\lambda_1=1-\frac{\sqrt{p}}{p+1}.
\]
For every prime $p$, the estimate above of $\lambda_1$ is strictly bigger than $1/2$. Then, by \cite[Theorem~1.1]{IzekiNayatani}, every uniform lattice in $\mathrm{SL}(n,\mathbb{F}_p((t^{-1})))$ has property $(\mathrm{F}_{\mathcal{CAT}(0)_{\leq 0}})$. Even though $\mathrm{SL}(n,\mathbb{F}_p[t])$ is a non-uniform lattice in $\mathrm{SL}(n,\mathbb{F}_p((t^{-1})))$, we obtain the same conclusion as in $(2)$ through $L_2$-induction process; see \cite[Section~8]{BaderFurmanGelanderMonod}.

Item $(1)$ of Theorem~\ref{theorem=StrongPropertyT} has been generalized to other higher rank lattices over non-archimedean fields; see \cite{Liao}.

\begin{remark}\label{remark=IzekiNayatani}
On $(2)$ of Theorem~\ref{theorem=StrongPropertyT}, with the aid of \cite[Proposition~6.3]{IzekiNayatani}, the following strengthening holds true: For every prime $p$ and for every $n\in \mathbb{N}_{\geq 3}$, the group $\mathrm{SL}(n,\mathbb{F}_p[t])$ has property $(\mathrm{F}_{\mathcal{M}})$, where $\mathcal{M}=\mathcal{CAT}(0)_{< \delta(p)}$ and  $\delta(p)$ is described above Corollary~\ref{corollary=SpecialLinear}. The key here is for every $\delta_0<\delta(p)$, it holds that
\[
(1-\delta_0)\left(1-\frac{\sqrt{p}}{p+1}\right)>\frac{1}{2}.
\]
By combining this with the inequality 
\[
\lambda_1(\Gamma,M)\geq (1-\delta(M))\lambda_1(\Gamma)
\]
as in the argument above Example~\ref{example=r-UniformlyConvex}, we have that for every $M$ in $\mathcal{CAT}(0)_{<\delta(p)}$, the Wang-type non-linear spectral gap of the link graph with target in $M$ is greater than $1/2$. Then, \cite[Theorem~1.1]{IzekiNayatani} applies and establishes property $(\mathrm{F}_M)$. 
\end{remark}

\begin{proof}[Proof of Corollary~$\ref{corollary=SpecialLinear}$]
Item $(1)$ follows from the main result \cite[Theorem~A]{MimuraSako} of our Part I paper, because the Cayley limit groups for $X'$ and $V'$ are both amenable. To show $(2)$, observe that  $\mathrm{SL}(\mathbb{Z},\mathbb{F}_p[t])\rtimes \mathbb{Z}$, $\mathrm{N}_{>}(\mathbb{Z},\mathbb{Z})\rtimes \mathbb{Z}$, $\mathrm{SL}(\mathbb{Z},\mathbb{Z})\rtimes \mathbb{Z}$ have infinite asymptotic dimension. Then combine it with Proposition~\ref{proposition=UniformCoarseEmbedding} and Remark~\ref{remark=AsymptoticDimension}. Items $(3)$ and $(4)$ both follow from $(i).(1)$ of Theorem~\ref{mtheorem=MainTheorem} (and Corollary~\ref{mcorollary=a-T-menable}) and Theorem~\ref{theorem=StrongPropertyT}, together with Remark~\ref{remark=IzekiNayatani}. Note that if a group $G$ contains an isomorphic copy of an infinite group with property $(\mathrm{F}_{\mathcal{M}})$, then $G$ \textit{fails} to be a-$\mathcal{M}$-menable.
\end{proof}

We make a remark that similar constructions to ones above are available for even numbers $m$ by slight modification; see the last part of \cite[Remark~5.10]{MimuraSako}.

\begin{remark}\label{remark=Bourdon}
The proof of $(2)$ of Theorem~\ref{theorem=StrongPropertyT} by Izeki and Nayatani \cite{IzekiNayatani} has been generalized to the case of fixed point properties with respect to $r$-uniformly convex metric spaces. More precisely, if 
\[
\Psi(\lambda^{(r)}(\Gamma,\mathbb{R}))> T
\]
is satisfied for a large enough $T=T(r,C)$ (for instance, $T=1/2$ works for several situations), then we may establish property $(\mathrm{F}_{\mathcal{UC}_{r,C}^{\Psi}})$ for the corresponding group. In this manner, results in \cite{Bourdon} may be utilized to demonstrate certain fixed point properties for groups acting on $\tilde{A}_2$-buildings.
\end{remark}

\subsection{Three markings one of whose limit is amenable but the others are non-exact}\label{subsection=NonExact}
We prove Theorem~\ref{mtheorem=NonExact}. The main ingredient is a remarkable result by Osajda \cite{OsajdaRF} of the existence of a (finitely generated) \textit{RF group that is non-exact}. In fact, what we need in our construction is the LEF property. This property is deduced in a much simpler way than the full argument in \cite{OsajdaRF}: Indeed, it is automatic because this group is constructed as a limit in the Cayley topology of RF groups, and \textit{such groups are always LEF}. In \cite{Osajda} and \cite{ArzhantsevaOsajda}, discussion on the LEF property was not explicitly written. In aforementioned work of Osajda \cite{OsajdaRF}, the construction that satisfies the condition above was given. 

\begin{remark}\label{remark=Osajda}
Osajda pointed out to the authors that although it is implicit in his paper,  the resulting group (RF but non-exact) in \cite{OsajdaRF} is furthermore a-$\mathrm{T}$-menable. To see this, he used a method developed in \cite{Osajda} to transfer wall structures on the finite presented graphical small cancellation groups in his construction at all finitary stages to that on the infinitely presented limit group. See also \cite{Osajda} and \cite{ArzhantsevaOsajda}. We employ this a-$\mathrm{T}$-menability in our proof of Theorem~\ref{mtheorem=NonExact}. 
\end{remark}

An outline of our construction as in Theorem~\ref{mtheorem=NonExact} goes as follows: We combine the absorption trick in Subsection~\ref{subsection=AbsorptionLemma}  with our examples as in Subsection~\ref{subsection=SpecialLinear} (Example~\ref{example=SpecialLinear}),
\begin{eqnarray*}
(\mathrm{SL}(m,\mathbb{F}_{p^{n_m}});\sigma^{(m)},\upsilon^{(m)},\tau^{(m)})&\stackrel{\mathrm{Cay}}{\longrightarrow}& \mathrm{N}_{>}(\mathbb{Z},\mathbb{F}_p[t])\rtimes \mathbb{Z} ,\\
(\mathrm{SL}(m,\mathbb{F}_{p^{n_m}});\sigma^{(m)},\sigma'^{(m)},\upsilon^{(m)},\tau^{(m)})&\stackrel{\mathrm{Cay}}{\longrightarrow}& \mathrm{SL}(\mathbb{Z},\mathbb{F}_p[t])\rtimes \mathbb{Z},
\end{eqnarray*}
with respect to suitable markings of the Cayley limit groups. Here recall that the former limit group $\mathrm{N}_{>}(\mathbb{Z},\mathbb{F}_p[t])\rtimes \mathbb{Z}$ is amenable, whereas the latter $\mathrm{SL}(\mathbb{Z},\mathbb{F}_p[t])\rtimes \mathbb{Z}$ contains a copy of $\mathrm{SL}(3,\mathbb{F}_p[t])$, which has property $(\mathrm{F}_{\mathcal{M}})$ for $\mathcal{M}=\mathcal{B}_{\mathrm{type}>1}$ and $\mathcal{M}=\mathcal{CAT}(0)_{<\delta(p)}$.

\begin{proof}[Proof of Theorem~$\ref{mtheorem=NonExact}$]
Let $G$ be the (finitely generated) RF group without property A constructed in \cite{OsajdaRF}, and $S=(g_1,\ldots ,g_k)$ be a $k$-marking of $G$. Take $(\mathbf{G}_n)_{n\in \mathbb{N}}$  a RF approximation of $(G;S)$ (in fact, what we need here in principle are a LEF group without property A and  a LEF approximation of it; see also \cite{Osajda} and \cite{ArzhantsevaOsajda}). For every $n\in \mathbb{N}$, write $\mathbf{G}_n=(G_n;g_1^{(n)},\ldots ,g_k^{(n)})$. Recall from Remark~\ref{remark=Osajda} that this $G$ is a-$\mathrm{T}$-menable.

Recall two systems $(\sigma^{(m)},\upsilon^{(m)},\tau^{(m)})_{m\in \mathbb{N}_{\mathrm{odd}}}$ and $(\sigma^{(m)},\sigma'^{(m)},\upsilon^{(m)},\tau^{(m)})_{m\in \mathbb{N}_{\mathrm{odd}}}$ of markings of $(\mathrm{SL}(m,\mathbb{F}_{p^{n_m}}))_{m\in \mathbb{N}_{\mathrm{odd}}}$ from Example~\ref{example=SpecialLinear}. Set $m=2n+3$, and rewrite $n_m$ and $\sigma^{(m)},\sigma'^{(m)},\upsilon^{(m)},\tau^{(m)}$, respectively, as $l_n$ and $\sigma_n,\sigma'_n,\upsilon_n,\tau_n$. Hence, we have two markings $(\sigma_n,\upsilon_n,\tau_n)$ and $(\sigma_n,\sigma'_n,\upsilon_n,\tau_n)$ of $\mathrm{SL}(2n+3,\mathbb{F}_{p^{l_n}})$. 

Let $H_{n,p}=G_n\wr \mathrm{SL}(2n+3,\mathbb{F}_{p^{l_n}})$. Let $d=k+4$. Then,
\[
(\mathrm{SL}(2n+3,\mathbb{F}_{p^{l_n}});\sigma_n,\upsilon_n,\tau_n)\quad\stackrel{\mathrm{Cay}}{\longrightarrow}\quad \mathrm{N}_{>}(\mathbb{Z},\mathbb{F}_p[t])\rtimes \mathbb{Z},
\]
with the suitable marking of $\mathrm{N}_{>}(\mathbb{Z},\mathbb{F}_p[t])\rtimes \mathbb{Z}$. For each $n\in \mathbb{N}$, take $R_n\in \mathbb{N}$ as in $(\star)$ in Lemma~\ref{lemma=Neighborhood} associated with the convergence above. Let 
\[
r_n=\min\{R_n, \left\lfloor\frac{\mathrm{diam}(\mathrm{CayD}(\mathrm{SL}(2n+3,\mathbb{F}_{p^{l_n}});\sigma_n,\upsilon_n,\tau_n))}{4k}\right\rfloor \}.
\]
Then $\lim_{m\to \infty}r_n=+\infty$. By definition of $r_n$, for each $n\in \mathbb{N}$, there exists $x_1^{(n)}=e_{\mathrm{SL}(2n+3,\mathbb{F}_{p^{l_n}})}, x_2^{(n)},\ldots ,x_{k}^{(n)} \in \mathrm{SL}(2n+3,\mathbb{F}_{p^{l_n}})$ such that the $r_n$-balls in the Cayley diagram $\mathrm{CayD}(\mathrm{SL}(2n+3,\mathbb{F}_{p^{l_n}});\sigma_n,\upsilon_n,\tau_n)$ centered at $x_j^{(n)}$, $j\in [k]$, are mutually disjoint. Finally, for every $n\in \mathbb{N}$, set a marking $S_n$ of $H_{n,p}$ as 
\begin{align*}
S_n=(&(g_1^{(n)}\delta_{e_{\mathrm{SL}(2n+3,\mathbb{F}_{p^{l_n}})}},e_{\mathrm{SL}(2n+3,\mathbb{F}_{p^{l_n}})}),(g_2^{(n)}\delta_{x_2^{(n)}},e_{\mathrm{SL}(2n+3,\mathbb{F}_{p^{l_n}})})\ldots,(g_k^{(n)}\delta_{x_k^{(n)}},e_{\mathrm{SL}(2n+3,\mathbb{F}_{p^{l_n}})}),\\
&(\mathbf{e},\sigma_n),(\mathbf{e},(\sigma_n)^{-1}),(\mathbf{e},\upsilon_n),(\mathbf{e},\tau_n)),
\end{align*}
where $\mathbf{e}$ is the group unit of $\bigoplus_{\mathrm{SL}(2n+3,\mathbb{F}_{p^{l_n}})}G_n$. 
(The $(\mathbf{e},(\sigma_n)^{-1})$ above is redundant as a marking; this element is added only in order to meet assertion $(1)$ of the theorem.) 

For the other two markings $(T_n)_n$ and $(U_n)_n$, without employing $x_1^{(n)},\ldots ,x_k^{(n)}$, we simply set
\begin{align*}
T_n=(&(g_1^{(n)}\delta_{e_{\mathrm{SL}(2n+3,\mathbb{F}_{p^{l_n}})}},e_{\mathrm{SL}(2n+3,\mathbb{F}_{p^{l_n}})}),\ldots,(g_k^{(n)}\delta_{e_{\mathrm{SL}(2n+3,\mathbb{F}_{p^{l_n}})}},e_{\mathrm{SL}(2n+3,\mathbb{F}_{p^{l_n}})}),\\
&(\mathbf{e},\sigma_n),(\mathbf{e},(\sigma_n)^{-1}),(\mathbf{e},\upsilon_n),(\mathbf{e},\tau_n)),\\
U_n=(&(g_1^{(n)}\delta_{e_{\mathrm{SL}(2n+3,\mathbb{F}_{p^{l_n}})}},e_{\mathrm{SL}(2n+3,\mathbb{F}_{p^{l_n}})}),\ldots,(g_k^{(n)}\delta_{e_{\mathrm{SL}(2n+3,\mathbb{F}_{p^{l_n}})}},e_{\mathrm{SL}(2n+3,\mathbb{F}_{p^{l_n}})}),\\
&(\mathbf{e},\sigma_n),(\mathbf{e},\sigma'_n),(\mathbf{e},\upsilon_n),(\mathbf{e},\tau_n)).
\end{align*}
Then as $n\to \infty$, respectively with suitable markings of the Cayley limit groups, we have the following Cayley convergences:
\begin{eqnarray*}
(H_{n,p};S_n) &\stackrel{\mathrm{Cay}}{\longrightarrow}&   (C_1\times C_2 \times \cdots \times C_k)\wr (\mathrm{N}_{>}(\mathbb{Z},\mathbb{F}_p[t])\rtimes \mathbb{Z}), \\
(H_{n,p};T_n) &\stackrel{\mathrm{Cay}}{\longrightarrow}&  G \wr (\mathrm{N}_{>}(\mathbb{Z},\mathbb{F}_p[t])\rtimes \mathbb{Z}), \\
(H_{n,p};U_n) &\stackrel{\mathrm{Cay}}{\longrightarrow}&   G \wr (\mathrm{SL}(\mathbb{Z},\mathbb{F}_p[t])\rtimes \mathbb{Z}), 
\end{eqnarray*}
where $C_1,\ldots, C_k$ are as in Lemma~\ref{lemma=Absorption}. Indeed, the first Cayley convergence follows from a variant of the \textit{absorption trick}; compare with Lemmata~\ref{lemma=Absorption} and \ref{lemma=AbsorptionZ}. By Theorem~\ref{theorem=StrongPropertyT} and Remark~\ref{remark=IzekiNayatani}, we confirm $(2)$ and $(4)$; note that exactness of countable discrete groups passes to subgroups. To see $(3)$, recall Remark~\ref{remark=Osajda} and the fact that in a short exact sequence of countable discrete groups,
\[
1\quad \longrightarrow \quad G_1 \quad \longrightarrow \quad G_2 \quad \longrightarrow \quad G_3 \quad \longrightarrow \quad  1,
\]
$G_2$ is a-$\mathrm{T}$-menable if $G_1$ is a-$\mathrm{T}$-menable and if $G_3$  is \textit{amenable}; see \cite[Example~6.1.6]{CCJJV}. 
Assertion $(1)$ is by construction; observe that $(\sigma_n)^{-1}$ and $\sigma'_n$ are conjugate in $\mathrm{SL}(2n+3,\mathbb{F}_{p^{l_n}})$. It ends our proof. 
\end{proof}

\subsection{Embedded Banach expanders}
In this subsection, we give a definition of \textit{embedded Banach expanders}.
\begin{definition}\label{definition=EmbeddedExpanders}
Let $\mathcal{E}$ be a non-empty class of Banach spaces and fix $q\in [1,\infty)$. A sequence of finite connected graphs $(\Gamma_m)_{m\in \mathbb{N}}$ of uniformly bounded degree is said to admit \textit{embedded Banach $(\mathcal{E},q)$-expanders} if there exist a subsequence $(m_n)_{n\in \mathbb{N}}$ of $(m)_m$ and a sequence of finite connected graphs $(\Lambda_{m_n})_{n\in \mathbb{N}}$ such that all of the following hold true:
\begin{itemize}
 \item There exists $D>0$ such that for each $n\in \mathbb{N}$, there exists an \textit{injective} map $\iota_{m_n}\colon V(\Lambda_{m_n})\to V(\Gamma_{m_n})$ between the vertex sets such that the map $\iota_{m_n}\colon$$(V(\Lambda_{m_n}),d_{\Lambda_{m_n}})\to (V(\Gamma_{m_n}),d_{\Gamma_{m_n}})$ is $D$-Lipschitz.
 \item There exists $d\in \mathbb{N}_{\geq 2}$ such that for every $n$, each vertex of $\Lambda_{m_n}$ has degree at most $d$.
 \item The number $\sharp(V(\Lambda_{m_n}))$ tends to $\infty$ as $n\to \infty$. \item (\textit{Poincar\'{e}-type inequality}) For every $E\in \mathcal{E}$, there exists $C_E>0$ such that the following holds true: For every $n\in \mathbb{N}$ and for every map $f_{m_n}\colon V(\Lambda_{m_n})\to E$, it holds that
\[
\frac{1}{\sharp(V(\Lambda_{m_n}))}\sum_{v\in V(\Lambda_{m_n})}\|f_{m_n}(v)-m(f_{m_n})\|^q  \leq C_E  \left(\frac{1}{\sharp(V(\Lambda_{m_n}))}\sum_{e=(v,w)\in E(\Lambda_{m_n})} \|f_{m_n}(v)-f_{m_n}(w)\|^q\right), 
\]
where $m(f_{m_n})_n$ denotes the \textit{mean} of $f_{m_n}$: 
\[
m(f_{m_n})=\frac{1}{\sharp(V(\Lambda_{m_n}))}\left(\sum_{v\in V(\Lambda_{m_n})}f_{m_n}(v)\right)\quad (\in E).
\]
The sum on the right-hand side of the inequality above runs over all edges $e\in E(\Lambda_{m_n})$ in $\Lambda_{m_n}$, and for each $e\in E(\Lambda_{m_n})$, by writing $e=(v,w)$ we express that $e$ connects the vertices $v$ and $w$.
\end{itemize}
We say that $(\Gamma_m)_{m\in \mathbb{N}}$ is a family of \textit{Banach $(\mathcal{E},q)$-expanders} if we can take $m_n=m$ and $\Lambda_{m_n}=\Gamma_m$ (that also means that $\iota_{m}=\mathrm{id}_{V(\Gamma_m)}$) for every $n\in \mathbb{N}$.
\end{definition}
The concept of ordinary expanders is one with $(\mathcal{E},q)=(\mathcal{H}\mathrm{ilbert},2)$. It is known from work of Q. Cheng \cite{Cheng} that the condition of being  Banach $(\mathcal{E},q)$-expanders does not depend on the choice of the exponent $q\in [1,\infty)$. Also, the Poincar\'{e}-type inequality above naturally relates to (unweighted and non-normalized version of) Wang-type non-linear spectral gaps; see Example~\ref{example=r-UniformlyConvex}. Here we use the mean of $f$, instead of the $r$-barycenter, thanks to the linear structure of the Banach space $E$.

The following is a variant of the well-known fact asserting that expanders do not admit a coarse embedding into a Hilbert space. For the sake of completeness, we provide a proof; compare with the proof of \cite[Theorem~5.6.5]{BookNowakYu}.
\begin{proposition}\label{proposition=EmbeddedExpanders}
Let $\mathcal{E}$ be a non-empty class of Banach spaces and let $q\in [1,\infty)$. If a sequence of finite connected graphs $(\Gamma_m)_{m\in \mathbb{N}}$ of uniformly bounded degree admits embedded Banach $(\mathcal{E},q)$-expanders $(\Lambda_{m_n})_{n\in \mathbb{N}}$, then for the disjoint union $\bigsqcup_{m\in \mathbb{N}}(\Gamma_m,d_{\Gamma_m})$ does not admit a coarse embedding into $\mathcal{E}$.

In particular, if $(\Gamma_m)_{m\in \mathbb{N}}$ admits  embedded $($ordinary$)$ expanders, then $\bigsqcup_{m\in \mathbb{N}}(\Gamma_m,d_{\Gamma_m})$ does not admit a coarse embedding into a Hilbert space.
\end{proposition}

\begin{proof}
Suppose that there exists a coarse embedding $f\colon \bigsqcup_{m\in \mathbb{N}}(\Gamma_m,d_{\Gamma_m}) \to E$ with control pair $(\rho,\omega)$. Then for every $n\in \mathbb{N}$ and for every $v,w\in V(\Lambda_{m_n})$ adjacent in $\Lambda_{m_n}$, it holds that $\|f(\iota(v))-f(\iota(w))\|\leq \omega(D)$. By the Poincar\'{e}-type inequality in the conditions above, we therefore have that
\[
\frac{1}{\sharp(V(\Lambda_{m_n}))}\sum_{v\in V(\Lambda_{m_n})}\|f(\iota_{m_n}(v))-m((f\circ \iota_{m_n})|_{V(\Lambda_{m_n})})\|^q  \leq C_E d \cdot \omega(D)^q.
\]
Since the right-hand side of the inequality above is independent of $n$, the images $f(\iota_{m_n}(V(\Lambda_{m_n})))$ must be concentrated around its mean $m((f\circ \iota_{m_n})|_{V(\Lambda_{m_n})})$. It contradicts the properness of $\rho$ as $n\to \infty$, because $\iota_{m_n}$ is injective, $\sharp(V(\Lambda_{m_n}))\to \infty$, and $(\Gamma_m)_m$ is of uniformly bounded degree.
\end{proof}

The proof above works for a more general setting of graphs that admit \textit{weakly embedded expanders}; see \cite{ArzhantsevaTessera}.

The following is deduced from \cite[Theorem~A]{MimuraSphereEquivalence}; compare with \cite[Theorem~1.10]{Naor} and  \cite[Appendix~A]{Ozawa}.
\begin{proposition}\label{proposition=SphereEquivalence}
Let $F$ be a Banach space and let $q\in [1,\infty)$. Let $E$ be a Banach space that is $\mathrm{sphere}$ $\mathrm{equivalent}$ to $F$, namely, there exists a bijection $\Phi\colon S(F)\to S(E)$ between unit spheres such that $\Phi$ and $\Phi^{-1}$ are both uniformly continuous. If a sequence of finite connected graphs $(\Gamma_m)_{m\in \mathbb{N}}$ admits embedded Banach $(F,q)$-expanders, then it admits embedded Banach $(E,q)$-expanders.
\end{proposition}

Recall that \cite[Proposition~6.3]{IzekiNayatani} shows the following inequality for (Wang-type) non-linear sepctral gaps: For $M\in \mathcal{CAT}(0)$ and  for every weighted finite connected graph $\Gamma$,
\[
\lambda_1(\Gamma,M)\geq (1-\delta(M))\lambda_1(\Gamma).
\]
From this, the following may be showed in a similar manner to one in the proof of Proposition~\ref{proposition=EmbeddedExpanders}.

\begin{proposition}\label{proposition=NonLinearSpectralGap}
If a sequence of finite connected graphs of uniformly bounded degree $(\Gamma_m)_{m\in \mathbb{N}}$ admits embedded  expanders, then  $\bigsqcup_{m\in \mathbb{N}}\Gamma_m$ does not admit a coarse embedding into $\mathcal{CAT}(0)_{<1}$.
\end{proposition}

Mendel and Naor \cite{MendelNaor2015} constructed a complete $\mathrm{CAT}(0)$ space $M$ and a sequence of graphs $(\Gamma_m)_m$ such that $(\Gamma_m)_m$ forms an expander family with respect to $M$, but that expanders coming from random graphs are not expanders with respect to $M$. This $M$ must have the Izeki--Nayatani invariant $1$.

\subsection{Uniformity is \textit{not} automatic for a-$\mathcal{M}$-menability}
For a non-empty class of metric spaces, we say that a non-empty set $K\subseteq \mathcal{G}(k)$ is \textit{pointwise a-$\mathcal{M}$-menable} if every $\mathbf{G}\in K$ is a-$\mathcal{M}$-menable. Concerning amenability and property $(\mathrm{T})$, uniformity is \textit{automatic} for Cayley-compact subsets, namely, the pointwise property automatically implies the uniform one; see \cite[Proposition~3.4]{MimuraSako} and \cite[Proposition~5.1]{MOSSPartIII}. In contrast, concerning a-$\mathcal{M}$-menablity, uniformity is \textit{not} automatic, as the example below indicates.

\begin{example}\label{example=Selberg}
The classical ping-pong argument shows that $\mathbf{F_2}\cong (G_0;\left(\begin{array}{cc}1 & 2\\ 0&1\end{array}\right), \left(\begin{array}{cc}1 & 0\\ 2&1\end{array}\right))$ as marked groups, where $G_0(\simeq F_2)$ is the group generated by these two elements. For each \textit{odd} prime $p$, consider the $\mathrm{mod}$ $p$ reduction. Then $G_0$ maps \textit{onto} $\mathrm{SL}(2,\mathbb{F}_p)$ and 
\[
(\mathrm{SL}(2,\mathbb{F}_p); \left(\begin{array}{cc}1 & 2\\ 0&1\end{array}\right)\ \mathrm{mod}\ p, \left(\begin{array}{cc}1 & 0\\ 2&1\end{array}\right)\ \mathrm{mod}\ p) \quad \stackrel{\mathrm{Cay}}{\longrightarrow} \quad \mathbf{F_2},
\]
as $p\to \infty$. We write the marked group in the left-hand side as $\mathbf{G}_p$. Then $K=\{\mathbf{G}_p:p \textrm{ odd prime.}\}\cup \{\mathbf{F_2}\}$ is a compact subset in $\mathcal{G}(2)$. This set $K$ is pointwise a-$\mathrm{T}$-menable, but \textit{not} uniformly a-$\mathrm{T}$-menable. Indeed, for the latter assertion, by work of A. Selberg \cite{Selberg}, it follows that $(\mathrm{Cay}(\mathbf{G}_p))_{p}$ forms an expander family; see also \cite{BookLubotzky}. By Proposition~\ref{proposition=EmbeddedExpanders}, there does not exist a common pair $(\rho,\omega)$ that serves as a control pair of all of the $\mathbf{G}_p$, $p$ odd primes.
\end{example}
In this example, the obstruction to uniformity is the coarse non-embeddability, \textit{not} the equivariant one of the sequence. Hence, we are able to utilize this observation to prove Proposition~\ref{proposition=Products} in the following manner.

\begin{proof}[Proof of Proposition~$\ref{proposition=Products}$]
By the way of contradiction. Assume that $\bigsqcup_{m\in \mathbb{N}}\Gamma_m$ does not admit a coarse embedding into $\mathcal{M}$. Choose an element $\Lambda_{\infty}$ in the rooted graph boundary $\partial_{r}(\Lambda_n)_{n\in \mathbb{N}}$; recall the definition from Remark~\ref{remark=RootedGraphs}. Then, for every $m\in \mathbb{N}$, the subsequence $(\Gamma_m\times \Lambda_n)_{n\in \mathbb{N}}$, with changing roots, has $\Gamma_m\times \Lambda_{\infty}$  as an accumulation point in the space of rooted graphs. Hence by Proposition~\ref{proposition=RootedGraphUniformCoarseEmbedding}, in particular, $(\Gamma_m\times \Lambda_{\infty})_{m\in \mathbb{N}}$, must admit equi-coarse embeddings into $\mathcal{M}$. This contradicts coarse non-embeddability of $\bigsqcup_{m\in \mathbb{N}}\Gamma_m$ into $\mathcal{M}$.
\end{proof}

\subsection{Upper triangular products}\label{subsection=UpperTriangular}
We saw in the previous subsection that by taking the disjoint union $\bigsqcup_{m,n\in \mathbb{N}}(\Gamma_m\times \Lambda_n)$, \textit{we can embed a copy of each $\Gamma_m$ and $\Lambda_n$ $($as an isometrically embedded subgraph$)$ in the rooted graph boundary}. In what follows, we slightly modify this construction and call the resulting object the \textit{upper triangular product}. We exhibit it in the context of the space of marked groups. 

Let $(\mathbf{G}_m)_{m\in \mathbb{N}}\subseteq \mathcal{G}(k_1)$ and $(\mathbf{H}_n)_{n\in \mathbb{N}}\subseteq \mathcal{G}(k_2)$. For each $\mathbf{G}=(G;s_1,\ldots, s_{k_1})$ and $\mathbf{G}=(H;t_1,\ldots, t_{k_2})$, define the \textit{direct product marked group} $\mathbf{G}\times \mathbf{H}$ by
\[
\mathbf{G}\times \mathbf{H}=(G\times H;(s_1,e_H),\ldots ,(s_{k_1},e_H),(e_G,t_1),\ldots ,(e_G,t_{k_2}))(\in \mathcal{G}(k_1+k_2)).
\]
\begin{definition}
The \textit{upper triangular product} of $\bigsqcup_{m\in \mathbb{N}} \mathrm{Cay}(\mathbf{G}_m)$ and $\bigsqcup_{n\in \mathbb{N}} \mathrm{Cay}(\mathbf{H}_n)$ is defined by
\[
\bigsqcup_{(m,n)\in \mathbb{N}\times \mathbb{N},m\leq n} \mathrm{Cay}(\mathbf{G}_m\times \mathbf{H}_n)
\]
equipped with the total order given by comparison firstly on $n$ and secondly on $m$ on the index set $\{(m,n)\}$; namely, $(0,0)<(0,1)<(1,1)<(0,2)<(1,2)<(2,2)<(0,3)<\cdots$. If $(\mathbf{G}_m)_m$ and $(\mathbf{H}_n)_n$ are indexed by sets that are respectively order isomorphic to $(\mathbb{N},>)$, then we modify the order accordingly.

We write the sequence $(\mathbf{G}_m\times \mathbf{H}_n)_{(m,n)\in \mathbb{N}^2,\ m\leq n}$ in $\mathcal{G}(k_1+k_2)$, with the enumeration with respect to the order above, identified with that by $l\in \mathbb{N}$, as $(\mathbf{G}_m)_m\bigtriangledown(\mathbf{H}_n)_n$. 
\end{definition}

Note that for the upper triangular product,
\[
\partial_{\mathrm{Cay}}((\mathbf{G}_m)_m\bigtriangledown(\mathbf{H}_n)_n)= \left(\bigcup_{m\in \mathbb{N}}(\mathbf{G}_m\times \partial_{\mathrm{Cay}}(\mathbf{H}_n)_n)\right)\cup \left(\bigcup_{\mathbf{G}_{\infty}\in \partial_{\mathrm{Cay}}(\mathbf{G}_m)_m}(\mathbf{G}_{\infty}\times \partial_{\mathrm{Cay}}(\mathbf{H}_n)_n)\right)
\]
In this way, \textit{we can embed $($isomorphic and isometric copies of$)$ $(\mathbf{G}_m)_m$ in the Cayley boundary} (as subgroups of respectively suitable Cayley boundary groups).

\begin{proof}[Proof of Theorem~$\ref{theorem=Exotic}$]
Take the sequence of marked groups $(\mathbf{G}_p)_p$ over odd primes $p$ as in Example~\ref{example=Selberg}, and construct the upper triangular product $(\mathbf{H}_l)_{l\in \mathbb{N}}=(\mathbf{G}_p)_p \bigtriangledown(\mathbf{G}_p)_p$. Set $\Gamma_l$ as $\mathrm{Cay}(\mathbf{H}_l)$. Since $(\mathrm{Cay}(\mathbf{G}_p))_p$ forms an expander family, so does $(\Gamma_l)_{l\in \mathbb{N}}$. The Cayley boundary of that sequence contains an isometric copy of $(\mathrm{Cay}(\mathbf{G}_p))_p$; hence by Proposition~\ref{proposition=UniformCoarseEmbedding} together with Propositions~\ref{proposition=EmbeddedExpanders}, \ref{proposition=SphereEquivalence} and \ref{proposition=NonLinearSpectralGap}, we confirm the second assertion. To see the third assertion, G. A. Margulis showed that there exists $c>0$ such that for all odd prime $p$,
\[
\mathrm{girth}(\mathrm{Cay}(\mathbf{G}_p))\geq c \cdot\mathrm{diam}(\mathrm{Cay}(\mathbf{G}_p))
\]
holds, where the \textit{girth} of a connected graph is the length of shortest cycle; see \cite[Appendix~A]{BookDavidoffSarnakValette}. For such a sequence of finite graphs $(\mathrm{Cay}(\mathbf{G}_p))_p$, T. Kondo \cite{Kondo} constructed  a complete $\mathrm{CAT}(0)$ space $M_0=M_0((\mathrm{Cay}(\mathbf{G}_p))_p)$ such that the disjoint union $\bigsqcup_{p}\mathrm{Cay}(\mathbf{G}_p)$ embeds biLipschitzly into $M_0$. Therefore, the disjoint union of $(\Gamma_l)_l$ admits a biLipschitz embedding into $M=(M_0\times M_0)_{\ell_2}$.
\end{proof}

\subsection{Embedded expanders from fixed point property, and exotic examples from symmetric groups}\label{subsection=ProofOfExotic}
Here we prove Theorem~\ref{mtheorem=SymmetricGroups}. First we prove the following proposition, which may be of its own interest. It may be regarded as a generalization of \cite[Corollary~1.2]{MOSSPartIII} of our Part III paper.
\begin{proposition}\label{proposition=Embedded}
Let $(\mathbf{G}_m=(G_m;s_1^{(m)},\ldots ,s_k^{(m)}))_{m\in \mathbb{N}}$ be a Cayley convergent sequence consisting of finite marked groups and $\mathbf{G}=(G;s_1,\ldots ,s_k)$ be the limit. Let $\mathcal{E}$ be a non-empty class of Banach spaces that satisfies both of the following two conditions:
\begin{enumerate}[$(1)$]
  \item There exists $q\in [1,\infty)$ such that for every $E\in \mathcal{E}$, it holds that $\ell_q(\mathbb{N},E)\in \mathcal{E}$.
  \item The class $\mathcal{E}$ can be written as a union of subclasses 
\[
\mathcal{E}=\bigcup_{\lambda}\mathcal{E}_{\lambda}
\]
such that each such subclass $\mathcal{E}_{\lambda}$ satisfies the following: For every $(E_m)_{m\in \mathbb{N}}$ with $E_m\in \mathcal{E}_{\lambda}$ for every $m$, there exists a non-principal ultrafilter $\mathcal{U}$ over $\mathbb{N}$ such that $\lim_{\mathcal{U}}(E_m,0) \in \mathcal{E}_{\lambda}$.
\end{enumerate}
Assume that $G$ contains an infinite subgroup $H$ with property $(\mathrm{F}_{\mathcal{E}})$. Then the sequence of Cayley graphs $(\mathrm{Cay}(\mathbf{G}_m))_{m\in \mathbb{N}}$ admits embedded Banach $(\mathcal{E},q)$-expanders.
\end{proposition}
By combining this with Proposition~\ref{proposition=EmbeddedExpanders}, we deduce that the disjoint union $\bigsqcup_{m\in \mathbb{N}}\mathrm{Cay}(\mathbf{G}_m))$ of such a sequence does not admit a coarse embedding into $\mathcal{E}$.

To prove Proposition~\ref{proposition=Embedded}, we employ the following three results. 

\begin{lemma}\label{lemma=Guichardet}
Assume that a non-empty class of Banach spaces $\mathcal{E}$ satisfies condition $(1)$ as in Proposition~$\ref{proposition=Embedded}$. Then, if a countable discrete group $H$ satisfies property $(\mathrm{F}_{\mathcal{E}})$, then $H$ is finitely generated.
\end{lemma}

\begin{proof}
Generalize the proof of \cite[Proposition~2.4.1 and Corollary~2.4.2]{BookBekkadelaHarpeValette}.
\end{proof}

\begin{proposition}\label{proposition=Tau}
Assume that a non-empty class of Banach spaces $\mathcal{E}$ satisfies condition $(1)$ as in Proposition~$\ref{proposition=Embedded}$. Let $\mathbf{H}=(H;T)$ be an infinite marked group such that $H$ has property $(\mathrm{F}_{\mathcal{E}})$. Let $(\mathbf{H}_{n}=(H_n;T_n))_{n\in \mathbb{N}}$ be a sequence of finite marked group quotients $($recall the definition from Definition~$\ref{definition=RFLEFLEA}$$)$ such that $\lim_{n\to \infty}\sharp(H_n)=\infty$. 

Then, the sequence $(\mathrm{Cay}(\mathbf{H}_n))_{n\in\mathbb{N}}$ forms a family of Banach $(\mathcal{E},q)$-expanders.
\end{proposition}

\begin{proof}
By \cite[3.a]{BaderFurmanGelanderMonod}, $H$ has property $(\mathrm{T}_{\mathcal{E}})$ in the sense of Bader--Furman--Gelander-Monod. For each $E\in \mathcal{E}$. in particular, $H$ has property $(\mathrm{T}_{\ell_q(\mathbb{N},E)})$. This implies that \textit{the $(\tau)$-type constant} associated with $(\mathbf{H},\ell_q(\mathbb{N},E))$, defined in our Part I paper \cite[Definition~6.6.$(2)$]{MimuraSako}, is strictly positive. Then, in a similar argument to one in the proof of \cite[Lemma~6.8]{MimuraSako} (by replacing the square sums there with $q$-sums), we deduce that $(\mathrm{Cay}(\mathbf{H}_n))_{n\in\mathbb{N}}$ satisfies the Poincar\'{e}-type inequality as in Definition~\ref{definition=EmbeddedExpanders}. By construction, degrees are bounded by $2k$, and $(\infty>)\sharp(H_n)\to \infty$.
\end{proof}

\begin{proposition}\label{proposition=GromovSchoen}
Assume that a non-empty class of Banach spaces $\mathcal{E}$ satisfies condition $(2)$ as in Proposition~$\ref{proposition=Embedded}$ with $\mathcal{E}_{\lambda}=\mathcal{E}$. Let $\mathbf{H}=(H;T)$ be an infinite marked group such that $H$ has property $(\mathrm{F}_{\mathcal{E}})$. Then there exists a $\mathrm{finitely}$ $\mathrm{presented}$ marked group $\tilde{\mathbf{H}}$ such that it has property $(\mathrm{F}_{\mathcal{E}})$ and there exists a marked quotient map $\tilde{\mathbf{H}}\twoheadrightarrow \mathbf{H}$. 
\end{proposition}
\begin{proof}
This follows from a well-known \textit{Gromov--Schoen argument}; see the survey \cite{Stalder} of  Stalder. More precisely, \cite[Theorem~1.5]{Stalder} implies that the subset of all marked groups in $\mathcal{G}(k)$ with property $(\mathrm{F}_{\mathcal{E}})$ forms an \textit{open} subset in the Cayley topology. Here $k=\sharp(T)$. If $\mathbf{H}$ itself is finitely presented, then we are done. Otherwise, there exists a Cayley convergent sequence $(\tilde{\mathbf{H}}_m)_{m\in\mathbb{N}}$ to $\mathbf{H}$
\[
\tilde{\mathbf{H}}_0=\mathbf{F_k}\twoheadrightarrow \tilde{\mathbf{H}}_1 \twoheadrightarrow \tilde{\mathbf{H}}_2 \twoheadrightarrow \cdots  \twoheadrightarrow \tilde{\mathbf{H}}_m \twoheadrightarrow \cdots \quad \stackrel{\mathrm{Cay}}{\longrightarrow}\quad  \mathbf{H}
\]
consisting of \textit{finitely presented} marked groups, constructed by putting relations of $\mathbf{H}$ one by one. By the openness property above, there must exist $m\in\mathbb{N}$ such that $\tilde{\mathbf{H}}_m$ has property $(\mathrm{F}_{\mathcal{E}})$. This $\tilde{\mathbf{H}}_m$ is a desired $\tilde{\mathbf{H}}$.
\end{proof}
On Proposition~\ref{proposition=GromovSchoen}, the case where $\mathcal{E}=\mathcal{H}\mathrm{ilbert}$ was proved by Shalom \cite{Shalom}; see also \cite{KorevaarSchoen}. In this case, property $(\mathrm{F}_{\mathcal{H}\mathrm{ilbert}})$ (for countable discrete groups) is equivalent to the celebrated \textit{property $(\mathrm{T})$} of D. Kazhdan; see \cite{BookBekkadelaHarpeValette}  on property $(\mathrm{T})$, including this equivalence (the Delorme--Guichardet theorem).

\begin{proof}[Proof of Proposition~$\ref{proposition=Embedded}$]
By Lemma~\ref{lemma=Guichardet}, $H$ is finitely generated. Fix a finite generating set $T=(t_1,\ldots ,t_l)$ of $H$. Then, each $t_j$, $j\in[l]$, may be written as a product of elements in $S=(s_1,\ldots ,s_k)$; fix such an expressions for each $j\in [l]$. For each $m\in \mathbb{N}$, $t_j^{(m)}$, $j\in [l]$, be the element in $G_m$ constructed by replacing $s_i$ with $s_i^{(m)}$ in that expression for all $i\in [k]$. Let $H_m(\leqslant G_m)$ be the group generated by these $t_1^{(m)},\ldots ,t_k^{(m)}$. Then for every $m\in \mathbb{N}$, $H_m$ is finite, and 
\[
(H_m;t_1^{(m)},\ldots ,t_k^{(m)}) \quad \stackrel{\mathrm{Cay}}{\longrightarrow} \quad \mathbf{H}.
\]
Now fix $E\in \mathcal{E}$. Then by condition $(1)$ and $(2)$, there exists a subclass  $\mathcal{E}_{\lambda}$ as in $(2)$ of $\mathcal{E}$ that contains $\ell_q(\mathbb{N},E)$. We apply Proposition~\ref{proposition=GromovSchoen} to $\mathcal{E}_{\lambda}$ and take \textit{finitely presented} marked lift $\tilde{\mathbf{H}}$ of $\mathbf{H}$ with property $(\mathrm{F}_{\mathcal{E}_{\lambda}})$. Then by finite presentation of $\tilde{\mathbf{H}}$, the set of all marked group quotients of $\tilde{\mathbf{H}}$ is an \textit{open} neighborhood of $\mathbf{H}$; recall Remark~\ref{remark=FinitelyPresented}. In particular, the sequence $((H_m;t_1^{(m)},\ldots ,t_k^{(m)}))_{m}$ eventually consists of marked group quotient of $\tilde{\mathbf{H}}$. Therefore, Proposition~\ref{proposition=Tau} applies and $(\Lambda_m)_m=(\mathrm{Cay}(H_m;t_1^{(m)},\ldots ,t_k^{(m)}))_{m}$ forms a Banach $(E,q)$-expander family. (Strictly speaking, for small $m$, the marked group might not be a marked group quotients of $\tilde{\mathbf{H}}$. However, since these are only finitely many, they do not affect the Banach $(E,q)$-expander property.) Because this holds for each $E\in \mathcal{E}$, $(\Lambda_m)_m$ forms a Banach $(\mathcal{E},q)$-expander family.

Finally we go back to the original graphs $(\Gamma_m)_{m\in\mathbb{N}}=(\mathrm{Cay}(\mathbf{G}_m))_{m\in \mathbb{N}}$. First, the vertex set $V(\Lambda_m)=H_m$ injects into $V(\Gamma_m)=G_m$ via $\iota_m\colon H_m \hookrightarrow G_m$ (as a subgroup $H_m\leqslant G_m$). Moreover, by construction of $(t_1^{(m)},\ldots ,t_l^{(m)})$, there exists $D>0$ such that for every $m\in \mathbb{N}$, the map $(\Lambda_m,d_{\Lambda_m})\to (\Gamma_m,d_{\Gamma_m})$ induced by $\iota_m$ is $D$-Lipschitz. This ends our proof.
\end{proof}

Before proceeding to the proof of Theorem~\ref{mtheorem=SymmetricGroups}, we state the following lemma, which enables us to \textit{encode} information of a Cayley convergence into symmetric groups. Here, for a non-empty set $B$, denote by $\mathrm{Sym}(B)$ the full symmetric group, and by $\mathrm{Sym}_{<\aleph_0}(B)$ the symmetric group with finite support, namely, the group of all permutations on $B$ that fix all but finitely many elements in $B$. For $l\in \mathbb{N}_{\geq 1}$, we abbereviate $\mathrm{Sym}([l])$ as $\mathrm{Sym}(l)$.

\begin{lemma}[Encoding into symmetric groups]\label{lemma=SymmetricGroups}
Let $k\in \mathbb{N}_{\geq 1}$. Let $(\mathbf{G}_{m})_{m\in \mathbb{N}}=(G_m;s_1^{(m)},\ldots ,s_k^{(m)}))_{m}$ be a LEF approximation of an infinite group $\mathbf{G}_{\infty}=(G_{\infty};s_1^{(\infty)},\ldots ,s_k^{(\infty)})$. Assume that for every $m\in \mathbb{N}\cup\{\infty\}$ and for every $j\in [k]$, it holds that $s_j^{(m)}\ne e_{G_m}$. 

Then we have the following Cayley convergence in $\mathcal{G}(2k)$:
\begin{eqnarray*}
& &(\mathrm{Sym}(G_m);\chi_{s_1^{(m)}}, \ldots ,\chi_{s_k^{(m)}},\theta_{s_1^{(m)}},\ldots ,\theta_{s_k^{(m)}}) \\
\quad \stackrel{\mathrm{Cay}}{\longrightarrow}\quad & &(\mathrm{Sym}_{<\aleph_0}(G_\infty)\rtimes G_{\infty};\chi_{s_1^{(\infty)}}, \ldots ,\chi_{s_k^{(\infty)}},\theta_{s_1^{(\infty)}},\ldots ,\theta_{s_k^{(\infty)}}).
\end{eqnarray*}

Here, $G_{\infty}$ acts on $\mathrm{Sym}_{<\aleph_0}(G_\infty)$ as permutations induced by right multiplication; for a countable group $G$ and for $\gamma \in G\setminus\{e_G\}$, we define  elements $\chi_{\gamma}\in \mathrm{Sym}_{<\aleph_0}(G)$ and $\theta_{\gamma}\in \mathrm{Sym}(G)$ by
\begin{eqnarray*}
\chi_{\gamma}&=&(\textrm{the transposition on $\{e_{G},\gamma\}$}), \\
\theta_{\gamma}&=&(\textrm{the permutation on $G$ given by the right-multiplication of $\gamma$}).
\end{eqnarray*}
\end{lemma}

For the proof, see \cite[the proof of Lemma~4.9]{MimuraNonExact}.

\begin{proof}[Proof of Theorem~$\ref{mtheorem=SymmetricGroups}$]
We take two sequences of marked groups $((G_m;S_m))_{m\in \mathbb{N}_{\mathrm{odd}}}$ and $((G_m;T_m))_{m\in \mathbb{N}_{\mathrm{odd}}}$ as in Example~\ref{example=SpecialLinear}. More precisely, $G_m=\mathrm{SL}(m,\mathbb{F}_{p^{n_m}})$, $S_m=(\sigma^{(m)},\upsilon^{(m)},\tau^{(m)})$ and $T_m=(\sigma^{(m)},\sigma'^{(m)},\upsilon^{(m)},\tau^{(m)})$. Let $(\mathbf{H}_n)_{n\in \mathbb{N}_{\geq 3}}=((\mathbb{Z}/n\mathbb{Z};1))_n$. Then, take upper triangular products
\begin{eqnarray*}
(\mathbf{I}_l)_{l\in \mathbb{N}}&=& ((G_m;S_m))_m \bigtriangledown (\mathbf{H}_n)_n \quad \textrm{in}\ \mathcal{G}(4),\\
(\mathbf{J}_l)_{l\in \mathbb{N}}&=& ((G_m;T_m))_m \bigtriangledown (\mathbf{H}_n)_n \quad \textrm{in}\ \mathcal{G}(5).
\end{eqnarray*}
By construction, concerning Cayley boundaries, we have that
\begin{eqnarray*}
\partial_{\mathrm{Cay}}(\mathbf{I}_l)_l&=& \{(G_m;\sigma^{(m)},\upsilon^{(m)},\tau^{(m)}):m\in \mathbb{N}_{\mathrm{odd}}\} \times \mathbf{Z}\\
  & &\cup\ \{(\mathrm{N}_{>}(\mathbb{Z},\mathbb{F}_p[t])\rtimes \mathbb{Z};\sigma^{(\infty)},\upsilon^{(\infty)},\tau^{(\infty)})\times \mathbf{Z}\},\\
\partial_{\mathrm{Cay}}(\mathbf{J}_l)_l&=& \{(G_m;\sigma^{(m)},\sigma'^{(m)},\upsilon^{(m)},\tau^{(m)}):m\in \mathbb{N}_{\mathrm{odd}}\} \times \mathbf{Z}\\
& &\cup\ \{(\mathrm{SL}(\mathbb{Z},\mathbb{F}_p[t])\rtimes \mathbb{Z};\sigma^{(\infty)},\sigma'^{(\infty)},\upsilon^{(\infty)},\tau^{(\infty)})\times \mathbf{Z}\},
\end{eqnarray*}
for some markings $(\sigma^{(\infty)},\upsilon^{(m)},\tau^{(\infty)})$ and $(\sigma^{(\infty)},\sigma'^{(\infty)},\upsilon^{(m)},\tau^{(\infty)})$. Here $\mathbf{Z}=(\mathbb{Z};1)$.

Note that for each $l\in \mathbb{N}$, the underlying groups of $\mathbf{I}_l$ and $\mathbf{J}_l$ are the same; we write it as $K_l$. The marking of $\mathbf{I}_l$ is of the form $(b_{1}^{(l)},b_2^{(l)},b_3^{(l)},c^{(l)})$ and the one of $\mathbf{J}_l$ is of the form $(b_1^{(l)},{b'_{1}}^{(l)},b_2^{(l)},b_3^{(l)},c^{(l)})$. Here $b_1$, $b'_{1}$, $b_2$, $b_3$ are associated, respectively, with $\sigma$, $\sigma'$, $\upsilon$, $\tau$, and $c$ corresponds to the generator $1$ of $\mathbf{H}_n$.

Finally, we employ the encoding into symmetric groups as in Lemma~\ref{lemma=SymmetricGroups}. More precisely, consider two systems of markings $(\Xi_l)_{l\in \mathbb{N}}$ and $(\Omega)_{l\in \mathbb{N}}$ of $(\mathrm{Sym}(K_l))_{l\in \mathbb{N}}$ by 
\begin{eqnarray*}
    \Xi_l&=&(\chi_{{b_1}^{(l)}},\chi_{{b_2}^{(l)}},\chi_{{b_3}^{(l)}},\chi_{{c}^{(l)}},\theta_{{b_1}^{(l)}},\theta_{{b_2}^{(l)}},\theta_{{b_3}^{(l)}},\theta_{{c}^{(l)}}),\\
 \Omega_l&=&(\chi_{{b_1}^{(l)}},\chi_{{b_2}^{(l)}},\chi_{{b_3}^{(l)}},\chi_{{c}^{(l)}},\theta_{{b_1}^{(l)}},\theta_{{b_2}^{(l)}},\theta_{{b_3}^{(l)}},\theta_{{c}^{(l)}},\theta_{{b'_1}^{(l)}}).
\end{eqnarray*}

In what follows, we will verify the assertions as in Theorem~\ref{mtheorem=SymmetricGroups}. Item $(1)$ is by construction. To see $(2)$, all underlying groups appearing in $\partial_{\mathrm{Cay}}(\mathbf{I}_l)_l$ are
\begin{align*}
\mathrm{Sym}_{<\aleph_0}(G_m\times \mathbb{Z})\rtimes (G_m\times \mathbb{Z}),\ m\in \mathbb{N}_{\mathrm{odd}},\quad \textrm{and}\quad \mathrm{Sym}_{<\aleph_0}(\tilde{G}_{\infty})\rtimes (\tilde{G}_{\infty}),
\end{align*}
where $\tilde{G}_{\infty}=(\mathrm{N}_{>}(\mathbb{Z},\mathbb{F}_p[t])\rtimes \mathbb{Z})\times \mathbb{Z}$. 
Since all of them are amenable, \cite[Theorem~A]{MimuraSako} implies that the disjoint union $\bigsqcup_{l}\mathrm{Cay}(\mathbf{I}_l)$ has property A. 

Finally, we deal with $(3)$. In a similar argument to one above, we  see that  the Cayley boundary $\partial_{\mathrm{Cay}}(\mathbf{J}_l)_l$ contains an isomorphic and isometric copies of $((G_m;T_m))_{m\in \mathbb{N}_{\mathrm{odd}}}$ (as subgroups of respectively suitable Cayley boundary groups). Now recall that 
\[
(G_m;T_m) \quad \stackrel{\mathrm{Cay}}{\longrightarrow} \quad \mathrm{SL}(\mathbb{Z},\mathbb{F}_p[t])\rtimes \mathbb{Z}
\]
with respect to a suitable marking of the limit, and that the Cayley limit group contains $\mathrm{SL}(3,\mathbb{F}_p[t])$, which has property $(\mathrm{F}_{\mathcal{B}_{\mathrm{type}>1}})$. Note that the class $\mathcal{B}_{\mathrm{type}>1}$ fulfills the two conditions in Proposition~\ref{proposition=Embedded}. Indeed, to see $(2)$, decompose as
\[
\mathcal{B}_{\mathrm{type}>1}=\bigcup_{r\in (1,2],\ C>0}\mathcal{B}_{r,C}^{\mathrm{type}}.
\]
Hence by Proposition~\ref{proposition=Embedded}, we conclude that $(\mathrm{Cay}(G_m;T_m))_{m\in \mathbb{N}_{\mathrm{odd}}}$ admits embedded Banach $(\mathcal{B}_{\mathrm{type}>1},2)$-expanders. This with Propositions~\ref{proposition=EmbeddedExpanders}, \ref{proposition=SphereEquivalence} and \ref{proposition=NonLinearSpectralGap} imply that $\partial_{\mathrm{Cay}}(\mathbf{J}_l)_l$ does \textit{not} admit equi-coarse embeddings into $\mathcal{M}$, where $\mathcal{M}$ is either of the two classes  as in the assertion of $(3)$. Thus by Proposition~\ref{proposition=UniformCoarseEmbedding} we complete the proof. Here for every $l\in \mathbb{N}$, we set $k_l=\sharp(K_l)$ and identify $\mathrm{Sym}(k_l)$ with $\mathrm{Sym}(K_l)$.
\end{proof}
\begin{remark}
In this specific example above, we do not need to appeal to Proposition~\ref{proposition=GromovSchoen} to obtain a finitely presented lift with property $(\mathrm{F}_{\mathcal{B}_{\mathrm{type}>1}})$. Indeed, it follows from work of H. Behr \cite{Behr} that $\mathrm{SL}(n,\mathbb{F}_{p^r}[t])$ is finitely presented for every prime $p$ and for every $r\in \mathbb{N}_{\geq 1}$, \textit{provided that $n\geq 4$}. Thus the Cayley limit group $\mathrm{SL}(\mathbb{Z},\mathbb{F}_p[t])\rtimes \mathbb{Z}$ of our concern in the example above contains a copy of a \textit{finitely presented} group $\mathrm{SL}(4,\mathbb{F}_p[t])$ with property $(\mathrm{F}_{\mathcal{B}_{\mathrm{type}>1}})$ as a subgroup.
\end{remark}

We make a final remark, which is similar to  one in the Part I paper \cite{MimuraSako}: The construction above is ``\textit{semi}-explicit'' because in general, there is an issue to have an explicit generator of $\mathbb{F}_{p^{n_m}}^{\times}$. To obtain a fully explicit construction, replace coefficient rings $(\mathbb{F}_{p^{n_m}})_m$ with explicit other quotient rings of $\mathbb{F}_p[t]$; for instance take $(\mathbb{F}_p[t]/(t^{n_m}-t))_m$, and replace $(t_{n_m}\in \mathbb{F}_{p^{n_m}})_m$ with $(t\in \mathbb{F}_p[t]/(t^{n_m}-t))_m$.

\section*{Acknowledgments}
The authors wish to express their gratitude to Professor Guoliang Yu and Professor Qin Wang for their kind invitation to Fudan University in Shanghai in July, 2013. Part of this work was done during that stay. Some other part of this work was done during the  two-year stay of the first-named author in the \'{E}cole Polytechnique F\'{e}d\'{e}rale de Lausanne supported by Grant-in-Aid for JSPS Oversea Research Fellowships. The first-named author wishes to express his gratitude to Professor Nicolas Monod and Mrs. Marcia Gouffon for their hospitality and help on his visit. The two authors thank Damian Osajda for discussion on his construction of RF groups without property A, specially for Remark~\ref{remark=Osajda}. They are grateful to Sylvain Arnt for discussion and the terminology of a-$\mathcal{M}$-menability, Goulnara Arzhantseva and Florent Baudier for discussion and comments, Tom Kaiser for the reference \cite{Kaiser}, and Hiroyasu Izeki, Shin Nayatani and Tetsu Toyoda for discussions on $r$-uniformly convex metric spaces and the Izeki--Nayatani invariants. The first-named author thanks Yash Lodha for several comments, including advice to build Section~\ref{section=Gadgets} as a distinct section, and his suggestion of the terminology of ``fragmentary actions'', and Manor Mendel for discussion on several concepts on metric geometry of  non-linear spaces.

\bibliographystyle{amsalpha}
\bibliography{mimura_sako_cayley2.bib}

\end{document}